\documentclass[a4paper,12pt]{amsart}
\usepackage{amssymb,eucal}
\usepackage[cmtip,all]{xy}
\usepackage{tikz-cd}
\usepackage{hyperref}

\calclayout

\newcommand{\+}{\nobreakdash-}
\renewcommand{\:}{\colon}

\newcommand{\rarrow}{\longrightarrow}
\newcommand{\larrow}{\longleftarrow}
\newcommand{\ot}{\otimes}
\newcommand{\st}{\star}

\newcommand{\bu}{{\text{\smaller\smaller$\scriptstyle\bullet$}}}

\newcommand{\lrarrow}{\mskip.5\thinmuskip\relbar\joinrel\relbar\joinrel
 \rightarrow\mskip.5\thinmuskip\relax}

\DeclareMathOperator{\Hom}{Hom}
\DeclareMathOperator{\Ext}{Ext}
\DeclareMathOperator{\Tor}{Tor}
\DeclareMathOperator{\Spec}{Spec}
\DeclareMathOperator{\soc}{soc}
\DeclareMathOperator{\ppd}{pd}
\DeclareMathOperator{\iid}{id}
\DeclareMathOperator{\cfd}{cfd}
\DeclareMathOperator{\fd}{fd}

\newcommand{\sop}{\mathsf{op}}
\newcommand{\id}{\mathrm{id}}

\newcommand{\fR}{\mathfrak R}
\newcommand{\fS}{\mathfrak S}

\newcommand{\bs}{\mathbf s}
\newcommand{\br}{\mathbf r}
\newcommand{\bt}{\mathbf t}

\newcommand{\qs}{\mathsf{qs}}

\newcommand{\sA}{\mathsf A}
\newcommand{\sB}{\mathsf B}
\newcommand{\sC}{\mathsf C}
\newcommand{\sD}{\mathsf D}
\newcommand{\sE}{\mathsf E}
\newcommand{\sF}{\mathsf F}
\newcommand{\sG}{\mathsf G}
\newcommand{\sH}{\mathsf H}
\newcommand{\sJ}{\mathsf J}
\newcommand{\sK}{\mathsf K}
\newcommand{\sP}{\mathsf P}
\newcommand{\sS}{\mathsf S}
\newcommand{\sT}{\mathsf T}
\newcommand{\sX}{\mathsf X}
\newcommand{\sY}{\mathsf Y}

\newcommand{\bb}{\mathsf b}
\newcommand{\abs}{\mathsf{abs}}
\newcommand{\co}{\mathsf{co}}
\newcommand{\ctr}{\mathsf{ctr}}
\newcommand{\bco}{\mathsf{bco}}
\newcommand{\bctr}{\mathsf{bctr}}
\newcommand{\inj}{\mathsf{inj}}
\newcommand{\proj}{\mathsf{proj}}
\newcommand{\ctrfl}{\mathsf{ctrfl}}
\newcommand{\fpinj}{\mathsf{fpinj}}
\newcommand{\sico}{\mathsf{sico}}
\newcommand{\sictr}{\mathsf{sictr}}
\newcommand{\bsico}{\mathsf{bsico}}
\newcommand{\bsictr}{\mathsf{bsictr}}

\newcommand{\boL}{\mathbb L}
\newcommand{\boR}{\mathbb R}
\newcommand{\boZ}{\mathbb Z}
\newcommand{\boQ}{\mathbb Q}

\newcommand{\Ac}{\mathsf{Ac}}

\newcommand{\Modl}{{\operatorname{\mathsf{--Mod}}}}

\newcommand{\tors}{{\operatorname{\mathsf{-tors}}}}
\newcommand{\ctra}{{\operatorname{\mathsf{-ctra}}}}
\newcommand{\psco}{{\operatorname{\mathsf{-psco}}}}
\newcommand{\psctr}{{\operatorname{\mathsf{-psctr}}}}
\newcommand{\dctrfl}{{\operatorname{\mathsf{-ctrfl}}}}
\newcommand{\dfpinj}{{\operatorname{\mathsf{-fpinj}}}}

\newcommand{\Section}[1]{\bigskip\section{#1}\medskip}
\theoremstyle{plain}
\newtheorem{thm}{Theorem}[section]
\newtheorem{prop}[thm]{Proposition}
\newtheorem{lem}[thm]{Lemma}
\newtheorem{cor}[thm]{Corollary}
\theoremstyle{definition}
\newtheorem{rem}[thm]{Remark}
\newtheorem{rems}[thm]{Remarks}

\newtheorem{ex}[thm]{Example}
\newtheorem{exs}[thm]{Examples}

\begin{document}

\title{Pseudo-dualizing complexes of torsion modules \\
and semi-infinite MGM duality}

\author{Leonid Positselski}

\address{Institute of Mathematics, Czech Academy of Sciences \\
\v Zitn\'a~25, 115~67 Prague~1 \\ Czech Republic} 

\email{positselski@math.cas.cz}

\begin{abstract}
 This paper is an MGM version of~\cite{Pps,Ppc} and a follow-up
to~\cite[Section~5]{Pmgm}.
 In the setting of a commutative ring $S$ with a weakly proregular
finitely generated ideal $J\subset S$, we consider the maximal, abstract,
and minimal corresponding classes of $J$\+torsion $S$\+modules and
$J$\+contramodule $S$\+modules with respect to a given pseudo-dualizing
complex of $J$\+torsion $S$\+modules $L^\bu$, and construct the related
triangulated equivalences.
 As a special case, we obtain an equivalence of the semiderived
categories for an $I$\+adically coherent commutative ring $R$ with
a weakly proregular ideal $I\subset R$, a dualizing complex of
$I$\+torsion $R$\+modules $D^\bu$, and a ring homomorphism
$f\:R\rarrow S$ such that $f(I)\subset J$ and $S$ is a flat $R$\+module.
 (If the ring $S$ is not Noetherian, then a certain further assumption,
which we call quotflatness of the morphism of pairs
$f\:(R,I)\rarrow(S,J)$, needs to be imposed.)
 In that case, the pseudo-dualizing complex $L^\bu$ is constructed as
a complex of $J$\+torsion $S$\+modules quasi-isomorphic to
the tensor product of $D^\bu$ with the infinite dual Koszul complex
for some set of generators of the ideal $J\subset S$.
\end{abstract}

\maketitle

\tableofcontents

\section*{Introduction}
\medskip

\setcounter{subsection}{-1}
\subsection{{}}
 The classical topic of \emph{MGM} (\emph{Matlis--Greenlees--May})
\emph{duality} (known also as \emph{MGM equivalence}) in commutative
algebra goes back to the papers~\cite{Mat,GM,DG}.
 The importance of the \emph{weak proregularity} condition
(generalizing the more restrictive Noetherianity assumption)
was established in~\cite{PSY}.
 The contemporary formulation involving the derived categories of
the abelian categories of torsion modules and contramodules was
given in the paper~\cite[Corollary~3.5 and Theorem~5.10]{Pmgm}.

 The construction of the \emph{semi-infinite homology} of certain
infinite-dimensional Lie algebras was introduced in~\cite{Feig}
(the contemporary formulation was given in~\cite[Section~3.8]{BD}).
 The present paper presumes the philosophy of semi-infinite
homological algebra as elaborated in the books~\cite{Psemi,Psemten},
emphasizing the constructions of \emph{semiderived categories}.
 The classical anti-equivalence of the categories of Verma modules
over the Virasoro Lie algebra on complementary levels~$c$
and $26-c$ \,\cite{FF0}, \cite[Remark~2.4]{FF}, \cite{RCW} was
interpreted as a triangulated equivalence of the semiderived
categories in~\cite[Corollary and Remark~D.3.1]{Psemi}.

 The latter result became the thematic example of what was called
the \emph{semimodule-semicontramodule correspondence} in~\cite{Psemi}.
 A perhaps more accessible exposition in the more familiar context
of modules over rings (rather than semimodules over semialgebras
over coalgebras or corings) can be found in
the paper~\cite[Sections~5\+-6]{Pfp}.
 See~\cite[Section~3.5]{Prev} for a discussion of the results
of~\cite{Psemi} in a survey paper.
 A further instance of the semico-semicontra correspondence in
the context of nonaffine schemes was worked out in
the preprint~\cite[Section~8]{Pcosh}.

\subsection{{}}
 The aim of the present paper is to construct a triangulated
equivalence of the semico-semicontra correspondence in the MGM context,
i.~e., for torsion modules and contramodules over a commutative ring
with a finitely generated ideal.
 In fact, following the relative nature of semi-infinite settings, we
consider a \emph{morphism of ring-ideal pairs} $f\:(R,I)\rarrow(S,J)$.
 So $I\subset R$ and $J\subset S$ are finitely generated ideals in
commutative rings, and $f\:R\rarrow S$ is a ring homomorphism such that
$f(I)\subset J$.

 Most of our results require the ideals $I$ and/or $J$ to be weakly
proregular.
 We also assume that $S$ is a flat $R$\+module, and utilize a further
\emph{quotflatness} assumption on the morphism~$f$, meaning that
suitable quotient rings of $S$ by ideals related to $J$ are flat over
respective quotient rings of $R$ by ideals related to~$I$.

 We also assume the ring $R$ to be \emph{$I$\+adically coherent},
which means that the quotient rings $R/I^n$ are coherent for all
integers $n\ge1$.
 A more restrictive condition is the \emph{$I$\+adic Noetherianity},
meaning that the ring $R/I$ is Noetherian (if this is the case, then
all the rings $R/I^n$ are Noetherian as well).
 Notice that all ideals in a Noetherian commutative ring are weakly
proregular, but the $I$\+adic Noetherianity of $R$ does \emph{not}
imply the weak proregularity of~$I$.

\subsection{{}}
 This paper is the third one in a series of the present author's
papers on \emph{pseudo-dualizing complexes} and \emph{pseudo-derived
equivalences}.
 In fact, the topic of pseudo-derived equivalences was originated
in the paper~\cite{PS2}.
 The first two papers in the series were~\cite{Pps} (on pseudo-dualizing
complexes of bimodules over associative rings) and~\cite{Ppc} (on
pseudo-dualizing complexes of bicomodules over coalgebras).
 Now we treat the pseudo-dualizing complexes of torsion modules over
a commutative rings with a weakly proregular finitely generated ideal.

 What we call ``pseudo-dualizing complexes'' may be known to a larger
circle of people as ``semi-dualizing complexes'', which is a previously
existing term~\cite{Chr,HW,EH}.
 As the prefix ``semi'' is used systematically in our
context~\cite{Psemi,Prev,Pfp,Psemten,Pcosh} with a closely related,
but quite different meaning (alluding to ``semi-infinite''), we chose
to rename semi-dualizing complexes into pseudo-dualizing complexes in
the papers~\cite{Pps,Ppc}, and we continue to use
the ``pseudo-dualizing complexes'' terminology in the present paper.
 In fact, if the term ``semidualizing complexes'' were not taken,
we would be eager to use it for what are called
``relative dualizing complexes'' in~\cite[Section~9]{Pps} and
in Section~\ref{semico-semicontra-secn} of the present paper.

 The definition of a pseudo-dualizing complex is obtained from that
of a dualizing complex by dropping the finite injective dimension
condition while retaining the finite generatedness/finite presentability
and the homothety isomorphism conditions.
 Another closely related concept is that of a \emph{dedualizing
complex}, introduced in~\cite{Pmgm} and discussed further
in~\cite[Section~6]{Pps} and~\cite[Section~8]{Ppc}.
 The definition of a dedualizing complex is obtained from that of
a dualizing complex by replacing the finite injective dimension
condition with a finite projective dimension condition.
 So both the dualizing complexes and the dedualizing complexes are
particular cases (and in some sense, two polar special cases) of
the pseudo-dualizing complexes.

\subsection{{}} \label{introd-morita-subsecn}
 In the context of (pseudo-dualizing) complexes of bimodules over
a pair of associative rings, as in the paper~\cite{Pps}, the dedualizing
complexes describe derived Morita equivalences~\cite[Section~6]{Pps}.
 Such or very similar complexes are called ``tilting complexes''
in~\cite[Theorem~1.1 and Definition~3.4]{Ric2}
and~\cite[Definition~14.4.1 and Theorem~14.4.16]{Yek-book} (these are
the ``small tilting complexes'', to be distinguished from the ``large
tilting complexes'' of~\cite{PV,NSZ,BHM}).
 What we call the homothety isomorphism condition is called ``derived
Morita property'' in~\cite[Definition~13.1.5]{Yek-book}.
 The one-term complex of $A$\+$A$\+bimodules $A$, for an associative
ring $A$, is the archetypal example of a dedualizing complex of
bimodules.

 In the context of (pseudo-dualizing) complexes of bicomodules over
a pair of coassociative coalgebras over a field, as in~\cite{Ppc},
the derived Morita--Takeuchi equivalences are described by
\emph{dualizing} complexes of bicomodules~\cite[Theorem~7.4]{Ppc},
which also induce the derived comodule-contramodule
correspondences~\cite[Corollary~7.2]{Ppc} (see~\cite[Section~1.2]{Prev}
and~\cite[Sections~8.7 and~9.4]{Pksurv} for an introductory discussion
of the latter).
 The one-term complex of $\mathcal C$\+$\mathcal C$\+bicomodules
$\mathcal C$, for a coassociative coalgebra $\mathcal C$ over
a field~$k$, is the archetypal example of a dualizing complex of
bicomodules.

 In the context of (pseudo-dualizing) complexes of $J$\+torsion
$S$\+modules, for a weakly proregular finitely generated ideal $J$
in a commutative ring $S$, one can speak of the infinite dual Koszul
complex $K^\bu_\infty(S,\bs)=\varinjlim_{n\ge1}K^\bu(S,\bs^n)$ as
the dedualizing complex.
 Here $\bs=(s_1,\dotsc,s_m)$ is any finite sequence of generators of
the ideal $J\subset S$, while $\bs^n$~is a shorthand notation for
the sequence of elements $(s_1^n,\dotsc,s_m^n)$ in~$S$.
 More precisely, any finite complex of $J$\+torsion $S$\+modules
isomorphic to $K^\bu_\infty(S,\bs)$ in the derived category
$\sD(S\Modl)$ is an example of a dedualizing complex of $J$\+torsion
$S$\+modules in the sense of~\cite[Sections~4\+-5]{Pmgm};
see~\cite[Examples~4.8 and~5.8]{Pmgm}.
 We will continue this discussion in
Section~\ref{introd-dedualizing-subsecn} below.

 In Sections~\ref{artinian-finiteness-secn}\+-\ref{dedualizing-secn}
of the present paper, we tie some loose ends left in~\cite{Pmgm} by
establishing a comparison between two definitions of a dedualizing
complex of torsion modules over a commutative ring given
in~\cite{Pmgm}, the one suitable for a more restrictive setting
in~\cite[Section~4]{Pmgm} and the more generally applicable one
in~\cite[Section~5]{Pmgm}.

\subsection{{}}
 The reader can find an elaborate discussion of the philosophy of
co-contra correspondence in the Introduction to the paper~\cite{Pmgm}.
 One highlight: the equivalences of the conventional derived categories
of comodule-like objects (such as torsion modules) with the conventional
derived categories of contramodule-like objects are induced by
dedualizing complexes, while dualizing complexes induce equivalences
between the \emph{coderived categories} of comodule-like objects and
the \emph{contraderived categories} of contramodule-like objects.

 The coderived and contraderived (as well as absolute derived)
categories are collectively known as the \emph{derived categories of
the second kind}.
 We refrain from going into a detailed discussion of derived categories
of the second kind (including the distinction between the Positselski
and the Becker versions of the co/contraderived categories) in this
introduction, as elaborate expositions of their history and philosophy
are now available.
 See, in particular, \cite[Remark~9.2]{PS4}
and~\cite[Section~7]{Pksurv}.

 Simply put, the difference between the derived, the coderived, and
the contraderived categories manifests itself in the context of
unbounded complexes (or DG\+modules over DG\+rings that are \emph{not}
nonpositively cohomologically graded).
 The coderived category can be simply defined as the homotopy category
of unbounded complexes of injective objects, while the contraderived
category is the homotopy category of unbounded complexes of projective
objects.

\subsection{{}} \label{introd-dedualizing-subsecn}
 Let $J$ be a finitely generated ideal in a commutative ring~$S$.
 In this context, the \emph{$J$\+torsion $S$\+modules} form
a full subcategory $S\Modl_{J\tors}$ closed under submodules,
quotients, extensions, and infinite direct sums in the module
category $S\Modl$.
 So $S\Modl_{J\tors}$ is an abelian category, just as $S\Modl$.

 There is a dual-analogous full subcategory $S\Modl_{J\ctra}\subset
S\Modl$ formed by what we call \emph{$J$\+contramodule $S$\+modules}
in the terminology of~\cite{Pmgm,Pcta}.
 In the terminology of~\cite[Section~3.4]{BS}
and~\cite[Section Tag~091N]{SP}, these are called ``derived
$J$\+complete modules''.
 All $J$\+contramodule $S$\+modules are $J$\+adically
complete~\cite[Theorem~5.6]{Pcta}, but they are not $J$\+adically
separated in general~\cite[Example~2.5]{Sim}, \cite[Example~3.20]{Yek0},
\cite[Example~4.33]{PSY}, \cite[Example~2.7(1)]{Pcta}.
 The full subcategory of $J$\+contramodule $S$\+modules
$S\Modl_{J\ctra}$ is closed under kernels, cokernels, extensions, and
infinite products in $S\Modl$.
 So the category $S\Modl_{J\ctra}$ is also abelian.

 In the full generality of arbitrary finitely generated ideals $J$
in commutative rings $S$, the \emph{MGM duality/equivalence}, as
interpreted in~\cite{Pmgm}, is a triangulated equivalence between
two full triangulated subcategories in the derived category
$\sD(S\Modl)$.
 The full subcategory $\sD_{J\tors}(S\Modl)$ formed by complexes with
$J$\+torsion cohomology modules in $\sD(S\Modl)$ is equivalent to
the full subcategory $\sD_{J\ctra}(S\Modl)$ formed by complexes with
$J$\+contramodule cohomology modules~\cite[Theorem~3.4]{Pmgm}
\begin{equation} \label{non-wpr-mgm-duality}
 \sD_{J\tors}(S\Modl)\simeq\sD_{J\ctra}(S\Modl).
\end{equation}
 Under the simplifying assumption of weak proregularity of the ideal
$J$, the equivalence~\eqref{non-wpr-mgm-duality} takes the form
of an equivalence between the derived categories of the two abelian
categories of $J$\+torsion $S$\+modules and $J$\+contramodule
$S$\+modules~\cite[Corollary~3.5 or Theorem~5.10]{Pmgm}
\begin{equation} \label{wpr-mgm-duality}
 \sD(S\Modl_{J\tors})\simeq\sD(S\Modl_{J\ctra}).
\end{equation}

 Both the triangulated equivalences~\eqref{non-wpr-mgm-duality}
and~\eqref{wpr-mgm-duality} are provided by the functors of tensor
product with and $\boR\Hom$ from the so-called \emph{infinite dual
Koszul complex} $K^\bu_\infty(S,\bs)$, where $\bs=(s_1,\dotsc,s_m)$
is a finite sequence of generators of the ideal $J\subset S$.
 The complex $K^\bu_\infty(S,\bs)$ is an augmented version of
the \v Cech complex computing the cohomology of the structure sheaf
on the quasi-compact open subscheme $U=\Spec S\setminus\Spec S/J$
in the affine scheme $X=\Spec S$.

 The complex $K^\bu_\infty(S,\bs)$ is a finite complex of countably
presented flat $S$\+modules with $J$\+torsion cohomology modules.
 In the approach of~\cite[Sections~4 and~5]{Pmgm}, one is supposed to
choose a finite complex of $J$\+torsion $S$\+modules $B^\bu$
quasi-isomorphic to $K^\bu_\infty(S,\bs)$.
 In the terminology of~\cite{Pmgm}, the complex $B^\bu$ is called
a \emph{dedualizing complex} of $J$\+torsion $S$\+modules.
 Let us emphasize that the equivalence~\eqref{wpr-mgm-duality},
induced by a dedualizing complex $B^\bu$, is \emph{an equivalence
of the conventional derived categories}.

\subsection{{}} \label{introd-dualizing-subsecn}
 The topic of \emph{covariant Serre--Grothendieck duality} was
initiated in the paper~\cite{IK} and taken up in the introduction
to~\cite{Neem} and in the dissertation~\cite{Mur}.
 The present author's take on it can be found in the papers~\cite{Pfp},
\cite[Section~7]{Pps} and the preprint~\cite[Section~6]{Pcosh}.
 The most straightforward formulation is that \emph{the datum of
a dualizing complex induces a covariant triangulated equivalence
between the homotopy categories of unbounded complexes of injective
and projective modules}.
 In our terminology, these are called the \emph{coderived} and
the \emph{contraderived} category.

 In the context closer to the present paper, for a commutative ring $R$
with an ideal $I$, one considers a \emph{dualizing complex of
$I$\+torsion $R$\+modules}~$D^\bu$ \,\cite[Remark~4.10]{Pmgm}.
 The simplest definition~\cite[Section~D.1]{Pcosh}, for an ideal
$I$ in a Noetherian ring $R$, says that a finite complex of injective
$I$\+torsion $R$\+modules $D^\bu$ is a \emph{dualizing complex} if,
for every integer $n\ge1$, the finite complex of injective
$R/I^n$\+modules $\Hom_R(R/I^n,D^\bu)$ is a dualizing complex for
the Noetherian commutative ring~$R/I^n$.
 (Cf.~\cite[Lemma~D.1.3]{Pcosh} or
Theorem~\ref{definitions-of-dualizing-complex-comparison}\,%
(2)\,$\Leftrightarrow$\,(3) below in this paper, claiming that
it suffices to impose this condition for $n=1$.)
 If a dualizing complex $D_R^\bu$ for the Noetherian commutative ring
$R$ is given, then a dualizing complex of $I$\+torsion $R$\+modules
$D^\bu$ can be constructed as the subcomplex of all $I$\+torsion
elements in~$D_R^\bu$.

 Then the result of~\cite[Theorem~D.1.4]{Pcosh} claims that the datum
of a dualizing complex of $I$\+torsion $R$\+modules $D^\bu$ induces
a triangulated equivalence between the coderived category of
the abelian category of $I$\+torsion $R$\+modules and the contraderived
category of the abelian category of $I$\+contramodule $R$\+modules,
\begin{equation} \label{covariant-serre-grothendieck-duality}
 \sD^\co(R\Modl_{I\tors})\simeq\sD^\ctr(R\Modl_{I\ctra}).
\end{equation}
 A more general version of~\eqref{covariant-serre-grothendieck-duality}
applicable to $I$\+adically coherent (rather than only Noetherian)
rings $R$ with a weakly proregular finitely generated ideal $I\subset R$
is proved in the present paper as
Corollary~\ref{dualizing-complex-becker-co-contra-derived-equiv}
or Corollary~\ref{dualizing-complex-all-derived-equivs}.

 Other generalizations of~\eqref{covariant-serre-grothendieck-duality},
which can be found in the preprint~\cite{Pcosh}, include a version
for torsion modules and contramodules with respect to centrally
generated ideals in noncommutative Noetherian
rings~\cite[Theorem~D.5.10]{Pcosh} and a version for discrete modules
and contramodules over topological rings with a countable base of
neighborhoods of zero formed by open two-sided
ideals~\cite[Theorem~E.2.9]{Pcosh}.

\subsection{{}}
 The main result of this paper is a joint generalization of
the triangulated equivalences~\eqref{wpr-mgm-duality}
and~\eqref{covariant-serre-grothendieck-duality} arising in the relative
context with a morphism of ring-ideal pairs $f\:(R,I)\rarrow(S,J)$.
 The MGM duality~\eqref{wpr-mgm-duality} along the fibers (i.~e.,
``in the direction of $(S,J)$ relative to~$(R,I)$'') is being built
on top of the covariant Serre--Grothendieck
duality~\eqref{covariant-serre-grothendieck-duality} along the base
of the fibration (i.~e., ``in the direction of~$(R,I)$'').

 Building~\eqref{wpr-mgm-duality} on top
of~\eqref{covariant-serre-grothendieck-duality}, rather than the other
way around, is dictated by very general considerations of the nature
of the main available construction of a mixture of the conventional
derived category with the coderived or contraderived category, called
the \emph{semiderived category}.
 The construction of the semiderived category builds the conventional
derived category on top of the co/contraderived category, and not
the other way around.

\subsection{{}}
 The terminology ``coderived category'', introduced originally in
the note~\cite{Kel}, refers to the basic observation that, in certain
contexts, one is supposed to consider the derived categories of modules
and the coderived categories of comodules (as well as the contraderived
categories of contramodules).
 This point of view was used, in particular, in the book~\cite{Psemi}.
 
 In the terminological system of the book~\cite{Psemi} and our
subsequent publications, the prefix ``semi'' means very roughly
``a half of this and a half of that'', or more specifically a mixture
of ring-like and coalgebra-like features.
 So a \emph{semialgebra} is ``an algebra in a half of the variables
and a coalgebra in the other half of the variables'', etc.
 Likewise, a \emph{semiderived category} is a mixture of
the conventional unbounded derived category with either the coderived
or the contraderived category.

 The constructions of semiderived categories presume a relative
situation with a semialgebra (``an algebra over a coalgebra'')
as in~\cite{Psemi}, or a homomorphism of rings as in~\cite{Pfp},
or a morphism of schemes as in~\cite[Section~8]{Pcosh}, or a morphism
of ind-schemes as in~\cite{Psemten}, etc.
 The construction refers to the respective forgetful functor in
algebraic contexts, or the direct image functor in geometric contexts.
 In the context of the present paper with a morphism of ring-ideal
pairs $f\:(R,I)\rarrow(S,J)$, the constructions of the semiderived
categories refer to the functors of restriction of scalars, which 
assign to a $J$\+torsion $S$\+module its underlying $I$\+torsion
$R$\+module, or assign to a $J$\+contramodule $S$\+module its
underlying $I$\+contramodule $R$\+module.

 The semiderived category (or more specifically, the semicoderived
category) of $J$\+torsion $S$\+modules
$\sD^\sico_{(R,I)}(S\Modl_{J\tors})$ is defined as the triangulated
Verdier quotient category of the homotopy category of unbounded
complexes of $J$\+torsion $S$\+modules by the thick subcategory of
complexes that are \emph{coacyclic as complexes of $I$\+torsion
$R$\+modules}.
 Similarly, the semiderived category (or more specifically,
the semicontraderived category) of $J$\+contramodule $S$\+modules
$\sD^\sictr_{(R,I)}(S\Modl_{J\ctra})$ is defined as the quotient
category of the homotopy category of unbounded complexes of
$J$\+contramodule $S$\+modules by the thick subcategory of complexes
that are \emph{contraacyclic as complexes of $I$\+contramodule
$S$\+modules}.

\subsection{{}} \label{introd-relative-dualizing-subsecn}
 The main result of this paper is the following triangulated
equivalence of \emph{semico-semicontra correspondence}.
 We consider a commutative ring $R$ with a weakly proregular finitely
generated ideal $I\subset R$, and a commutative ring $S$ with
a weakly proregular finitely generated ideal $J\subset S$.

 The ring $R$ assumed to be $I$\+adically coherent.
 There are further additional assumptions on homological dimension,
most notably that all fp\+injective $I$\+torsion $R$\+modules have
finite injective dimensions (this trivially holds if the ring $R$
is $I$\+adically Noetherian, as all fp\+injective $I$\+torsion
$R$\+modules are injective as $I$\+torsion $R$\+modules in this case).
 Most importantly, we assume that a dualizing complex of $I$\+torsion
$R$\+modules $D^\bu$ is given.

 Then we consider a ring homomorphism $f\:R\rarrow S$ such that
$f(I)\subset J$, and assume the ring $S$ to be a flat $R$\+module.
 A further ``quotflatness'' assumption on the morphism of ring-ideal
pairs $f\:(R,I)\rarrow(S,J)$ needs to be imposed if the ring $S$ is
not Noetherian.

 Under the listed assumptions, we construct a triangulated equivalence
\begin{equation} \label{semico-semicontra-correspondence}
 \sD^\sico_{(R,I)}(S\Modl_{J\tors})\simeq
 \sD^\sictr_{(R,I)}(S\Modl_{J\ctra}).
\end{equation}
 See our Theorem~\ref{becker-semico-semicontra-correspondence}
or~\ref{positselski-semico-semicontra-correspondence}.

 The triangulated
equivalence~\eqref{semico-semicontra-correspondence} is provided
by the left derived functor of tensor product with and the functor
$\boR\Hom$ from what we call a \emph{relative dualizing complex}
of $J$\+torsion $S$\+modules~$U^\bu$.
 The complex $U^\bu$ is constructed as a finite complex of
$J$\+torsion $S$\+modules quasi-isomorphic to the tensor product
$K^\bu_\infty(S,\bs)\ot_R D^\bu$, where $K^\bu_\infty(S,\bs)$ is
the infinite dual Koszul complex of $S$\+modules (for a finite
sequence of generators~$\bs$ of the ideal $J\subset S$) and
$D^\bu$ is our dualizing complex of $I$\+torsion $R$\+modules.
 So $U^\bu$ is ``a mixture of the dualizing complex along $R$
and the dedualizing complex in the direction of $S$ relative to~$R$''.

\subsection{{}}
 We refer to the introduction to~\cite{Pps} for a further discussion
of the philosophy and examples of intermediate and mixed versions
of the co-contra correspondence, including versions with one of
the dualities built on top of another one.
 The most general results of this paper apply to
a \emph{pseudo-dualizing complex} of $J$\+torsion $S$\+modules,
in the context of a weakly proregular finitely generated ideal $J$
in a commutative ring~$S$.

 The pseudo-dualizing complexes are a common generalization of
the dedualizing complexes from
Section~\ref{introd-dedualizing-subsecn}, the dualizing complexes
from Section~\ref{introd-dualizing-subsecn}, and the relative dualizing
complexes from Section~\ref{introd-relative-dualizing-subsecn}.
 The definition of a pseudo-dualizing complex of $J$\+torsion
$S$\+modules is obtained from the definition of a dedualizing complex
given in~\cite[Section~5]{Pmgm} by dropping the finite
projective/contraflat dimension condition and suitably relaxing
the finite generatedness condition.
 
 The exposition in Sections~\ref{auslander-and-bass-secn}\+-%
\ref{minimal-classes-secn} of the present paper, dedicated
to pseudo-dualizing complexes and pseudo-derived equivalences, is
parallel (and largely similar) to the exposition in the respective
sections of the papers~\cite{Pps,Ppc}.
 Detailed discussions are available in the introductions to~\cite{Pps}
and~\cite{Ppc} (see~\cite[Sections~0.5 and 0.7]{Pps}
or~\cite[Sections~1.6\+-1.7]{Ppc}), so we restrict ourselves here
to a brief sketch.

\subsection{{}}
 A pseudo-dualizing complex of $J$\+torsion $S$\+modules $L^\bu$
is supposed to be, first of all, a finite complex (of $J$\+torsion
$S$\+modules).
 Let $d_1$ and~$d_2$ be two integers such that the complex $L^\bu$
is concentrated in the cohomological degrees from~$-d_1$ to~$d_2$.
 The key concept of \emph{corresponding classes} (of $J$\+torsion
$S$\+modules and $J$\+contramodule $S$\+modules) and the related
constructions of the \emph{maximal} and \emph{minimal} corresponding
classes depend on numerical (integer) parameters $l_1\ge d_1$ and
$l_2\ge d_2$.

 We construct an increasing sequence of pairs of maximal corresponding
classes $\sE_{l_1}(L^\bu)\subset S\Modl_{J\tors}$ and
$\sF_{l_1}(L^\bu)\subset S\Modl_{J\ctra}$.
 So we have $\sE_{d_1}(L^\bu)\subset\sE_{d_1+1}(L^\bu)\subset
\sE_{d_1+2}(L^\bu)\subset\dotsb$ and $\sF_{d_1}(L^\bu)\subset
\sF_{d_1+1}(L^\bu)\subset\sF_{d_1+2}(L^\bu)\subset\dotsb$.
 The class $\sF_{l_1}(L^\bu)$ is also known as
the \emph{Auslander class}, while the class $\sE_{l_1}(L^\bu)$
is called the \emph{Bass class}.

 As the integer $l_1\ge d_1$ varies, the classes $\sE_{l_1}$ and
$\sF_{l_1}$ only differ from each other ``by finite (co)resolution
dimension'', so their derived categories stay the same.
 We put $\sD_{L^\bu}'(S\Modl_{J\tors})=\sD(\sE_{l_1})$
and $\sD_{L^\bu}''(S\Modl_{J\ctra})=\sD(\sF_{l_1})$.
 The triangulated category $\sD_{L^\bu}'(S\Modl_{J\tors})$ can be called
the \emph{lower pseudo-coderived category of $J$\+torsion $S$\+modules},
and the triangulated category $\sD_{L^\bu}''(S\Modl_{J\ctra})$ is
the \emph{lower pseudo-contraderived category of $J$\+contramodule
$S$\+modules}.

 Our Theorem~\ref{bass-auslander-derived-equivalence} claims, as one
of its cases, a triangulated equivalence
\begin{equation} \label{lower-pseudo-derived-equivalence}
 \sD_{L^\bu}'(S\Modl_{J\tors}) \simeq
 \sD_{L^\bu}''(S\Modl_{J\ctra})
\end{equation}
provided by the derived functors of the tensor product with and
$\Hom$ from the pseudo-dualizing complex~$L^\bu$.

 We also construct a decreasing sequence of pairs of minimal
corresponding classes $\sE^{l_2}(L^\bu)\subset S\Modl_{J\tors}$ and
$\sF^{l_2}(L^\bu)\subset S\Modl_{J\ctra}$.
 So we have $\sE^{d_2}(L^\bu)\supset\sE^{d_2+1}(L^\bu)\supset
\sE^{d_2+2}(L^\bu)\supset\dotsb$ and $\sF^{d_2}(L^\bu)\supset
\sF^{d_2+1}(L^\bu)\supset\sF^{d_2+2}(L^\bu)\supset\dotsb$.

 Once again, as the integer $l_2\ge d_2$ varies, the classes $\sE^{l_2}$
and $\sF^{l_2}$ only differ from each other ``by finite (co)resolution
dimension'', so their derived categories stay the same.
 We put $\sD^{L^\bu}_\prime(S\Modl_{J\tors})=\sD(\sE^{l_2})$
and $\sD^{L^\bu}_{\prime\prime}(S\Modl_{J\ctra})=\sD(\sF^{l_2})$.
 The triangulated category $\sD^{L^\bu}_\prime(S\Modl_{J\tors})$
can be called the \emph{upper pseudo-coderived category of $J$\+torsion
$S$\+modules}, and the triangulated category
$\sD^{L^\bu}_{\prime\prime}(S\Modl_{J\ctra})$ is the \emph{upper
pseudo-contraderived category of $J$\+contramodule $S$\+modules}.

 Our Theorem~\ref{minimal-classes-derived-equivalence} claims, as
one of its cases, a triangulated equivalence
\begin{equation} \label{upper-pseudo-derived-equivalence}
 \sD^{L^\bu}_\prime(S\Modl_{J\tors}) \simeq
 \sD^{L^\bu}_{\prime\prime}(S\Modl_{J\ctra}),
\end{equation}
which is also provided by the derived functors of the tensor product
with and $\Hom$ from the pseudo-dualizing complex~$L^\bu$.

\subsection{{}}
 Summarizing the results of Sections~\ref{auslander-and-bass-secn}\+-%
\ref{minimal-classes-secn} of the present paper and using
the discussion of adjoint functors in~\cite[Section~2]{Ppc}, we
obtain a diagram of triangulated functors, triangulated equivalences, 
commutativities, and adjunctions
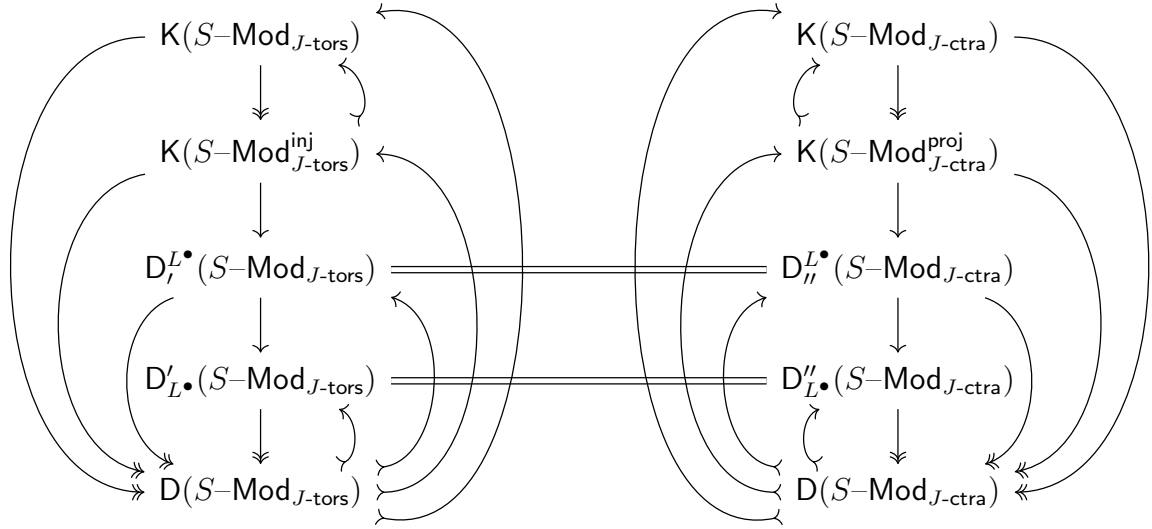
\begin{equation} \label{big-diagram}
\!\!\!\!\!\!\!\!\!\!\begin{tikzcd}
\sK(S\Modl_{J\tors}) \arrow[d, two heads]
\arrow[dddd, two heads, bend right=90] &&&&&
\sK(S\Modl_{J\ctra}) \arrow[d, two heads]
\arrow[dddd, two heads, bend left=90] \\
\sK(S\Modl_{J\tors}^\inj) \arrow[d] 
\arrow[u, tail, bend right=70]
\arrow[ddd, two heads, bend right=80] &&&&&
\sK(S\Modl_{J\ctra}^\proj) \arrow[d] 
\arrow[u, tail, bend left=70]
\arrow[ddd, two heads, bend left=80] \\
\sD^{L^\bu}_{\prime}(S\Modl_{J\tors}) \arrow[d]
\arrow[rrrrr, Leftrightarrow, no head, no tail]
\arrow[dd, two heads, bend right=72] &&&&&
\sD^{L^\bu}_{\prime\prime}(S\Modl_{J\ctra}) \arrow[d]
\arrow[dd, two heads, bend left=72] \\
\sD'_{L^\bu}(S\Modl_{J\tors}) \arrow[d, two heads]
\arrow[rrrrr, Leftrightarrow, no head, no tail] &&&&&
\sD''_{L^\bu}(S\Modl_{J\ctra}) \arrow[d, two heads] \\
\sD(S\Modl_{J\tors}) \arrow[uuuu, tail, bend right=102]
\arrow[uuu, tail, bend right=90]
\arrow[uu, tail, bend right=78]
\arrow[u, tail, bend right=70] &&&&&
\sD(S\Modl_{J\ctra}) \arrow[uuuu, tail, bend left=102]
\arrow[uuu, tail, bend left=90]
\arrow[uu, tail, bend left=78]
\arrow[u, tail, bend left=70]
\end{tikzcd}\!\!\!\!\!\!\!\!
\end{equation}
 Here the notation $\sK(\sT)$ stands for the homotopy category of
(unbounded) complexes in an additive category~$\sT$.
 The full subcategory of injective objects in an abelian category
$\sA$ is denoted by $\sA^\inj\subset\sA$, while the full subcategory
of projective objects in an abelian category $\sB$ is denoted by
$\sB^\proj\subset\sB$.
 The homotopy category of injective objects $\sK(S\Modl_{J\tors}^\inj)$
is otherwise known as the \emph{Becker coderived category}
$\sK(S\Modl_{J\tors}^\inj)\simeq\sD^\bco(S\Modl_{J\tors})$.
 The homotopy category of projective objects
$\sK(S\Modl_{J\ctra}^\proj)$ is otherwise known as the \emph{Becker
contraderived category}
$\sK(S\Modl_{J\ctra}^\proj)\simeq\sD^\bctr(S\Modl_{J\ctra})$.
 See Theorem~\ref{becker-co-contra-derived-of-loc-pres-abelian}.

 The two horizontal double lines are the triangulated
equivalences~\eqref{lower-pseudo-derived-equivalence}
and~\eqref{upper-pseudo-derived-equivalence}.
 The middle square including these two horizontal double lines
is commutative.

 The arrows with double heads denote triangulated Verdier quotient
functors, while the arrows with tails denote fully faithful
triangulated functors.
 The downwards-directed outer curvilinear arrows with double heads
are the compositions of the downwards directed straight arrows.
 The upwards-directed inner curvilinear arrows with tails are
adjoint on the respective sides to the respective
downwards-directed arrows.
 Specifically, in the left-hand part of the diagram,
the upwards-directed inner curvilinear arrows with tails are right
adjoint to the respective arrows going down.
 In the right-hand part of the diagram, the upwards-directed inner
curvilinear arrows with tails are left adjoint to the respective arrows
going down (just as the relative positions of these arrows to
the left or to the right of one another may suggest).

 The upper curvilinear fully faithful functors between the homotopy
categories $\sK(S\Modl_{J\tors}^\inj)\rarrow\sK(S\Modl_{J\tors})$ and
$\sK(S\Modl_{J\ctra}^\proj)\rarrow\sK(S\Modl_{J\ctra})$ are
induced by the inclusions of additive/abelian categories
$S\Modl_{J\tors}^\inj\rarrow S\Modl_{J\tors}$ and
$S\Modl_{J\ctra}^\proj\rarrow S\Modl_{J\ctra}$.
 The straight downwards-directed arrows in the leftmost column, with
the exception of the uppermost one, are the triangulated functors
between the homotopy and derived categories induced by the exact
inclusions of additive/exact/abelian categories $S\Modl_{J\tors}^\inj
\rarrow\sE^{l_2}\rarrow\sE_{l_1}\rarrow S\Modl_{J\tors}$.
 The straight downwards-directed arrows in the rightmost column, with
the exception of the uppermost one, are the triangulated functors
between the homotopy and derived categories induced by the exact
inclusions of additive/exact/abelian categories $S\Modl_{J\ctra}^\proj
\rarrow\sF^{l_2}\rarrow\sF_{l_1}\rarrow S\Modl_{J\ctra}$.

 See also diagrams~\eqref{abstract-big-diagram}
and~\eqref{pseudo-derived-equivs-commutativity-diagram}
in Section~\ref{abstract-classes-secn}.

\subsection{{}}
 In the case of a dedualizing complex $L^\bu=B^\bu$, one has
$\sE_{l_1}=S\Modl_{J\tors}$ and $\sF_{l_1}=S\Modl_{J\ctra}$ for
large enough values of the integer parameter~$l_1$.
 So the triangulated functors $\sD'_{L^\bu}(S\Modl_{J\tors})
\rarrow\sD(S\Modl_{J\tors})$ and $\sD''_{L^\bu}(S\Modl_{J\ctra})
\rarrow\sD(S\Modl_{J\ctra})$ are triangulated equivalences,
the lower pseudo-derived categories coincide with the conventional
derived categories, and the lower level of
the diagram~\eqref{big-diagram} collapses.

 In the case of a dualizing complex $L^\bu=D^\bu$, depending on
the specifics of injective/projective dimension assumptions, the upper
pseudo-coderived category $\sD^{L^\bu}_{\prime}(S\Modl_{J\tors})$
coincides with the Becker coderived category
$\sD^\bco(S\Modl_{J\tors})$ as well as with the Positselski coderived
category $\sD^\co(S\Modl_{J\tors})$.
 The upper pseudo-contraderived category
$\sD^{L^\bu}_{\prime\prime}(S\Modl_{J\ctra})$ coincides with
the Becker contraderived category $\sD^\bctr(S\Modl_{J\ctra})$,
and often also with the Positselski contraderived category
$\sD^\ctr(S\Modl_{J\ctra})$.
 So the triangulated functors $\sK(S\Modl_{J\tors}^\inj)\rarrow
\sD^{L^\bu}_{\prime}(S\Modl_{J\tors})$ and
$\sK(S\Modl_{J\ctra}^\proj)\rarrow
\sD^{L^\bu}_{\prime\prime}(S\Modl_{J\ctra})$ are triangulated
equivalences, and the next-to-upper level of
the diagram~\eqref{big-diagram} collapses.

\subsection*{Acknowledgement}
 I~am grateful to Jan \v St\!'ov\'\i\v cek for helpful discussions.
 I~wish to thank Amnon Yekutieli for his comments on the Introduction
to this paper, which motivated me to write the explanatory
Section~\ref{introd-morita-subsecn}.
 The author is supported by the GA\v CR project 23-05148S
and the Institute of Mathematics, Czech Academy of Sciences
(research plan RVO:~67985840).

\Section{Preliminaries on the Weak Proregularity Condition}
\label{prelims-on-wpr-secn}

 We refer to the papers~\cite{PSY,Pmgm,Yek3,Pdc} for a discussion of
weakly proregular finitely generated ideals in commutative rings.
 This section offers a brief sketch.

 Let $S$ be a commutative ring and $s\in S$ be an element.
 The notation $K_\bu(S,s)$ stands for the two-term Koszul complex
of free $S$\+modules $S\overset s\rarrow S$ concentrated in
the homological degrees $0$ and~$1$ (i.~e., the cohomological degrees
$-1$ and~$0$).
 The notation $K^\bu(S,s)=\Hom_S(K_\bu(S,s),S)$ stands for the same
complex placed in the cohomological degrees $0$ and~$1$; so we have
$K^\bu(S,s)=K_\bu(S,s)[-1]$.

 Given an integer $n\ge1$, consider also the complexes $K_\bu(S,s^n)$
and $K^\bu(S,s^n)$ (where $s^n$ is the $n$\+th power of~$s$).
 The complexes $K_\bu(S,s^n)$ form a projective system with respect to
the natural maps
$$
 \xymatrixcolsep{3em}
 \xymatrix{
  S \ar[r]^{s^{n+1}} \ar[d]_s & S \ar[d]^{\id} \\
  S \ar[r]^{s^n} & S
 }
$$
while the complexes $K^\bu(S,s^n)$ form an inductive system with respect
to the dual maps
$$
 \xymatrixcolsep{3em}
 \xymatrix{
  S \ar[r]^{s^n} \ar[d]_{\id} & S \ar[d]^s \\
  S \ar[r]^{s^{n+1}} & S
 }
$$
 Put $K_\infty^\bu(S,s)=\varinjlim_{n\ge1}K^\bu(S,s^n)$; so
$K_\infty^\bu(S,s)$ is the two-term complex $S\rarrow S[s^{-1}]$ 
concentrated in the cohomological degrees~$0$ and~$1$.
 Here $S[s^{-1}]$ is the notation for the ring $S$ with
the element~$s$ formally inverted, i.~e., in other words, $S[s^{-1}]$
is the localization of $S$ at the multiplicative subset
$\{1,s,s^2,s^3,\dotsc\}$.

 Let $J$ be a finitely generated ideal in a commutative ring~$S$.
 Choose a finite sequence of generators $s_1$,~\dots,~$s_m$ of the ideal
$J\subset S$, and denote it for brevity by $\bs=(s_1,\dotsc,s_m)$.
 Put
$$
 K_\bu(S,\bs)=
 K_\bu(S,s_1)\ot_S K_\bu(S,s_2)\ot_S\dotsb\ot_S K_\bu(S,s_m)
$$
and
$$
 K^\bu(S,\bs)=
 K^\bu(S,s_1)\ot_S K^\bu(S,s_2)\ot_S\dotsb\ot_S K^\bu(S,s_m).
$$
 So $K_\bu(S,\bs)$ is a finite complex of finitely generated free
$S$\+modules concentrated in the homological degrees from $0$ to~$m$
(which means the cohomological degrees from $-m$ to~$0$), while
$K^\bu(S,\bs)\simeq\Hom_S(K_\bu(S,\bs),S)\simeq K_\bu(S,\bs)[-m]$ is
a finite complex of finitely generated free $S$\+modules concentrated
in the cohomological degrees from $0$ to~$m$.

 Put $\bs^n=(s_1^n,\dotsc,s_m^n)$.
 Taking the tensor products of the natural maps of complexes above,
one obtains a projective system of complexes $K_\bu(S,\bs^n)$ and
an inductive system of complexes $K^\bu(S,\bs^n)$.
 Finally, we set
$$
 K^\bu_\infty(S,\bs) = K^\bu_\infty(S,s_1)\ot_S K^\bu_\infty(S,s_2)
 \ot_S\dotsb\ot_S K^\bu_\infty(S,s_m).
$$
 So $K^\bu_\infty(S,\bs)=\varinjlim_{n\ge1}K^\bu(S,\bs^n)$ is a finite
complex of countably presented flat $S$\+modules concentrated in
the cohomological degrees from $0$ to~$m$.
 (In fact, $K^\bu_\infty(S,\bs)$ is a complex of \emph{very flat}
$S$\+modules in the sense of~\cite[Section~1.1]{Pcosh}.)

 The complex $K_\bu(S,\bs)$ is called the \emph{Koszul complex}, while
the complex $K^\bu(S,\bs)$ is called the \emph{dual Koszul complex}.
 The complex $K^\bu_\infty(S,\bs)$ is called the \emph{infinite dual
Koszul complex}.

 A construction of a finite complex of countably generated free
$S$\+modules
$$
 T^\bu(S,\bs)=T^\bu(S,s_1)\ot_S T^\bu(S,s_2)\ot_S\dotsb\ot_S
 T^\bu(S,s_m)
$$
quasi-isomorphic to the complex $K^\bu(S,\bs)$ can be found
in~\cite[formula~(6.7) and Lemma~6.9]{DG},
\cite[Section~5]{PSY}, or~\cite[Section~2]{Pmgm}.
 Just as the complex $K^\bu(S,\bs)$, the complex $T^\bu(S,\bs)$ is
concentrated in the cohomological degrees from $0$ to~$m$.

 The complex $T^\bu(S,\bs)$ is the direct limit of a direct system
of finite complexes of finitely generated free $S$\+modules
$T_n^\bu(S,\bs)$ with termwise split monomorphisms
$T_n^\bu(S,\bs)\rarrow T_{n+1}^\bu(S,\bs)$ as the transition maps.
 The complex $T_n^\bu(S,\bs)$ is naturally homotopy equivalent to
the complex $K^\bu(S,\bs^n)$ \,\cite[Section~5]{PSY},
\cite[Sections~2 and~5]{Pmgm}.

 The complex $T^\bu(S,\bs)$ does \emph{not} depend on the choice of
a finite sequence of generators of a given finitely generated ideal
$J\subset S$, up to a natural homotopy
equivalence~\cite[Theorem~6.1]{PSY}.
 In other words, the complex $K^\bu_\infty(S,\bs)$ does \emph{not}
depend on the sequence~$\bs$, but only on the ideal $J\subset S$, up to
a natural chain of quasi-isomorphisms~\cite[Proposition~2.20]{Yek3},
\cite[Lemma~2.1]{Pdc}.

 A projective system of $S$\+modules (or abelian groups) $H_1\larrow
H_2\rarrow H_3\larrow\dotsb$ indexed by the poset of positive integers
is said to be \emph{pro-zero} if for every integer $j\ge1$ there exists
an integer $k>j$ such that the transition map $H_k\rarrow H_j$ vanishes.
 A finite sequence of elements~$\bs$ in a commutative ring $S$ is said
to be \emph{weakly proregular} if the projective system of the homology
groups of the Koszul complexes $(H_iK_\bu(S,\bs^n))_{n\ge1}$ is pro-zero
for every $i>0$.

 The weak proregularity property of a finite sequence of elements~$\bs$
in a commutative ring $S$ depends only on the ideal $J$ generated
by~$\bs$ in $S$ (and even only on the radical $\sqrt{J}$ of
the ideal~$J$), rather than on the sequence~$\bs$
itself~\cite[Corollary~6.2 or~6.3]{PSY}.
 Thus one can speak about \emph{weakly proregular finitely generated
ideals} $J$ in commutative rings~$S$.
 In a Noetherian commutative ring $S$, all ideals are weakly
proregular~\cite[Theorem~4.34]{PSY}, \cite[Section~1]{Pmgm},
\cite[Theorem~3.3]{Yek3}.

 Let $S$ be a commutative ring, $s\in S$ be an element, and $J\subset S$
be an ideal.
 An $S$\+module $M$ is said to be \emph{$s$\+torsion} if for every
$m\in M$ there exists an integer $n\ge1$ such that $s^nm=0$ in~$M$.
 Equivalently, this means that $S[s^{-1}]\ot_SM=0$.
 An $S$\+module $M$ is said to be \emph{$J$\+torsion} if $M$ is
$s$\+torsion for every $s\in S$.
 It suffices to check this condition for the element~$s$ ranging over
any chosen set of generators~$\{s_j\}$ of the ideal~$J$.

 The full subcategory $S\Modl_{J\tors}$ of $J$\+torsion $S$\+modules
is closed under extensions, submodules, quotients, and infinite
direct sums in the abelian category of $S$\+modules $S\Modl$.
 In other words, one says that $S\Modl_{J\tors}$ is a \emph{Serre
subcategory closed under infinite direct sums}, or a \emph{localizing
subcategory}, or in a different terminology, a \emph{hereditary
torsion class} in $S\Modl$.
 It follows that $S\Modl_{J\tors}$ is a Grothendieck abelian category,
and the fully faithful inclusion functor $S\Modl_{J\tors}\rarrow S\Modl$
is exact and preserves infinite direct sums.

 An $S$\+module $P$ is said to be an \emph{$s$\+contramodule} if
$\Hom_S(S[s^{-1}],P)=0=\Ext^1_S(S[s^{-1}],P)$.
 One does not need to impose higher Ext vanishing conditions, as
the projective dimension of the $S$\+module $S[s^{-1}]$ never
exceeds~$1$ \,\cite[proof of Lemma~2.1]{Pcta}.
 An $S$\+module $P$ is said to be a \emph{$J$\+contramodule} (or
a \emph{$J$\+contramodule $S$\+module}) if $P$ is an $s$\+contramodule
for every $s\in S$.
 It suffices to check this condition for the element~$s$ ranging over
any chosen set of generators~$\{s_j\}$ of the ideal~$J$
\,\cite[Theorem~5.1]{Pcta}.

 The full subcategory $S\Modl_{J\ctra}$ of $J$\+contramodule
$S$\+modules is closed under extensions, kernels, cokernels, and
infinite products in the abelian category of $S$\+modules $S\Modl$
\cite[Proposition~1.1]{GL}, \cite[Theorem~1.2(a)]{Pcta}.
 It follows that $S\Modl_{J\ctra}$ is an abelian category with
infinite products, and the fully faithful inclusion functor
$S\Modl_{J\ctra}\rarrow S\Modl$ is exact and preserves
infinite products.

 Let $J$ be a finitely generated ideal in a commutative ring~$S$.
 To any $S$\+module $P$, one can assign its $J$\+adic completion
$\Lambda_J(P)=\varprojlim_{n\ge1}P/J^nP$ \,\cite[Section~1]{GM},
\cite[Section~1]{Yek0}.
 One says that $P$ is \emph{$J$\+adically separated} if the natural
completion map $\lambda_{J,P}\:P\rarrow\Lambda_J(P)$ is injective, and
that $P$ is \emph{$J$\+adically complete} if the map~$\lambda_{J,P}$
is surjective.
 The assumption of finite generatedness of the ideal $J$ implies that
the $S$\+module $\Lambda_J(P)$ is $J$\+adically (separated and)
complete~\cite[Corollaries~1.7 and~3.6]{Yek0}.

 Any $J$\+adically separated and complete $S$\+module is
a $J$\+contramodule (because any $S/J^n$\+module is a $J$\+contramodule
$S$\+module and the class of $J$\+contramodules is closed under
projective limits in $S\Modl$).
 Any $J$\+contramodule $S$\+module is $J$\+adically
complete~\cite[Theorem~5.6]{Pcta}, but it need not be $J$\+adically
separated~\cite[Example~2.5]{Sim}, \cite[Example~3.20]{Yek0},
\cite[Example~4.33]{PSY}, \cite[Example~2.7(1)]{Pcta}.

 A $J$\+contramodule $S$\+module is said to be \emph{quotseparated}
if it is a quotient $S$\+module of a $J$\+adically separated and
complete $S$\+module.
 The full subcategory $S\Modl_{J\ctra}^\qs\subset S\Modl_{J\ctra}$ of
quotseparated $J$\+contramodule $S$\+modules is closed under kernels,
cokernels, and infinite products in $S\Modl_{J\ctra}$ and $S\Modl$
\,\cite[Lemma~1.2]{Pdc}.
 It follows that the category $S\Modl_{J\ctra}^\qs$ is abelian, and
its fully faithful inclusion functors $S\Modl_{J\ctra}^\qs\rarrow
S\Modl_{J\ctra}$ and $S\Modl_{J\ctra}^\qs\rarrow S\Modl$ are exact
and preserve infinite products.
 Every $J$\+contramodule $S$\+module is an extension of two
quotseparated $J$\+contramodule
$S$\+modules~\cite[Proposition~1.6]{Pdc}.

 If the ideal $J\subset S$ is weakly proregular, then every
$J$\+contramodule $S$\+module is
quotseparated~\cite[Corollary~3.7]{Pdc}.
 In fact, a certain (small) part of the weak proregularity condition
on a finitely generated ideal $J\subset S$ is equivalent to all
$J$\+contramodule $S$\+modules being
quotseparated~\cite[Remark~3.8]{Pdc}.

 The following lemma is very basic.
 For a generalization to complexes, see
Lemma~\ref{tors-contra-tensor-hom-for-complexes} below.

\begin{lem} \label{tors-contra-tensor-hom-lemma}
 Let $S$ be a commutative ring and $J\subset S$ be an ideal.
 In this context: \par
\textup{(a)} if $M$ and $N$ are $S$\+modules and either $M$ or $N$ is
$J$\+torsion, then the $S$\+module $M\ot_SN$ is $J$\+torsion; \par
\textup{(b)} if $M$ and $P$ are $S$\+modules and $M$ is $J$\+torsion,
then the $S$\+module\/ $\Hom_S(M,P)$ is a $J$\+contramodule (in
fact, a $J$\+adically separated and complete $S$\+module, if the ideal
$J$ is finitely generated); \par
\textup{(c)} if $M$ and $P$ are $S$\+modules and $P$ is
a $J$\+contramodule, then the $S$\+module\/ $\Hom_S(M,P)$ is
a $J$\+contramodule.
\end{lem}

\begin{proof}
 All the assertions with exception of the one in parentheses in
part~(b) are covered by~\cite[Lemma~6.1]{Pcta}.
 The remaining parenthetical assertion is provable by representing
$M$ as the direct union of its submodules annihilated by $J^n$,
\,$n\ge1$, and noticing that projective limits of $J$\+adically
separated $S$\+modules are $J$\+adically separated.
\end{proof}

 The exact, fully faithful inclusion functor $S\Modl_{J\tors}\rarrow
S\Modl$ has a right adjoint functor, denoted by $\Gamma_J\:
S\Modl\rarrow S\Modl_{J\tors}$.
 The functor $\Gamma_J$ assigns to an $S$\+module $M$ its (obviously
unique) maximal $J$\+torsion submodule~\cite[Section~1]{GM},
\cite[Section~3]{PSY}, \cite[Section~1]{Pmgm}.
 As any Grothendieck category, the abelian category $S\Modl_{J\tors}$
has enough injective objects.
 The injective objects of $S\Modl_{J\tors}$ are precisely all
the direct summands of the $S$\+modules $\Gamma_J(K)$, where $K$
ranges over the class of injective $S$\+modules~\cite[Section~5]{Pmgm}.
 A $J$\+torsion $S$\+module $K$ is injective as an object of
$S\Modl_{J\tors}$ if and only if the submodule of all elements
annihilated by $J^n$ in $K$ is an injective $S/J^n$\+module for
every $n\ge1$.

 The exact, fully faithful inclusion functor $S\Modl_{J\ctra}\rarrow
S\Modl$ has a left adjoint functor, denoted by $\Delta_J\:
S\Modl\rarrow S\Modl_{J\ctra}$.
 In the case of a finitely generated ideal $J\subset S$, the functor
$\Delta_J$ was constructed explicitly in~\cite[Proposition~2.1]{Pmgm};
a more detailed discussion can be found in~\cite[Sections~6--7]{Pcta}.
 In the general case of an arbitrary ideal $J\subset S$, one can
apply~\cite[Example~1.3(4)]{Pper} to a two-term projective resolution
$U^{-1}\rarrow U^0$ of the $S$\+module $\bigoplus_{s\in J}S[s^{-1}]$.
 The abelian category $S\Modl_{J\ctra}$ is locally presentable in
the sense of~\cite[Definition~1.17 and Theorem~1.20]{AR} (locally
$\aleph_1$\+presentable in the case of a finitely generated ideal~$J$)
and has enough projective objects.
 The projective objects of $S\Modl_{J\ctra}$ are precisely all
the direct summands of the $S$\+modules $\Delta_J(P)$, where $P$
ranges of the class of projective (or free)
$S$\+modules~\cite[Section~1]{Pdc}.

 Assume that the ideal $J\subset S$ is finitely generated.
 Then the exact, fully faithful inclusion functor $S\Modl_{J\ctra}^\qs
\rarrow S\Modl$ has a left adjoint functor, denoted by
$\boL_0\Lambda_J\:S\Modl\rarrow S\Modl_{J\ctra}^\qs$.
 It is the $0$\+th left derived functor of the $J$\+adic completion
functor $\Lambda_J$, which is neither left nor right exact
(cf.~\cite[Section~3]{PSY}); see~\cite[Proposition~1.3]{Pdc}.
 The abelian category $S\Modl_{J\ctra}^\qs$ is locally
$\aleph_1$\+presentable and has enough projective objects.
 The projective objects of $S\Modl_{J\ctra}^\qs$ are precisely all
the direct summands of the $S$\+modules $\Lambda_J(P)=
\boL_0\Lambda_J(P)$, where $P$ ranges of the class of projective
(or free) $S$\+modules~\cite[Section~1]{Pdc}.
 A quotseparated $J$\+contramodule $S$\+module $F$ is projective
as an object of $S\Modl_{J\ctra}^\qs$ if and only if
the $S/J^n$\+module $F/J^nF$ is projective for every $n\ge1$
(this is a particular case of~\cite[Corollary~E.1.10(a)]{Pcosh}
in view of~\cite[Proposition~1.5]{Pdc}).

\Section{Preliminaries on Exotic Derived Categories}
\label{prelim-exotic-derived-secn}

 We suggest the survey paper~\cite{Bueh} as the background reference
source on \emph{exact categories in the sense of Quillen}.
 In particular, any abelian category can be viewed as an exact
category with the \emph{abelian exact category structure}.
 Given an exact category $\sT$ and a full additive subcategory
$\sE\subset\sT$ such that $\sE$ is closed under extensions in $\sT$,
we will always endow $\sE$ with the \emph{inherited exact category
structure} in which the admissible short exact sequences in $\sE$
are the admissible short exact sequences in $\sT$ with the terms
belonging to~$\sE$.

 Let $\sE$ be an exact category.
 The definitions of the (bounded or unbounded) conventional derived
categories $\sD^\st(\sE)$ with the symbols $\st=\bb$, $+$, $-$,
or~$\varnothing$ are discussed in~\cite{Neem0}
and~\cite[Section~10]{Bueh}.

 We refer to~\cite[Appendix~A]{Pmgm} and~\cite[Sections~A.1
and~B.7]{Pcosh} for more detailed discussions of the exotic derived
categories $\sD^\st(\sE)$ with the derived category symbols
$\st=\abs+$, $\abs-$, $\abs$, $\co$, $\ctr$, $\bco$, and~$\bctr$.
 Their names are the (bounded or unbounded) \emph{absolute derived
categories}, the \emph{Positselski coderived and contraderived}
categories, and the \emph{Becker coderived and contraderived}
categories.
 A discussion of the Becker coderived and contraderived categories
in the context of abelian categories $\sE$ can be also found in
the paper~\cite{PS4}; see in particular~\cite[Remark~9.2]{PS4} for
the history and terminology.
 The following section is a brief sketch.

 For any symbol $\st=\bb$, $+$, $-$, or~$\varnothing$, we denote by
$\sC^\st(\sE)$ the category of (respectively bounded or unbounded)
complexes in~$\sE$ (and closed morphisms of degree~$0$ between them).
 The notation $\sK^\st(\sE)$ stands for the homotopy category of  
complexes in~$\sE$, i.~e., the additive quotient category of
$\sC^\st(\sE)$ by the ideal of morphisms cochain homotopic to zero.
 So $\sK^\st(\sE)$ is a triangulated category.
 
 A short sequence $0\rarrow K^\bu\rarrow L^\bu\rarrow M^\bu\rarrow0$
of complexes in $\sE$ is said to be (\emph{admissible}) \emph{exact}
if it is exact in $\sE$ at every degree, i.~e., the short sequence
$0\rarrow K^n\rarrow L^n\rarrow M^n\rarrow0$ is admissible exact in
$\sE$ for every integer $n\in\boZ$.
 The class of all such short exact sequences of complexes in $\sE$
defines the \emph{degreewise exact structure} on $\sC(\sE)$.
 A short exact sequence of complexes in $\sE$ can be viewed as
a bicomplex with three rows; as such, it has the total complex.

 A complex in $\sE$ is said to be \emph{absolutely
acyclic}~\cite[Section~A.1]{Pcosh}, \cite[Appendix~A]{Pmgm} if it
belongs to the minimal thick subcategory of $\sK(\sE)$ containing
all the totalizations of short exact sequences of complexes in~$\sE$.
 By~\cite[Proposition~8.12]{PS5}, the full subcategory of absolutely
acyclic complexes in $\sC(\sE)$ is precisely the closure of the class
of all contractible complexes under extensions (in the degreewise exact
structure) and direct summands.
 We denote the full subcategory of absolutely acyclic complexes by
$\Ac^\abs(\sE)\subset\sK(\sE)$ or $\Ac^\abs(\sE)\subset\sC(\sE)$.

 The definitions of the full subcategories $\Ac^{\abs+}(\sE)\subset
\sK^+(\sE)$ and $\Ac^{\abs-}(\sE)\subset\sK^-(\sE)$ are similar
(the same construction is performed within the realm of bounded
below or bounded above complexes, respectively).
 In fact, a bounded below (respectively, above) complex is absolutely acyclic as a bounded below (resp., above) complex if and only if it is
absolutely acyclic in the world of unbounded complexes.
 A bounded complex is absolutely acyclic if and only if it is acyclic
in the conventional sense~\cite[Lemma~A.1.2]{Pcosh}.

 The (one-sided bounded or unbounded) \emph{absolute derived categories}
of an exact category $\sE$ are defined as the triangulated Verdier
quotient categories
$$
 \sD^\abs(\sE)=\sK(\sE)/\Ac^\abs(\sE)
 \quad\text{and}\quad
 \sD^{\abs\pm}(\sE)=\sK^\pm(\sE)/\Ac^{\abs\pm}(\sE).
$$

 An exact category $\sE$ is said to have \emph{exact functors of
infinite direct sum} if all the infinite direct sums (coproducts)
exist in $\sE$ and the infinite direct sums of admissible short
exact sequences are admissible short exact sequences.
 The notion of an exact category with \emph{exact functors of infinite
product} is defined dually.

 Let $\sE$ be an exact category with exact functors of infinite direct
sum.
 A complex in $\sE$ is said to be \emph{Positselski-coacyclic} if it
belongs to the minimal triangulated subcategory of $\sK(\sE)$
containing the totalizations of short exact sequences of complexes in
$\sE$ and closed under infinite direct sums.
 The full subcategory of Positselski-coacyclic complexes in $\sE$ is
denoted by $\Ac^\co(\sE)\subset\sK(\sE)$.
 The \emph{Positselski coderived category} of $\sE$ is defined as
the triangulated Verdier quotient category
$$
 \sD^\co(\sE)=\sK(\sE)/\Ac^\co(\sE)
$$
\cite[Section~2.1]{Psemi}, \cite[Section~A.1]{Pcosh},
\cite[Appendix~A]{Pmgm}.

 Dually, let $\sE$ be an exact category with exact functors of infinite
product.
 A complex in $\sE$ is said to be \emph{Positselski-contraacyclic} if
it belongs to the minimal triangulated subcategory of $\sK(\sE)$
containing the totalizations of short exact sequences of complexes in
$\sE$ and closed under infinite products.
 The full subcategory of Positselski-contraacyclic complexes in $\sE$
is denoted by $\Ac^\ctr(\sE)\subset\sK(\sE)$.
 The \emph{Positselski contraderived category} of $\sE$ is defined as
the triangulated Verdier quotient category
$$
 \sD^\ctr(\sE)=\sK(\sE)/\Ac^\ctr(\sE)
$$
\cite[Section~4.1]{Psemi}, \cite[Section~A.1]{Pcosh},
\cite[Appendix~A]{Pmgm}.

 We refer to~\cite[Section~11]{Bueh} for the definitions of injective
and projective objects in exact categories.
 Given an exact category $\sE$, we denote by $\sE^\inj\subset\sE$
the full subcategory of injective objects in $\sE$ and by
$\sE^\proj\subset\sE$ the full subcategory of projective objects
in~$\sE$.

 A complex $A^\bu$ in an exact category $\sE$ is said to be
\emph{Becker-coacyclic}~\cite[Proposition~1.3.8(2)]{Bec},
\cite[Section~B.7]{Pcosh} if, for every complex of injective objects
$J^\bu$ in $\sE$, every morphism of complexes $A^\bu\rarrow J^\bu$ is
homotopic to zero.
 All absolutely acyclic complexes are Becker-coacyclic.
 If the exact category $\sE$ has exact functors of infinite direct
sum, then all Positselski-coacyclic complexes are
Becker-coacyclic~\cite[Lemma~B.7.1(a\+-b)]{Pcosh}.
 The \emph{Becker coderived category} of $\sE$ is defined as
the triangulated Verdier quotient category
$$
 \sD^\bco(\sE)=\sK(\sE)/\Ac^\bco(\sE).
$$

 Dually, a complex $B^\bu$ in an exact category $\sE$ is said to be
\emph{Becker-contraacyclic}~\cite[Proposition~1.3.8(1)]{Bec},
\cite[Section~B.7]{Pcosh} if, for every complex of projective objects
$P^\bu$ in $\sE$, every morphism of complexes $P^\bu\rarrow B^\bu$ is
homotopic to zero.
 All absolutely acyclic complexes are Becker-contraacyclic.
 If the exact category $\sE$ has exact functors of infinite product,
then all Positselski-contraacyclic complexes are
Becker-contraacyclic~\cite[Lemma~B.7.1(a,c)]{Pcosh}.
 The \emph{Becker contraderived category} of $\sE$ is defined as
the triangulated Verdier quotient category
$$
 \sD^\bctr(\sE)=\sK(\sE)/\Ac^\bctr(\sE).
$$

\begin{lem} \label{becker-co-contra-acyclic-are-acyclic}
\textup{(a)} Let\/ $\sE$ be an exact category with enough injective
objects.
 Assume that the cokernels of all morphisms exist in the additive
category\/~$\sE$.
 Then every Becker-coacyclic complex in\/ $\sE$ is acyclic. \par
\textup{(b)} Let\/ $\sE$ be an exact category with enough projective
objects.
 Assume that the kernels of all morphisms exist in the additive
category\/~$\sE$.
 Then every Becker-contraacyclic complex in\/ $\sE$ is acyclic.
\end{lem}

\begin{proof}
 This is~\cite[Lemma~B.7.3]{Pcosh}.
 See~\cite[Remark~B.7.4]{Pcosh} for a discussion.
\end{proof}

\begin{thm} \label{becker-co-contra-derived-of-loc-pres-abelian}
\textup{(a)} Let\/ $\sA$ be a Grothendieck category (viewed as
an exact category with the abelian exact structure).
 Then the inclusion of additive/abelian categories\/
$\sA^\inj\rarrow\sA$ induces an equivalence between the homotopy
category and the Becker coderived category,
$$
 \sK(\sA^\inj)\simeq\sD^\bco(\sA).
$$ \par
\textup{(b)} Let\/ $\sB$ be a locally presentable abelian category
with enough projective objects (viewed as an exact category with
the abelian exact structure).
 Then the inclusion of additive/abelian categories\/
$\sB^\proj\rarrow\sB$ induces an equivalence between the homotopy
category and the Becker contraderived category,
$$
 \sK(\sB^\proj)\simeq\sD^\bctr(\sB).
$$
\end{thm}

\begin{proof}
 Part~(a) is~\cite[Theorem~2.13]{Neem2}, \cite[Corollary~5.13]{Kra3},
\cite[Theorem~4.2]{Gil4}, or~\cite[Corollary~9.5]{PS4}.
 Part~(b) is~\cite[Corollary~7.4]{PS4}.
\end{proof}

\Section{Corollaries of the Derived Full-and-Faithfulness Theorems}
\label{cors-of-derived-fullyf-thms-secn}

 In this section we discuss some of the more advanced results from
the paper~\cite{Pmgm} and their corollaries.
 Firstly, let $J$ be an arbitrary finitely generated ideal in
a commutative ring~$S$.

\begin{lem} \label{compact-torsion-dualization-lemma}
 Let $K^\bu$ be a finite complex of finitely generated projective
$S$\+modules with $J$\+torsion cohomology modules.
 Then the finite complex of finitely generated projective $S$\+modules\/
$\Hom_S(K^\bu,S)$ also has $J$\+torsion cohomology modules.
\end{lem}

\begin{proof}
 This is~\cite[Lemma~5.4(a)]{Pmgm}.
\end{proof}

\begin{lem} \label{quis-tested-on-tensor-hom-with-compact-torsion}
\textup{(a)} Let $M^\bu$ be a complex of $S$\+modules with $J$\+torsion
cohomology modules.
 Assume that, for every finite complex of finitely generated projective
$S$\+modules $K^\bu$ with $J$\+torsion cohomology modules, the complex
$K^\bu\ot_SM^\bu$ is acyclic.
 Then the complex $M^\bu$ is acyclic.  \par
\textup{(b)} Let $P^\bu$ be a complex of $S$\+modules with
$J$\+contramodule cohomology modules.
 Assume that, for every finite complex of finitely generated projective
$S$\+modules $K^\bu$ with $J$\+torsion cohomology modules, the complex\/
$\Hom_S(K^\bu,P^\bu)$ is acyclic.
 Then the complex $P^\bu$ is acyclic. 
\end{lem}

\begin{proof}
 Part~(a) follows from~\cite[Lemma~1.1(c)]{Pmgm}; cf.~\cite[proof of
Proposition~5.1]{Pmgm}.
 Part~(b) similarly follows from~\cite[Lemma~2.2(c)]{Pmgm}.
\end{proof}

 Given two complexes of $S$\+modules $M^\bu$ and $N^\bu$, we use
the simplified notation
$$
 \Ext_S^n(M^\bu,N^\bu)=H^n\boR\Hom_S(M^\bu,N^\bu)=
 \Hom_{\sD(S\Modl)}(M^\bu,N^\bu[n])
$$
and
$$
 \Tor^S_n(M^\bu,N^\bu)=H^{-n}(N^\bu\ot_S^\boL M^\bu),
 \qquad n\in\boZ,
$$
where $\boR\Hom_S({-},{-})$ and ${-}\ot_S^\boL{-}$ are the usual
derived functors of $\Hom$ and tensor product of unbounded complexes
of $S$\+modules, acting on the conventional derived category
$\sD(S\Modl)$ and constructed in terms of homotopy injective,
homotopy projective, and/or homotopy flat resolutions.

 The next lemma is a generalization of
Lemma~\ref{tors-contra-tensor-hom-lemma}.

\begin{lem} \label{tors-contra-tensor-hom-for-complexes}
 Let $S$ be a commutative ring and $J\subset S$ be an ideal.
 In this context: \par
\textup{(a)} if $M^\bu$ and $N^\bu$ are complexes of $S$\+modules,
and all the cohomology $S$\+modules of the complex $M^\bu$ are
$J$\+torsion, then all the cohomology $S$\+modules of the complex
$M^\bu\ot_S^\boL N^\bu$ are also $J$\+torsion; \par
\textup{(b)} if $M^\bu$ and $P^\bu$ are complexes of $S$\+modules,
and all the cohomology $S$\+modules of the complex $M^\bu$ are
$J$\+torsion, then all the cohomology $S$\+modules of the complex\/
$\boR\Hom_S(M^\bu,P^\bu)$ are $J$\+contramodules; \par
\textup{(c)} if $M^\bu$ and $P^\bu$ are complexes of $S$\+modules,
and all the cohomology $S$\+modules of the complex $P^\bu$ are
$J$\+contramodules, then all the cohomology $S$\+modules of
the complex\/ $\boR\Hom_S(M^\bu,P^\bu)$ are also $J$\+contramodules.
\end{lem}

\begin{proof}
 This is~\cite[Lemma~6.2]{Pcta}.
\end{proof}

 The utility of the conventional module-theoretic derived functors
of $\Hom$ and tensor product as above in the context involving
$J$\+torsion and $J$\+contramodule $S$\+modules is largely based on
the following results of~\cite[Theorems~1.3 and~2.9]{Pmgm}.

\begin{thm} \label{torsion-modules-inclusion-derived-fully-faithful}
 Let $S$ be a commutative ring and $J\subset S$ be a weakly proregular
finitely generated ideal.
 Then, for any derived category symbol\/ $\st=\bb$, $+$, $-$,
$\varnothing$, $\abs+$, $\abs-$, $\co$, or\/~$\abs$, the exact
inclusion of abelian categories $S\Modl_{J\tors}\rarrow S\Modl$ induces
a fully faithful triangulated functor
\begin{equation} \label{torsion-modules-triangulated-inclusion-formula}
 \sD^\st(S\Modl_{J\tors})\lrarrow\sD^\st(S\Modl).
\end{equation}
 For any conventional derived category symbol\/ $\st=\bb$, $+$, $-$,
or\/~$\varnothing$, the essential image of
the functor~\eqref{torsion-modules-triangulated-inclusion-formula}
consists precisely of all the (respectively bounded or unbounded)
complexes of $S$\+modules with $J$\+torsion cohomology modules.
\end{thm}

\begin{proof}
 The first assertion is~\cite[Theorem~1.3]{Pmgm}, and the second one
is~\cite[Corollary~1.4]{Pmgm}.
 For the converse result, claiming that the ideal $J$ is weakly
proregular whenever the functor $\sD^\st(S\Modl_{J\tors})
\rarrow\sD^\st(S\Modl)$ is fully faithful, see~\cite[Theorem~4.1]{Pdc}.
\end{proof}

\begin{thm} \label{contramodules-inclusion-derived-fully-faithful}
 Let $S$ be a commutative ring and $J\subset S$ be a weakly proregular
finitely generated ideal.
 Then, for any derived category symbol\/ $\st=\bb$, $+$, $-$,
$\varnothing$, $\abs+$, $\abs-$, $\ctr$, or\/~$\abs$, the exact
inclusion of abelian categories $S\Modl_{J\ctra}\rarrow S\Modl$ induces
a fully faithful triangulated functor
\begin{equation} \label{contramodules-triangulated-inclusion-formula}
 \sD^\st(S\Modl_{J\ctra})\lrarrow\sD^\st(S\Modl).
\end{equation}
 For any conventional derived category symbol\/ $\st=\bb$, $+$, $-$,
or\/~$\varnothing$, the essential image of
the functor~\eqref{contramodules-triangulated-inclusion-formula}
consists precisely of all the (respectively bounded or unbounded)
complexes of $S$\+modules with $J$\+contramodule cohomology modules.
\end{thm}

\begin{proof}
 The first assertion is~\cite[Theorem~2.9]{Pmgm}, and the second one
is~\cite[Corollary~2.10]{Pmgm}.
 In fact, a weaker assumption than the weak proregularity of the ideal
$J$ is sufficient for the validity of these assertions;
see~\cite[Remark~3.8 and Theorem~4.3]{Pdc}.
 Notice that the weak proregularity of $J$ also implies that all
the $J$\+contramodule $S$\+modules are quotseparated, as per
the discussion in Section~\ref{prelims-on-wpr-secn}.
 According to~\cite[Theorem~4.2]{Pdc}, any one of the functors
$\sD^\st(S\Modl_{J\ctra}^\qs)\rarrow\sD^\st(S\Modl)$ is fully faithful
if and only if the ideal $J$ is weakly proregular.
\end{proof}

\begin{lem} \label{Hom-tensor-underived=derived-lemma}
 Let $J$ be a weakly proregular finitely generated ideal in
a commutative ring~$S$.
 In this context: \par
\textup{(a)} Let $N^\bu$ be a complex of $J$\+torsion $S$\+modules
and $H^\bu$ be a bounded below complex of injective objects in
the abelian category of $J$\+torsion $S$\+modules $S\Modl_{J\tors}$.
 Then the complex of $S$\+modules\/ $\Hom_S(N^\bu,H^\bu)$ represents
the derived category object\/ $\boR\Hom_S(N^\bu,H^\bu)$.
 In other words, the natural morphism
$$
 \Hom_S(N^\bu,H^\bu)\lrarrow\boR\Hom_S(N^\bu,H^\bu)
$$
is an isomorphism in\/ $\sD(S\Modl)$. \par
\textup{(b)} Let $Q^\bu$ be a complex of $J$\+contramodule $S$\+modules
and $P^\bu$ be a bounded above complex of projective objects in
the abelian category of $J$\+contramodule $S$\+modules
$S\Modl_{J\ctra}$.
 Then the complex of $S$\+modules\/ $\Hom_S(P^\bu,Q^\bu)$ represents
the derived category object\/ $\boR\Hom_S(P^\bu,Q^\bu)$.
 In other words, the natural morphism
$$
 \Hom_S(P^\bu,Q^\bu)\lrarrow\boR\Hom_S(P^\bu,Q^\bu)
$$
is an isomorphism in\/ $\sD(S\Modl)$. \par
\textup{(c)} Let $N^\bu$ be a complex of $J$\+torsion $S$\+modules
and $P^\bu$ be a bounded above complex of projective objects in
the abelian category of $J$\+contramodule $S$\+modules.
 Then the complex of $S$\+modules\/ $N^\bu\ot_S P^\bu$ represents
the derived category object $N^\bu\ot_S^\boL P^\bu$.
 In other words, the natural morphism
$$
 N^\bu\ot_S^\boL P^\bu \lrarrow N^\bu\ot_S P^\bu.
$$
is an isomorphism in\/ $\sD(S\Modl)$.
\end{lem}

\begin{proof}
 Part~(a), which is a generalization of~\cite[Lemma~5.5(b)]{Pmgm},
follows from the first assertion of
Theorem~\ref{torsion-modules-inclusion-derived-fully-faithful}
(for $\st=\varnothing$).
 Part~(b) similarly follows from the first assertion of
Theorem~\ref{contramodules-inclusion-derived-fully-faithful}
(for $\st=\varnothing$).
 Part~(c), which is a generalization of~\cite[Lemma~5.4(c)]{Pmgm},
is deduced from part~(b) in the following way.
 The conservative contravariant triangulated functor
$\Hom_\boZ({-},\boQ/\boZ)\:\sD(S\Modl)^\sop\rarrow\sD(S\Modl)$
transforms the derived category morphism in question into the morphism
$$
 \Hom_S(P^\bu,\Hom_\boZ(N^\bu,\boQ/\boZ))\lrarrow
 \boR\Hom_S(P^\bu,\Hom_\boZ(N^\bu,\boQ/\boZ)),
$$
which is an isomorphism by part~(b).
\end{proof}

 Let $J$ be a finitely generated ideal in a commutative ring~$S$.
 A $J$\+contramodule $S$\+module $F$ is said to be \emph{contraflat}
if the functor ${-}\ot_S\nobreak F\:\allowbreak S\Modl_{J\tors}\rarrow
S\Modl_{J\tors}$ is exact.
 One can easily see that a $J$\+contramodule $S$\+module $F$ is
contraflat if and only if the $S/J^n$\+module $F/J^nF$ is flat for
every $n\ge1$.
 Since the functor $\Hom_\boZ({-},\boQ/\boZ)\:S\Modl^\sop\rarrow S\Modl$
is exact and faithful, and takes $S\Modl_{J\tors}$ to
$S\Modl_{J\ctra}^\qs\subset S\Modl_{J\ctra}$ (see
Lemma~\ref{tors-contra-tensor-hom-lemma}(b)), the natural isomorphism
$\Hom_\boZ(M\ot_SP,\>\boQ/\boZ)\simeq\Hom_S(P,\Hom_\boZ(M,\boQ/\boZ))$
implies that all projective objects of the abelian category
$S\Modl_{J\ctra}^\qs$, as well as all projective objects of the abelian
category $S\Modl_{J\ctra}$, are contraflat.
 Denote the class of contraflat $J$\+contramodule $S$\+modules
by $S\Modl_{J\ctra}^\ctrfl\subset S\Modl_{J\ctra}$.

\begin{lem} \label{contraflat-contramodules-lemma}
 Let $J$ be a weakly proregular finitely generated ideal in
a commutative ring~$S$.
 In this context: \par
\textup{(a)} The class of contraflat $J$\+contramodule $S$\+modules is
closed under extensions and kernels of surjective morphisms in
$S\Modl_{J\ctra}$.
 For any $J$\+torsion $S$\+module $M$, the functor $M\ot_S{-}$
preserves exactness of short exact sequences of contraflat
$J$\+contramodule $S$\+modules. \par
\textup{(b)} One has\/ $\Tor^S_n(M,F)=0$ for any $J$\+torsion
$S$\+module $M$, any contraflat $J$\+contramodule $S$\+module $F$,
and all $n\ge1$.
\end{lem}

\begin{proof}
 Part~(a) can be obtained as a special case of~\cite[Lemma~8.4]{Pflcc},
which is applicable in view of~\cite[Proposition~1.5 and
Corollary~3.7]{Pdc}.
 This argument shows that a weaker assumption than the weak
proregularity of the ideal $J$ is sufficient for the validity
of part~(a); see~\cite[Remark~3.8]{Pdc}.
 Without the weak proregularity assumption, part~(a) holds for
quotseparated $J$\+contramodule $S$\+modules.

 Part~(b) is essentially a result of Yekutieli;
see~\cite[Theorem~1.6(1) or~6.9]{Yek2}.
 The validity of part~(b) is equivalent to the weak proregularity of
the ideal~$J$; see~\cite[Theorem~7.2]{Pdc}.

 It is easy to deduce part~(b) from the combination of part~(a)
and Lemma~\ref{Hom-tensor-underived=derived-lemma}(c).
 Indeed, let $P_\bu$ be a projective resolution of a contraflat
$J$\+contramodule $F$ in the abelian category $S\Modl_{J\ctra}$.
 Then it is clear from part~(a) that the complex $M\ot_SP_\bu\rarrow
M\ot_SF\rarrow0$ is acyclic.
 On the other hand, by Lemma~\ref{Hom-tensor-underived=derived-lemma}(c)
we have $\Tor^S_n(M,P)=0$ for all projective objects
$P\in S\Modl_{J\ctra}$ and all $n\ge1$; so the complex
$M\ot_SP_\bu$ computes the derived functor $\Tor^S_*(M,F)$.
\end{proof}

 The following corollary is a further generalization of
Lemma~\ref{Hom-tensor-underived=derived-lemma}(c).

\begin{cor} \label{contraflat-contramods-tensor-underived=derived-cor}
 Let $J$ be a weakly proregular finitely generated ideal in
a commutative ring~$S$.
 Let $N^\bu$ be a complex of $J$\+torsion $S$\+modules
and $F^\bu$ be a bounded above complex of contraflat $J$\+contramodule
$S$\+modules.
 Then the complex of $S$\+modules\/ $N^\bu\ot_S F^\bu$ represents
the derived category object $N^\bu\ot_S^\boL F^\bu$.
 In other words, the natural morphism
$$
 N^\bu\ot_S^\boL F^\bu \lrarrow N^\bu\ot_S F^\bu.
$$
is an isomorphism in\/ $\sD(S\Modl)$.
\end{cor}

\begin{proof}
 Let $P^\bu$ be a bounded above complex of projective objects in
$S\Modl_{J\ctra}$ endowed with a quasi-isomorphism of complexes
$P^\bu\rarrow F^\bu$.
 In view of Lemma~\ref{Hom-tensor-underived=derived-lemma}(c), we
only need to prove that the induced map of complexes
$N^\bu\ot_SP^\bu\rarrow N^\bu\ot_S F^\bu$ is a quasi-isomorphism.
 Indeed, denote by $G^\bu$ the cone of the morphism of complexes
$P^\bu\rarrow F^\bu$.
 So $G^\bu$ is a bounded above acyclic complex of contraflat
$J$\+contramodule $S$\+modules.
 By Lemma~\ref{contraflat-contramodules-lemma}(a), the complex
$M\ot_R G^\bu$ is acyclic for any $J$\+torsion $S$\+module~$M$.
 It follows that the complex $M^\bu\ot_R G^\bu$ is acyclic for any
finite complex of $J$\+torsion $S$\+modules~$M^\bu$.
 It remains to represent the given complex of $J$\+torsion $S$\+modules
$N^\bu$ as a direct limit of finite complexes of $J$\+torsion
$S$\+modules, which can be done using the canonical truncations on
one side and the silly truncations on the other side, in order to
prove that the complex $N^\bu\ot_R G^\bu$ is acyclic.
\end{proof}

\Section{Auslander and Bass classes}  \label{auslander-and-bass-secn}

 Let $S$ be a commutative ring and $J\subset S$ be a weakly proregular
finitely generated ideal.
 Denote by $\fS=\varprojlim_{n\ge0}S/J^n$ the $J$\+adic completion of
the ring~$S$.

 A \emph{pseudo-dualizing complex of $J$\+torsion $S$\+modules} $L^\bu$
is a finite complex of $J$\+torsion $S$\+modules satisfying
the following two conditions:
\begin{enumerate}
\renewcommand{\theenumi}{\roman{enumi}}
\setcounter{enumi}{1}
\item for every finite complex of finitely generated projective
$S$\+modules $K^\bu$ with $J$\+torsion cohomology modules, the complex
of $S$\+modules $\Hom_S(K^\bu,L^\bu)$ is quasi-isomorphic to a bounded
above complex of finitely generated projective $S$\+modules;
\item the homothety map
$\fS\rarrow\Hom_{\sD^\bb(S\Modl)}(L^\bu,L^\bu[*])$ is an isomorphism
of graded rings.
\end{enumerate}

 Assume that the finite complex $L^\bu$ is concentrated in
the cohomological degrees $-d_1\le m\le d_2$.
 Choose an integer $l_1\ge d_1$, and consider the following full
subcategories in the abelian categories of $J$\+torsion and
$J$\+contramodule $S$\+modules:
\begin{itemize}
\item $\sE_{l_1}=\sE_{l_1}(L^\bu)\subset S\Modl_{J\tors}$ is
the full subcategory consisting of all the $J$\+torsion $S$\+modules
$E$ such that $\Ext_S^n(L^\bu,E)=0$ for all $n>l_1$ and the adjunction
morphism $L^\bu\ot_S^\boL\boR\Hom_S(L^\bu,E)\rarrow E$ is
an isomorphism in $\sD^-(S\Modl)$;
\item $\sF_{l_1}=\sF_{l_1}(L^\bu)\subset S\Modl_{J\ctra}$ is
the full subcategory consisting of all the $J$\+contramodule
$S$\+modules $F$ such that $\Tor^S_n(L^\bu,F)=0$ for all $n>l_1$ and
the adjunction morphism $F\rarrow\boR\Hom_S(L^\bu,\>L^\bu\ot_S^\boL F)$
is an isomorphism in $\sD^+(S\Modl)$.
\end{itemize}
 Clearly, for any $l_1''\ge l_1'\ge d_1$, one has $\sE_{l_1'}\subset
\sE_{l_1''}\subset S\Modl_{J\tors}$ and $\sF_{l_1'}\subset\sF_{l_1''}
\subset S\Modl_{J\ctra}$.
 The category $\sF_{l_1}$ can be called the \emph{Auslander class of
$J$\+contramodule $S$\+modules} corresponding to a pseudo-dualizing
complex $L^\bu$, while the category $\sE_{l_1}$ is the \emph{Bass class
of $J$\+torsion $S$\+modules} (cf.~\cite[Section~3]{Pps}
and~\cite[Section~4]{Ppc}).

 Given an exact category $\sT$, a full subcategory $\sE\subset\sT$ is
said to be \emph{coresolving} if $\sE$ is closed under extensions and
cokernels of admissible monomorphisms in $\sT$, and for every object
$T\in\sT$ there exists an admissible monomorphism $T\rarrow E$ in $\sT$
with $E\in\sE$.
 Dually, a full subcategory $\sF\subset\sT$ is said to be
\emph{resolving} if $\sF$ is closed under extensions and kernels of
admissible epimorphisms in $\sT$, and for every object $T\in\sT$
there exists an admissible epimorphism $F\rarrow T$ in $\sT$ with
$F\in\sF$.
 The following two lemmas imply that the full subcategory $\sE_{l_1}$
is coresolving in the abelian category $S\Modl_{J\tors}$, while
the full subcategory $\sF_{l_1}$ is resolving in the abelian category
$S\Modl_{J\ctra}$.

\begin{lem} \label{auslander-bass-classes-closure-properties-lemma}
\textup{(a)} The full subcategory\/ $\sE_{l_1}\subset S\Modl_{J\tors}$
is closed under the cokernels of injective morphisms, extensions, and
direct summands. \par
\textup{(b)} The full subcategory\/ $\sF_{l_1}\subset S\Modl_{J\ctra}$
is closed under the kernels of surjective morphisms, extensions, and
direct summands.  \qed
\end{lem}

 The next lemma, which is our version of~\cite[Lemma~3.2]{Pps}
and~\cite[Lemma~4.2]{Ppc}, plays a key role.

\begin{lem} \label{injectives-projectives-belong-to-auslander-bass}
\textup{(a)} The full subcategory\/ $\sE_{l_1}\subset S\Modl_{J\tors}$
contains all the injective objects of the abelian category
$S\Modl_{J\tors}$. \par
\textup{(b)} The full subcategory\/ $\sF_{l_1}\subset S\Modl_{J\ctra}$
contains all the contraflat $J$\+con\-tra\-mod\-ule $S$\+modules.
 In particular, all the projective objects of the abelian category
$S\Modl_{J\ctra}$ belong to\/~$\sF_{l_1}$.
\end{lem}

\begin{proof}
 Part~(a): let $H$ be an injective object of $S\Modl_{J\tors}$.
 Then, first of all, one has
$$
 \Ext^n_S(L^\bu,H)=H^n\Hom_S(L^\bu,H)=0
$$
for all $n\ge d_1$ by
Lemma~\ref{Hom-tensor-underived=derived-lemma}(a).
 It remains to check that the adjunction morphism
$L^\bu\ot_S^\boL\boR\Hom_S(L^\bu,H)\rarrow H$ is an isomorphism
in $\sD(S\Modl)$.

 Indeed, both $L^\bu\ot_S^\boL\boR\Hom_S(L^\bu,H)$ and $H$ are
complexes of $S$\+modules with $J$\+torsion cohomology modules
(see Lemma~\ref{tors-contra-tensor-hom-for-complexes}(a)).
 Let $K^\bu$ be a finite complex of finitely generated projective
$S$\+modules with $J$\+torsion cohomology modules.
 By Lemma~\ref{quis-tested-on-tensor-hom-with-compact-torsion}(a),
it suffices to check that the morphism of complexes
$$
 (K^\bu\ot_SL^\bu)\ot_S^\boL\boR\Hom_S(L^\bu,H)
 \lrarrow K^\bu\ot_SH
$$
is a quasi-isomorphism.

 By Lemma~\ref{compact-torsion-dualization-lemma} and condition~(ii),
there exists a bounded above complex of finitely generated projective
$S$\+modules $M^\bu$ together with a quasi-isomorphism of complexes
of $S$\+modules $M^\bu\rarrow K^\bu\ot_SL^\bu$.
 For every complex of $J$\+torsion $S$\+modules $N^\bu$, the complex of
$S$\+modules $\Hom_S(N^\bu,H)$ represents the derived category object
$\boR\Hom_S(N^\bu,H)$ by
Lemma~\ref{Hom-tensor-underived=derived-lemma}(a).
 So we have a natural isomorphism
\begin{multline*}
 (K^\bu\ot_SL^\bu)\ot_S^\boL\boR\Hom_S(L^\bu,H)
 = M^\bu\ot_S\Hom_S(L^\bu,H) \\
 \simeq \Hom_S(\Hom_S(M^\bu,L^\bu),H)
 = \boR\Hom_S(\boR\Hom_S(K^\bu\ot_S L^\bu,\>L^\bu),\>H)
\end{multline*}
in the derived category $\sD(S\Modl)$.
 Here we are using the fact that $\Hom_S(M^\bu,L^\bu)\simeq
\Hom_S(M^\bu,S)\ot_SL^\bu$ is a complex of $J$\+torsion $S$\+modules
by Lemma~\ref{tors-contra-tensor-hom-lemma}(a).

 By condition~(iii), the homothety map
$$
 \Hom_S(K^\bu,\fS)\lrarrow\boR\Hom_S(K^\bu\ot_S L^\bu,\>L^\bu)
$$
is an isomorphism in $\sD(S\Modl)$.
 It remains to point out that the map
$$
 \Hom_S(K^\bu,S)\lrarrow\Hom_S(K^\bu,\fS)
$$
induced by the completion map $S\rarrow\fS$ is a quasi-isomorphism
of complexes of $S$\+modules by
Lemma~\ref{compact-torsion-dualization-lemma}
and~\cite[Lemma~5.4(b)]{Pmgm}.

 Part~(b): let $P$ be a contraflat $J$\+contramodule $S$\+module.
 Then, first of all, one has
$$
 \Tor_n^S(L^\bu,P)=H^{-n}(L^\bu\ot_SP)=0
$$
for all $n\ge d_1$ by
Corollary~\ref{contraflat-contramods-tensor-underived=derived-cor}.
 It remains to check that the adjunction morphism
$P\rarrow\boR\Hom_S(L^\bu,\>L^\bu\ot_S^\boL P)$ is an isomorphism
in $\sD(S\Modl)$.

 Indeed, both $P$ and $\boR\Hom_S(L^\bu,\>L^\bu\ot_S^\boL P)$ are
complexes of $S$\+modules with $J$\+contramodule cohomology modules
(see Lemma~\ref{tors-contra-tensor-hom-for-complexes}(b)).
 Let $K^\bu$ be a finite complex of finitely generated projective
$S$\+modules with $J$\+torsion cohomology modules.
 By Lemma~\ref{quis-tested-on-tensor-hom-with-compact-torsion}(b),
it suffices to check that the morphism of complexes
$$
 \Hom_S(K^\bu,P)\lrarrow\boR\Hom_S(K^\bu\ot_SL^\bu,\>L^\bu\ot_S^\boL P)
$$
is a quasi-isomorphism.

 As in part~(a), we use Lemma~\ref{compact-torsion-dualization-lemma}
and condition~(ii), and pick a bounded above complex of finitely
generated projective $S$\+modules $M^\bu$ together with
a quasi-isomorphism of complexes of $S$\+modules
$M^\bu\rarrow K^\bu\ot_SL^\bu$.
 For every complex of $J$\+torsion $S$\+modules $N^\bu$, the complex of
$S$\+modules $N^\bu\ot_S P$ represents the derived category object
$N^\bu\ot_S^\boL P$ by
Corollary~\ref{contraflat-contramods-tensor-underived=derived-cor}.
 So we have a natural isomorphism
\begin{multline*}
 \boR\Hom_S(K^\bu\ot_SL^\bu,\>L^\bu\ot_S^\boL P)=
 \Hom_S(M^\bu,\>L^\bu\ot_S P) \\
 \simeq\Hom_S(M^\bu,L^\bu)\ot_S P =
 \boR\Hom_S(K^\bu\ot_SL^\bu,\>L^\bu)\ot_S^\boL P
\end{multline*}
in the derived category $\sD(S\Modl)$.
 Once again, we are using the fact that $\Hom_S(M^\bu,L^\bu)$ is
a complex of $J$\+torsion $S$\+modules.
 As in part~(a), the argument finishes with the observations that
the maps
$$
 \Hom_S(K^\bu,S)\lrarrow\Hom_S(K^\bu,\fS)
 \lrarrow\boR\Hom_S(K^\bu\ot_S L^\bu,\>L^\bu)
$$
are isomorphisms in $\sD(S\Modl)$ by~\cite[Lemma~5.4(b)]{Pmgm}
and condition~(iii).
\end{proof}

\begin{rem} \label{Bass-Auslander-direct-sum-product-closedness-remark}
 Similarly to~\cite[Remark~4.3]{Ppc}, we \emph{do not know} whether
the analogue of~\cite[Lemma~3.3]{Pps} holds in the context of
pseudo-dualizing complexes of $J$\+torsion $S$\+modules, i.~e.,
whether the Bass class of $J$\+torsion $S$\+modules $\sE_{l_1}$ is
always closed under infinite direct sums in $S\Modl_{J\tors}$, and
whether the Auslander class of $J$\+contramodule $S$\+modules
$\sF_{l_1}$ is always closed under infinite direct products
in $S\Modl_{J\ctra}$.
 These questions are open even in the case of a Noetherian ring $S$,
when the class of injective $J$\+torsion $S$\+modules is closed under
infinite direct sums (since an object of $S\Modl_{J\tors}$ is injective
in $S\Modl_{J\tors}$ if and only if it is injective in $S\Modl$, as
one can see from the Artin--Rees lemma) and the class of contraflat
$J$\+contramodule $S$\+modules is closed under infinite products (since
a $J$\+contramodule $S$\+module is contraflat if and only if it is flat
as an $S$\+module; see~\cite[Corollary~10.3(a)]{Pcta}).
\end{rem}

\begin{lem} \label{finite-complexes-in-Bass-Auslander-classes}
\textup{(a)} Let $M^\bu$ be a complex of $J$\+torsion $S$\+modules
concentrated in the cohomological degrees $-n_1\le m\le n_2$.
 Then $M^\bu$ is quasi-isomorphic to a complex of $J$\+torsion
$S$\+modules concentrated in the cohomological degrees
$-n_1\le m\le n_2$ with the terms belonging to the full subcategory\/
$\sE_{l_1}\subset S\Modl_{J\tors}$ if and only if\/
$\Ext_S^n(L^\bu,M^\bu)=0$ for all $n>n_2+l_1$ and the adjunction
morphism $L^\bu\ot_S^\boL\boR\Hom_S(L^\bu,M^\bu)\rarrow M^\bu$
is an isomorphism in\/ $\sD^-(S\Modl)$. \par
\textup{(b)} Let $Q^\bu$ be a complex of $J$\+contramodule $S$\+modules
concentrated in the cohomological degrees $-n_1\le m\le n_2$.
 Then $Q^\bu$ is quasi-isomorphic to a complex of $J$\+contramodule
$S$\+modules concentrated in the cohomological degrees
$-n_1\le m\le n_2$ with the terms belonging to the full subcategory\/
$\sF_{l_1}\subset S\Modl_{J\ctra}$ if and only if\/
$\Tor^S_n(L^\bu,Q^\bu)=0$ for $n>n_1+l_1$ and the adjunction morphism
$Q^\bu\rarrow\boR\Hom_S(L^\bu,\>L^\bu\ot_S^\boL Q^\bu)$ is
an isomorphism in\/ $\sD^+(S\Modl)$.
\end{lem}

\begin{proof}
 Part~(a): The ``only if'' implication is obvious.
 To prove the ``if'', replace $M^\bu$ by a quasi-isomorphic complex
${}'\!M^\bu$ in $S\Modl_{J\tors}$ concentrated in the same cohomological
degrees $-n_1\le m\le n_2$ such that ${}'\!M^m$ is an injective object 
of $S\Modl_{J\tors}$ for all $-n_1\le m<n_2$.
 Then use Lemma~\ref{injectives-projectives-belong-to-auslander-bass}(a)
in order to check that ${}'\!M^m\in\sE_{l_1}$ for all
$-n_1\le m\le n_2$.
 Part~(b): to prove the ``if'', replace $Q^\bu$ by a quasi-isomorphic
complex ${}'Q^\bu$ in $S\Modl_{J\ctra}$ concentrated in the same 
cohomological degrees $-n_1\le m\le n_2$ such that ${}'Q^m$ is
a projective object of $S\Modl_{J\ctra}$ for all $-n_1<m\le n_2$.
 Then use Lemma~\ref{injectives-projectives-belong-to-auslander-bass}(b)
in order to check that ${}'Q^m\in\sF_{l_1}$ for all $-n_1\le m\le n_2$.
\end{proof}

 It follows from
Lemma~\ref{finite-complexes-in-Bass-Auslander-classes}(a) that
the full subcategory $\sD^\bb(\sE_{l_1})\subset\sD(S\Modl_{J\tors})$
consists of all the complexes of $J$\+torsion $S$\+modules $M^\bu$
with bounded cohomology such that the complex $\boR\Hom_S(L^\bu,M^\bu)$
also has bounded cohomology and the adjunction morphism
$L^\bu\ot_S^\boL\boR\Hom_S(L^\bu,M^\bu)\rarrow M^\bu$ is an isomorphism.
 Similarly, by
Lemma~\ref{finite-complexes-in-Bass-Auslander-classes}(b), the full
subcategory $\sD^\bb(\sF_{l_1})\subset\sD(S\Modl_{J\ctra})$
consists of all the complexes of $J$\+contramodule $S$\+modules
$Q^\bu$ with bounded cohomology such that the complex
$L^\bu\ot_S^\boL Q^\bu$ also has bounded cohomology and the adjunction
morphism $Q^\bu\rarrow\boR\Hom_S(L^\bu,\>L^\bu\ot_S^\boL Q^\bu)$ is
an isomorphism.

 These two full subcategories can be called the \emph{derived Bass
class of $J$\+torsion $S$\+modules} and the \emph{derived Auslander
class of $J$\+contramodule $S$\+modules}.
 Any pair of adjoint functors between two categories restricts to
an equivalence between the full subcategories of objects whose
adjunction morphisms are isomorphisms~\cite[Theorem~1.1]{FJ}
(see also~\cite[Proposition~2.1]{GLT}); so the functors
$\boR\Hom_S(L^\bu,{-})$ and $L^\bu\ot_S^\boL{-}$ restrict to
a triangulated equivalence between the derived Bass and Auslander
classes
\begin{equation} \label{bounded-derived-bass-auslander-equivalence}
 \sD^\bb(\sE_{l_1})\simeq\sD^\bb(\sF_{l_1}).
\end{equation}

\begin{lem} \label{Hom-tensor-take-bass-auslander-to-complexes-of}
\textup{(a)} For any $J$\+torsion $S$\+module $E\in\sE_{l_1}$,
the object\/ $\boR\Hom_S(L^\bu,E)\in\sD^\bb(S\Modl)$ can be represented
by a complex of $J$\+contramodule $S$\+modules concentrated in
the cohomological degrees $-d_2\le m\le l_1$ with the terms belonging
to\/~$\sF_{l_1}$. \par
\textup{(b)} For any $J$\+contramodule $S$\+module $F\in\sF_{l_1}$,
the object $L^\bu\ot_S^\boL F\in\sD^\bb(S\Modl)$ can be represented by
a complex of $J$\+torsion $S$\+modules concentrated in
the cohomological degrees $-l_1\le m\le d_2$ with the terms belonging
to\/~$\sE_{l_1}$.
\end{lem}

\begin{proof}
 Part~(a) follows from
Lemma~\ref{finite-complexes-in-Bass-Auslander-classes}(b), as
the derived category object $L^\bu\ot_S^\boL\boR\Hom_S(L^\bu,E)\simeq E
\in\sD(S\Modl)$ has no cohomology in the cohomological degrees
$-n<-d_2-l_1$ (since $-d_2-l_1\le-d_2-d_1\le0$).
 Part~(b) follows from
Lemma~\ref{finite-complexes-in-Bass-Auslander-classes}(a), as
the derived category object $\boR\Hom_S(L^\bu,\>L^\bu\ot_S^\boL F)
\simeq F\in\sD(S\Modl)$ has no cohomology in the cohomological
degrees $n>d_2+l_1$ (since $d_2+l_1\ge d_2+d_1\ge0$).
\end{proof}

 Let $\sT$ be a weakly idempotent-complete exact category (in
the sense of~\cite[Section~7]{Bueh}), $\sE\subset\sT$ be a coresolving
subcategory, and $\sF\subset\sT$ be a resolving subcategory.
 We refer to~\cite[Section~2]{Sto} or~\cite[Section~A.5]{Pcosh} for
a discussion of the \emph{$\sE$\+coresolution dimensions} and
the \emph{$\sF$\+resolution dimensions} of the objects of~$\sT$.
 The key point is that the (co)resolution dimension does not depend
on the choice of a (co)resolution~\cite[Lemma~2.1]{Zhu},
\cite[Proposition~2.3(1)]{Sto}, \cite[Corollary~A.5.2]{Pcosh}.

\begin{lem} \label{bass-auslander-co-resolution-dimension}
\textup{(a)} For any integers $l_1''\ge l_1'\ge d_1$, the full
subcategory\/ $\sE_{l_1''}\subset S\Modl_{J\tors}$ consists precisely
of all the $J$\+torsion $S$\+modules whose\/ $\sE_{l_1'}$\+coresolution
dimension does not exceed $l_1''-l_1'$. \par
\textup{(b)} For any integers $l_1''\ge l_1'\ge d_1$, the full
subcategory\/ $\sF_{l_1''}\subset S\Modl_{J\ctra}$ consists precisely of
all the $J$\+contramodule $S$\+modules whose\/ $\sF_{l_1'}$\+resolution
dimension does not exceed $l_1''-l_1'$.
\end{lem}

\begin{proof}
 Part~(a) follows from
Lemma~\ref{finite-complexes-in-Bass-Auslander-classes}(a) applied
to a one-term complex of $J$\+torsion $S$\+modules $M^\bu=E$,
concentrated in the cohomological degree~$0$, with the numerical
parameters $n_1=0$, \ $n_2=l_1''-l_1'$, and $l_1=l_1'$.
 Part~(b) similarly follows from
Lemma~\ref{finite-complexes-in-Bass-Auslander-classes}(b) applied
to a one-term complex of $J$\+contramodule $S$\+modules $Q^\bu=F$,
concentrated in the cohomological degree~$0$, with the numerical
parameters $n_2=0$, \ $n_1=l_1''-l_1'$, and $l_1=l_1'$.
\end{proof}

\begin{rem}
 It is clear from
Lemmas~\ref{injectives-projectives-belong-to-auslander-bass}
and~\ref{bass-auslander-co-resolution-dimension} that, for any integer
$n\ge0$, all the objects of injective dimension not exceeding~$n$ in
the abelian category $S\Modl_{J\tors}$ belong to $\sE_{d_1+n}$ and all
the $J$\+contramodule $S$\+modules of contraflat dimension not
exceeding~$n$ belong to~$\sF_{d_1+n}$.
 Here the \emph{contraflat dimension} of a $J$\+contramodule
$S$\+module is simply defined as the resolution dimension with
respect to the resolving subcategory of contraflat $J$\+contramodule
$S$\+modules in $S\Modl_{J\ctra}$.
 Clearly, the contraflat dimension of a $J$\+contramodule $S$\+module
never exceeds its projective dimension as a object of $S\Modl_{J\ctra}$.
\end{rem}

\begin{prop} \label{bass-auslander-derived-independence-of-l1}
\textup{(a)} For any integers $l_1''\ge l_1'\ge d_1$ and any
conventional or exotic derived category symbol\/ $\st=\bb$, $+$, $-$,
$\varnothing$, $\abs+$, $\abs-$, $\bco$, or\/~$\abs$, the exact
inclusion functor\/ $\sE_{l_1'}\rarrow\sE_{l_1''}$ induces
a triangulated equivalence
$$
 \sD^\st(\sE_{l_1'})\simeq\sD^\st(\sE_{l_1''}).
$$ \par
\textup{(b)} For any integers $l_1''\ge l_1'\ge d_1$ and any
conventional or exotic derived category symbol\/ $\st=\bb$, $+$, $-$,
$\varnothing$, $\abs+$, $\abs-$, $\bctr$, or\/~$\abs$, the exact
inclusion functor\/ $\sF_{l_1'}\rarrow\sF_{l_1''}$ induces
a triangulated equivalence
$$
 \sD^\st(\sF_{l_1'})\simeq\sD^\st(\sF_{l_1''}).
$$
\end{prop}

\begin{proof}
 Part~(b) follows from
Lemma~\ref{bass-auslander-co-resolution-dimension}(b) in view
of~\cite[Propositions~A.5.8 and~B.7.9]{Pcosh}.
 Part~(a) follows from
Lemma~\ref{bass-auslander-co-resolution-dimension}(a) in view
the dual versions of~\cite[Propositions~A.5.8 and~B.7.9]{Pcosh}.
\end{proof}

 The cases $\st=\bco$ and $\st=\bctr$ in the context of
Proposition~\ref{bass-auslander-derived-independence-of-l1} are
actually trivial, and are only included in the formulation for the sake
of completeness and for comparison with~\cite[Proposition~3.8]{Pps}.
 Using the results of~\cite[Corollary~9.5]{PS4} for
$\sA=S\Modl_{J\tors}$ and~\cite[Corollary~7.4]{PS4} for
$\sB=S\Modl_{J\ctra}$, one can easily show that
$\sD^\bco(\sE_{l_1'})\simeq\sD^\bco(\sE_{l_1''})\simeq
\sD^\bco(S\Modl_{J\tors})$ and
$\sD^\bctr(\sF_{l_1'})\simeq\sD^\bctr(\sF_{l_1''})\simeq
\sD^\bctr(S\Modl_{J\ctra})$.

 As a particular case of
Proposition~\ref{bass-auslander-derived-independence-of-l1},
the conventional unbounded derived category of the Bass class of
$J$\+torsion $S$\+modules $\sD(\sE_{l_1})$ is the same for all
$l_1\ge d_1$, and the conventional unbounded derived category of
the Auslander class of $J$\+contramodule $S$\+modules $\sD(\sF_{l_1})$
is the same for all $l_1\ge d_1$.
 Following the notation in~\cite[Section~3]{Pps}
and~\cite[Section~4]{Ppc}, we put
$$
 \sD'_{L^\bu}(S\Modl_{J\tors})=\sD(\sE_{l_1})
 \quad\text{and}\quad
 \sD''_{L^\bu}(S\Modl_{J\ctra})=\sD(\sF_{l_1}).
$$
 The next theorem, generalizing the triangulated
equivalence~\eqref{bounded-derived-bass-auslander-equivalence},
provides, in particular, a triangulated equivalence
$$
 \sD'_{L^\bu}(S\Modl_{J\tors})=\sD(\sE_{l_1})\simeq
 \sD(\sF_{l_1})=\sD''_{L^\bu}(S\Modl_{J\ctra}).
$$

\begin{thm} \label{bass-auslander-derived-equivalence}
 For any conventional or absolute derived category symbol\/
$\st=\bb$, $+$, $-$, $\varnothing$, $\abs+$, $\abs-$, or\/~$\abs$,
there is a triangulated equivalence
$$
 \sD^\st(\sE_{l_1})\simeq\sD^\st(\sF_{l_1})
$$
provided by (appropriately defined) mutually inverse derived functors\/
$\boR\Hom_S(L^\bu,{-})$ and $L^\bu\ot_S^\boL{-}$.
\end{thm}

\begin{proof}
 This is a particular case of
Theorem~\ref{abstract-classes-derived-equivalence} below.
\end{proof}

 Let us make some final comments before this section is finished.
 According to~\cite[Proposition~5.5]{PS2}, there is a natural
degenerate t\+structure of the derived type on the triangulated
category $\sD'_{L^\bu}(S\Modl_{J\tors})=\sD(\sE_{l_1})$ with
the heart equivalent to $S\Modl_{J\tors}$.
 Dual-analogously, by~\cite[Proposition~5.7]{PS2}, there is a natural
degenerate t\+structure of the derived type on the triangulated
category $\sD''_{L^\bu}(S\Modl_{J\ctra})=\sD(\sF_{l_1})$ with
the heart equivalent to $S\Modl_{J\ctra}$.
 See also the discussion in~\cite[Section~1.2 and Remark~5.3]{Ppc}.

 Following the discussion in~\cite[Section~2]{Ppc}, the functor
$\sD'_{L^\bu}(S\Modl_{J\tors})\rarrow\sD(S\Modl_{J\tors})$ induced
by the exact inclusion of exact/abelian categoires
$\sE_{l_1}\rarrow S\Modl_{J\tors}$ is a triangulated Verdier quotient
functor having a (fully faithful) right adjoint.
 Dual-analogously, the functor $\sD''_{L^\bu}(S\Modl_{J\ctra})\rarrow
\sD(S\Modl_{J\ctra})$ induced by the exact inclusion of exact/abelian
categories $\sF_{l_1}\rarrow S\Modl_{J\ctra}$ is a triangulated
Verdier quotient functor having a (fully faithful) left adjoint.
 See also the discussion of diagram~\eqref{abstract-big-diagram}
in the next Section~\ref{abstract-classes-secn}.

\Section{Abstract Corresponding Classes}  \label{abstract-classes-secn}

 More generally, suppose that we are given two full subcategories
$\sE\subset S\Modl_{J\tors}$ and $\sF\subset S\Modl_{J\ctra}$ satisfying
the following conditions (for some fixed integers $l_1$ and~$l_2$):
\begin{enumerate}
\renewcommand{\theenumi}{\Roman{enumi}}
\item the class of objects $\sE$ is closed under extensions and
cokernels of injective morphisms in $S\Modl_{J\tors}$, and contains
all the injective objects of $S\Modl_{J\tors}$;
\item the class of objects $\sF$ is closed under extensions and
kernels of surjective morphisms in $S\Modl_{J\ctra}$, and contains
all the projective objects of $S\Modl_{J\ctra}$;
\item for any $J$\+torsion $S$\+module $E\in\sE$, the derived category
object $\boR\Hom_S(L^\bu,E)\allowbreak\in\sD^+(S\Modl)$ can be
represented by a complex of $J$\+contramodule $S$\+modules concentrated
in the cohomological degrees $-l_2\le m\le l_1$ with the terms belonging
to~$\sF$;
\item for any $J$\+contramodule $S$\+module $F\in\sF$, the derived
category object $L^\bu\ot_S^\boL F\in\sD^-(S\Modl)$ can be represented
by a complex of $J$\+torsion $S$\+modules concentrated in
the cohomological degrees $-l_1\le m\le l_2$ with the terms belonging
to~$\sE$.
\end{enumerate}

 Similarly to~\cite[Section~4]{Pps} and~\cite[Section~5]{Ppc},
one can see from conditions~(I) and~(III), or~(II) and~(IV), that
$l_1\ge d_1$ and $l_2\ge d_2$ whenever
$H^{-d_1}(L^\bu)\ne0\ne H^{d_2}(L^\bu)$.
 One also needs to use
Lemma~\ref{Hom-tensor-underived=derived-lemma}(a,c).

 According to
Lemmas~\ref{auslander-bass-classes-closure-properties-lemma},
\ref{injectives-projectives-belong-to-auslander-bass},
and~\ref{Hom-tensor-take-bass-auslander-to-complexes-of},
the Bass and Auslander classes $\sE=\sE_{l_1}$ and $\sF=\sF_{l_2}$
satisfy conditions~(I\+-IV) with $l_2=d_2$.
 The following lemma can be viewed as providing a converse implication.

\begin{lem} \label{abstract-corresponding-classes-adjunction-isoms}
\textup{(a)} For any $J$\+torsion $S$\+module $E\in\sE$, the adjunction
morphism $L^\bu\ot_S^\boL\boR\Hom_S(L^\bu,E)\rarrow E$ is an isomorphism
in\/ $\sD^\bb(S\Modl)$. \par
\textup{(b)} For any $J$\+contramodule $S$\+module $F\in\sF$,
the adjunction morphism $F\rarrow\boR\Hom_S(L^\bu,\>L^\bu\ot_S^\boL F)$
is an isomorphism in\/ $\sD^\bb(S\Modl)$.
\end{lem}

\begin{proof}
 This is similar to~\cite[Lemma~4.1]{Pps} and~\cite[Lemma~5.1]{Ppc}.
 A direct argument along the lines of~\cite[proof of Lemma~4.1]{Pps}
is applicable, or alternatively, the assertions can be obtained from
(the proof of) Theorem~\ref{abstract-classes-derived-equivalence} below.
 In any case, the proof is based on
Lemmas~\ref{tors-contra-tensor-hom-lemma}(a\+-b),
\ref{Hom-tensor-underived=derived-lemma}(a,c),
and~\ref{injectives-projectives-belong-to-auslander-bass}.
\end{proof}

 Assuming that $l_1\ge d_1$ and $l_2\ge d_2$, it is clear from
conditions (III\+-IV) and
Lemma~\ref{abstract-corresponding-classes-adjunction-isoms}
that the inclusions $\sE\subset\sE_{l_1}$ and $\sF\subset\sF_{l_1}$
hold for any two classes of objects $\sE\subset S\Modl_{J\tors}$
and $\sF\subset S\Modl_{J\ctra}$ satisfying~(I\+-IV).
 Furthermore, it follows from conditions (I\+-II) that the triangulated
functors $\sD^\bb(\sE)\rarrow\sD^\bb(S\Modl_{J\tors})$ and
$\sD^\bb(\sF)\rarrow\sD^\bb(S\Modl_{J\ctra})$ are fully faithful.
 Hence the triangulated functors $\sD^\bb(\sE)\rarrow\sD^\bb(\sE_{l_1})$
and $\sD^\bb(\sF)\rarrow\sD^\bb(\sF_{l_1})$ are fully faithful, too.
 Using again conditions~(III\+-IV), we conclude that
the equivalence~\eqref{bounded-derived-bass-auslander-equivalence}
restricts to a triangulated equivalence
\begin{equation} \label{bounded-derived-abstract-classes-equivalence}
 \sD^\bb(\sE)\simeq\sD^\bb(\sF).
\end{equation}

 Let us introduce simplified notation
$S\Modl_{J\tors}^\inj=(S\Modl_{J\tors})^\inj$ and
$S\Modl_{J\ctra}^\proj\allowbreak=(S\Modl_{J\ctra})^\proj$ for
the full subcategories of injective objects in $S\Modl_{J\tors}$ and
projective objects in $S\Modl_{J\ctra}$.

 The following theorem is the first main result of this paper.

\begin{thm} \label{abstract-classes-derived-equivalence}
 Let\/ $\sE\subset S\Modl_{J\tors}$ and\/ $\sF\subset S\Modl_{J\ctra}$
be a pair of full subcategories of $J$\+torsion and $J$\+contramodule
$S$\+modules satisfying conditions (I\+-IV) for a pseudo-dualizing
complex of $J$\+torsion $S$\+modules~$L^\bu$.
 Then, for any conventional or absolute derived category symbol\/
$\st=\bb$, $+$, $-$, $\varnothing$, $\abs+$, $\abs-$, or\/~$\abs$,
there is a triangulated equivalence
$$
 \sD^\st(\sE)\simeq\sD^\st(\sF)
$$
provided by (appropriately defined) mutually inverse derived functors\/
$\boR\Hom_S(L^\bu,{-})$ and $L^\bu\ot_S^\boL{-}$.
\end{thm}

\begin{proof}
 The proof is completely similar to those of~\cite[Theorem~4.2]{Pps}
and~\cite[Theorem~5.2]{Ppc}.
 The words ``appropriately defined'' here mean ``produced by
the constructions of~\cite[Appendix~A]{Pps}''.
 In the context of the latter, we set
\begin{alignat*}{5}
 &\sA&&=S\Modl_{J\tors}&&\supset\sE
 &&\supset\sJ&&=S\Modl_{J\tors}^\inj, \\
 &\sB&&=S\Modl_{J\ctra}&&\supset\sF
 &&\supset\sP&&=S\Modl_{J\ctra}^\proj.
\end{alignat*}

 Consider the adjoint pair of DG\+functors
\begin{alignat*}{2}
 \Psi=\Hom_S(L^\bu,{-})&\:\sC^+(\sJ)&&\lrarrow\sC^+(\sB), \\
 \Phi=L^\bu\ot_S{-}\,&\:\sC^-(\sP)&&\lrarrow\sC^-(\sA)
\end{alignat*}
(see Lemma~\ref{tors-contra-tensor-hom-lemma}(a\+-b)).
 Then the constructions of~\cite[Sections~A.2--A.3]{Pps} provide
the desired derived functors $\boR\Psi\:\sD^\st(\sE)\rarrow
\sD^\st(\sF)$ and $\boL\Phi\:\sD^\st(\sF)\rarrow\sD^\st(\sE)$.
 According to~\cite[Section~A.4]{Pps}, the functor $\boL\Phi$ is
left adjoint to the functor~$\boR\Psi$.

 Finally, the result of~\cite[first assertion of Theorem~A.5]{Pps}
allows to deduce the claim that $\boR\Psi$ and $\boL\Phi$ are
mutually inverse equivalences from the particular case of $\st=\bb$,
which is the triangulated
equivalence~\eqref{bounded-derived-abstract-classes-equivalence}.
 Alternatively, applying~\cite[second assertion of Theorem~A.5]{Pps}
together with
Lemma~\ref{injectives-projectives-belong-to-auslander-bass} (and keeping
Lemma~\ref{Hom-tensor-underived=derived-lemma}(a,c) in mind) allows one
to reprove the triangulated
equivalence~\eqref{bounded-derived-abstract-classes-equivalence}
instead of using it, thus obtaining a proof
of Lemma~\ref{abstract-corresponding-classes-adjunction-isoms}.
\end{proof}

 Let us make some comments generalizing the discussion at the end of
Section~\ref{auslander-and-bass-secn}.
 According to~\cite[Proposition~5.5]{PS2}, there is a natural
degenerate t\+structure of the derived type on the triangulated
category $\sD(\sE)$ with the heart equivalent to $S\Modl_{J\tors}$.
 Dual-analogously, by~\cite[Proposition~5.7]{PS2}, there is a natural
degenerate t\+structure of the derived type on the triangulated
category $\sD(\sF)$ with the heart equivalent to $S\Modl_{J\ctra}$.
 See also the discussion in~\cite[Section~1.2 and Remark~5.3]{Ppc}.

 The category of $J$\+torsion $S$\+modules $S\Modl_{J\tors}$ is
a Grothendieck abelian categoyry.
 Hence, by~\cite[Theorem~3.13 and Lemma~3.7(ii)]{Ser},
\cite[Corollary~7.1]{Gil}, or~\cite[Corollary~8.5]{PS4}, there are
enough homotopy injective complexes of injective objects in
$S\Modl_{J\tors}$.
 So the result of~\cite[Theorem~2.1(a)]{Ppc} is applicable, telling us
that the triangulated functor $\sD(\sE)\rarrow\sD(S\Modl_{J\tors})$
induced by the exact inclusion of exact/abelian categories
$\sE\rarrow S\Modl_{J\tors}$ is a triangulated Verdier quotient
functor having a (fully faithful) right adjoint.

 Dual-analogously, the category of $J$\+contramodule $S$\+modules
$S\Modl_{J\ctra}$ is a locally presentable (in fact, locally
$\aleph_1$\+presentable) abelian category with enough projective
objects.
 Hence, by~\cite[Lemma~6.1 and Corollary~6.7]{PS4}, there are enough
homotopy projective complexes of projective objects in
$S\Modl_{J\ctra}$.
 So the result of~\cite[Theorem~2.1(b)]{Ppc} is applicable, telling us
that the triangulated functor $\sD(\sF)\rarrow\sD(S\Modl_{J\ctra})$
induced by the exact inclusion of exact/abelian categories
$\sF\rarrow S\Modl_{J\ctra}$ is a triangulated Verdier quotient
functor having a (fully faithful) left adjoint.

 In other words, we have a diagram of triangulated functors,
triangulated equivalences, commutativities, and adjunctions
\begin{equation} \label{abstract-big-diagram}
\begin{tikzcd}
\sK(S\Modl_{J\tors}) \arrow[d, two heads]
\arrow[ddd, two heads, bend right=71] &&&&
\sK(S\Modl_{J\ctra}) \arrow[d, two heads]
\arrow[ddd, two heads, bend left=71] \\
\sK(S\Modl_{J\tors}^\inj) \arrow[d] 
\arrow[u, tail, bend right=70]
\arrow[dd, two heads, bend right=62] &&&&
\sK(S\Modl_{J\ctra}^\proj) \arrow[d] 
\arrow[u, tail, bend left=70]
\arrow[dd, two heads, bend left=62] \\
\sD(\sE) \arrow[d, two heads]
\arrow[rrrr, Leftrightarrow, no head, no tail] &&&&
\sD(\sF) \arrow[d, two heads] \\
\sD(S\Modl_{J\tors}) \arrow[uuu, tail, bend right=82]
\arrow[uu, tail, bend right=70]
\arrow[u, tail, bend right=60] &&&&
\sD(S\Modl_{J\ctra}) \arrow[uuu, tail, bend left=82]
\arrow[uu, tail, bend left=70]
\arrow[u, tail, bend left=60]
\end{tikzcd}
\end{equation}
with the notation and description very similar to the discussion
of the diagram~\eqref{big-diagram} in the Introduction.
 (Cf.\ the discussion in~\cite[Section~9]{Ppc}.)

 The arrows that are present on both the diagrams~\eqref{big-diagram}
and~\eqref{abstract-big-diagram} denote the same functors.
 The horizontal double line in~\eqref{abstract-big-diagram} is
the triangulated equivalence from
Theorem~\ref{abstract-classes-derived-equivalence}.
 The downwards-directed straight arrows in the leftmost column
denote the triangulated functors between the homotopy/derived
categories induced by the exact inclusions of additive/exact/abelian
categories $S\Modl_{J\tors}^\inj\rarrow\sE\rarrow S\Modl_{J\tors}$.
 The downwards-directed straight arrows in the rightmost column
denote the triangulated functors between the homotopy/derived
categories induced by the exact inclusions of additive/exact/abelian
categories $S\Modl_{J\ctra}^\proj\rarrow\sF\rarrow S\Modl_{J\ctra}$.

 The upper levels of both the leftmost and the rightmost columns
in~\eqref{abstract-big-diagram} are provided by
Theorem~\ref{becker-co-contra-derived-of-loc-pres-abelian}.
 The triangulated functors $\sK(S\Modl_{J\tors}^\inj)\rarrow
\sD(S\Modl_{J\tors})$ and $\sK(S\Modl_{J\ctra}^\proj)\rarrow
\sD(S\Modl_{J\ctra})$ are Verdier quotient functors in view of
Lemma~\ref{becker-co-contra-acyclic-are-acyclic}
and Theorem~\ref{becker-co-contra-derived-of-loc-pres-abelian}.

 Now suppose that we have two pairs of full subcategories
$\sE_\prime\subset\sE'\subset S\Modl_{J\tors}$ and
$\sF_{\prime\prime}\subset\sF''\subset S\Modl_{J\ctra}$ such that
both the pairs $(\sE_\prime,\sF_{\prime\prime})$ and
$(\sE',\sF'')$ satisfy conditions~(I\+-IV).
 Then for every symbol $\st=\bb$, $+$, $-$, $\varnothing$, $\abs+$,
$\abs-$, or $\abs$ there is a commutative diagram of triangulated
functors and triangulated equivalences
\begin{equation} \label{pseudo-derived-equivs-commutativity-diagram}
\begin{gathered}
 \xymatrix{
  \sD^\st(\sE_\prime) \ar@{=}[r] \ar[d]
  & \sD^\st(\sF_{\prime\prime}) \ar[d] \\
  \sD^\st(\sE') \ar@{=}[r] & \sD^\st(\sF'')
 }
\end{gathered}
\end{equation}
 The vertical functors are induced by the exact inclusions of exact
categories $\sE_\prime\rarrow\sE'$ and $\sF_{\prime\prime}\rarrow
\sF''$, while the horizontal equivalences are provided by
Theorem~\ref{abstract-classes-derived-equivalence}.

\Section{Minimal Corresponding Classes}  \label{minimal-classes-secn}

 Let $J$ be a weakly proregular finitely generated ideal in
a commutative ring $S$, and let $L^\bu$ be a pseudo-dualizing complex
of $J$\+torsion $S$\+modules concentrated in the cohomological
degrees $-d_1\le m\le d_2$.

\begin{prop}
 Fix $l_1=d_1$ and $l_2\ge d_2$.
 Then there exists a unique minimal pair of full subcategories\/
$\sE^{l_2}=\sE^{l_2}(L^\bu)\subset S\Modl_{J\tors}$ and\/
$\sF^{l_2}=\sF^{l_2}(L^\bu)\subset S\Modl_{J\ctra}$ satisfying
conditions (I\+-IV) from Section~\ref{abstract-classes-secn}.
 For any pair of full subcategories\/ $\sE\subset S\Modl_{J\tors}$
and\/ $\sF\subset S\Modl_{J\ctra}$ satisfying (I\+-IV), one has\/
$\sE^{l_2}\subset\sE$ and\/ $\sF^{l_2}\subset\sF$.
\end{prop}

\begin{proof}
 The full subcategories $\sE^{l_2}\subset S\Modl_{J\tors}$ and
$\sF^{l_2}\subset S\Modl_{J\ctra}$ are constructed simultaneously
by a generation process similar to the ones in~\cite[proof of
Proposition~5.1]{Pps} and~\cite[proof of Proposition~6.1]{Ppc}.
 The difference is that, \emph{like} in~\cite{Ppc} and \emph{unlike}
in~\cite{Pps}, we do not require the classes $\sE^{l_2}$ and $\sF^{l_2}$
to be closed under infinite direct sums and products.
 Accordingly, no transfinite iterations of the generation process
are needed.
\end{proof}

\begin{rem} \label{minimal-classes-remark}
 Moreover, for any two integers $l_1\ge d_1$ and $l_2\ge d_2$ and
any two full subcategories $\sE\subset S\Modl_{J\tors}$ and
$\sF\subset S\Modl_{J\ctra}$ satisfying conditions (I\+-IV) with
the parameters $l_1$ and~$l_2$, one has $\sE^{l_2}\subset\sE$ and
$\sF^{l_2}\subset\sF$.
 This can be easily seen from the construction of the classes
$\sE^{l_2}$ and~$\sF^{l_2}$ (cf.~\cite[Remark~5.2]{Pps}
and~\cite[Remark~6.2]{Ppc}).
\end{rem}

 One observes that the conditions (III\+-IV) become weaker as
the parameter~$l_2$ increases.
 It follows that one has $\sE^{l_2}\supset\sE^{l_2+1}$ and
$\sF^{l_2}\supset\sF^{l_2+1}$ for all $l_2\ge d_2$.
 So the inclusions between our classes of $J$\+torsion $S$\+modules
and $J$\+contramodule $S$\+modules have the form
\begin{alignat*}{7}
 &\dotsb\subset\sE^{d_2+2}&&\subset\sE^{d_2+1}&&\subset\sE^{d_2}
 &&\subset\sE_{d_1}&&\subset\sE_{d_1+1}&&\subset\sE_{d_1+2}
 &&\subset\dotsb\subset S\Modl_{J\tors}, \\
 &\dotsb\subset\sF^{d_2+2}&&\subset\sF^{d_2+1}&&\subset\sF^{d_2}
 &&\subset\sF_{d_1}&&\subset\sF_{d_1+1}&&\subset\sF_{d_1+2}
 &&\subset\dotsb\subset S\Modl_{J\ctra}.
\end{alignat*}

\begin{lem} \label{minimal-classes-co-resolution-dimension}
 Let $n\ge0$, \,$l_1\ge d_1$, and $l_2\ge d_2+n$ be some integers,
and let\/ $\sE\subset S\Modl_{J\tors}$ and $\sF\subset S\Modl_{J\ctra}$
be a pair of full subcategories satisfying conditions (I\+-IV) with
the parameters $l_1$ and~$l_2$.
 Denote by\/ $\sE(n)\subset S\Modl_{J\tors}$ the full subcategory of
all $J$\+torsion $S$\+modules of\/ $\sE$\+coresolution
dimension\/~$\le n$ and by\/ $\sF(n)\subset S\Modl_{J\ctra}$ the full
subcategory of all $J$\+contramodule $S$\+modules of\/ $\sF$\+resolution
dimension\/~$\le n$.
 Then the pair of classes of $J$\+torsion and $J$\+contramodule
$S$\+modules\/ $\sE(n)$ and\/ $\sF(n)$ satisfies conditions (I\+-IV)
with the parameters $l_1+n$ and $l_2-n$.
\end{lem}

\begin{proof}
 Similar to~\cite[Lemma~5.3]{Pps} and~\cite[Lemma~6.3]{Ppc}.
\end{proof}

\begin{prop} \label{minimal-classes-derived-independence-of-l2}
\textup{(a)} For any integers $l_2''\ge l_2'\ge d_2$ and any
conventional or exotic derived category symbol\/ $\st=\bb$, $+$, $-$,
$\varnothing$, $\abs+$, $\abs-$, $\bco$, or\/~$\abs$, the exact
inclusion functor\/ $\sE^{l_2''}\rarrow\sE^{l_2'}$ induces
a triangulated equivalence
$$
 \sD^\st(\sE^{l_2''})\simeq\sD^\st(\sE^{l_2''}).
$$ \par
\textup{(b)} For any integers $l_2''\ge l_2'\ge d_2$ and any
conventional or exotic derived category symbol\/ $\st=\bb$, $+$, $-$,
$\varnothing$, $\abs+$, $\abs-$, $\bctr$, or\/~$\abs$, the exact
inclusion functor\/ $\sF^{l_2''}\rarrow\sF^{l_2'}$ induces
a triangulated equivalence
$$
 \sD^\st(\sF^{l_2''})\simeq\sD^\st(\sF^{l_2'}).
$$
\end{prop}

\begin{proof}
 Similar to~\cite[Proposition~5.4]{Pps} and~\cite[Proposition~6.4]{Ppc}.
 In part~(b), one uses
Lemma~\ref{minimal-classes-co-resolution-dimension} in order to check
that the $\sF^{l_2''}$\+resolution dimension of any object from
$\sF^{l_2'}$ does not exceed $l_2''-l_2'$.
 Then one applies~\cite[Propositions~A.5.8 and~B.7.9]{Pcosh},
as in the proof of
Proposition~\ref{bass-auslander-derived-independence-of-l1}.
 In part~(a), one similarly uses
Lemma~\ref{minimal-classes-co-resolution-dimension} in order to check
that the $\sE^{l_2''}$\+coresolution dimension of any object from
$\sE^{l_2'}$ does not exceed $l_2''-l_2'$.
 Then one applies the dual versions
of~\cite[Propositions~A.5.8 and~B.7.9]{Pcosh}.
\end{proof}

 As in Proposition~\ref{bass-auslander-derived-independence-of-l1},
the cases $\st=\bco$ and $\st=\bctr$ in the context of
Proposition~\ref{minimal-classes-derived-independence-of-l2} are
actually trivial, and are only included in the formulation for the sake
of completeness and for comparison with~\cite[Proposition~5.4]{Pps}.
 Using the results of~\cite[Corollary~9.5]{PS4} for
$\sA=S\Modl_{J\tors}$ and~\cite[Corollary~7.4]{PS4} for
$\sB=S\Modl_{J\ctra}$, one can easily show that
$\sD^\bco(\sE^{l_2''})\simeq\sD^\bco(\sE^{l_2'})\simeq
\sD^\bco(S\Modl_{J\tors})$ and
$\sD^\bctr(\sF^{l_2''})\simeq\sD^\bctr(\sF^{l_2'})\simeq
\sD^\bctr(S\Modl_{J\ctra})$.

 As a particular case of
Proposition~\ref{minimal-classes-derived-independence-of-l2},
the conventional unbounded derived category $\sD(\sE^{l_2})$ of
the minimal corresponding class of $J$\+torsion $S$\+modules
$\sE^{l_2}$ is the same for all $l_2\ge d_2$, and the conventional
unbounded derived category $\sD(\sF^{l_2})$ of the minimal
corresponding class of $J$\+contramodule $S$\+modules $\sF^{l_2}$
is the same for all $l_2\ge d_2$.
 We put
$$
 \sD_\prime^{L^\bu}(S\Modl_{J\tors})=\sD(\sE^{l_2})
 \quad\text{and}\quad
 \sD_{\prime\prime}^{L^\bu}(S\Modl_{J\ctra})=\sD(\sF^{l_2}).
$$
 The next theorem provides, in particular, a triangulated equivalence
$$
 \sD_\prime^{L^\bu}(S\Modl_{J\tors})=\sD(\sE^{l_2})\simeq
 \sD(\sF^{l_2})=\sD_{\prime\prime}^{L^\bu}(S\Modl_{J\ctra}).
$$

\begin{thm} \label{minimal-classes-derived-equivalence}
 For any conventional or absolute derived category symbol\/
$\st=\bb$, $+$, $-$, $\varnothing$, $\abs+$, $\abs-$, or\/~$\abs$,
there is a triangulated equivalence
$$
 \sD^\st(\sE^{l_2})\simeq\sD^\st(\sF^{l_2})
$$
provided by (appropriately defined) mutually inverse derived functors\/
$\boR\Hom_S(L^\bu,{-})$ and $L^\bu\ot_S^\boL{-}$.
\end{thm}

\begin{proof}
 This is another particular case of
Theorem~\ref{abstract-classes-derived-equivalence}.
\end{proof}

 Similarly to the discussion at the end of
Section~\ref{auslander-and-bass-secn}, and as a particular case
of the discussion in Section~\ref{abstract-classes-secn}, we mention
the following observations.
 According to~\cite[Proposition~5.5]{PS2}, there is a natural
degenerate t\+structure of the derived type on the triangulated
category $\sD_\prime^{L^\bu}(S\Modl_{J\tors})=\sD(\sE^{l_2})$ with
the heart equivalent to $S\Modl_{J\tors}$.
 Dual-analogously, by~\cite[Proposition~5.7]{PS2}, there is a natural
degenerate t\+structure of the derived type on the triangulated
category $\sD_{\prime\prime}^{L^\bu}(S\Modl_{J\ctra})=\sD(\sF^{l_2})$
with the heart equivalent to $S\Modl_{J\ctra}$.
 See also the discussion in~\cite[Section~1.2 and Remark~5.3]{Ppc}.

 Following the discussion in~\cite[Section~2]{Ppc}, the functor
$\sD_\prime^{L^\bu}(S\Modl_{J\tors})\rarrow\sD(S\Modl_{J\tors})$
induced by the exact inclusion of exact/abelian categoires
$\sE^{l_2}\rarrow S\Modl_{J\tors}$ is a triangulated Verdier quotient
functor having a (fully faithful) right adjoint.
 Dual-analogously, the functor
$\sD_{\prime\prime}^{L^\bu}(S\Modl_{J\ctra})\rarrow\sD(S\Modl_{J\ctra})$
induced by the exact inclusion of exact/abelian categories
$\sF^{l_2}\rarrow S\Modl_{J\ctra}$ is a triangulated Verdier quotient
functor having a (fully faithful) left adjoint.
 See also the diagrams~\eqref{big-diagram} in the Introduction
and~\eqref{abstract-big-diagram} in Section~\ref{abstract-classes-secn}.

\Section{Finiteness Conditions for an Ideal with Artinian Quotient Ring}
\label{artinian-finiteness-secn}

 Let $S$ be a Noetherian commutative ring and $J\subset S$ be an ideal
such that the quotient ring $S/J$ is Artinian.
 The aim of this section is to compare two finiteness conditions on
a finite complex of $J$\+torsion $S$\+modules: viz., condition~(ii)
from the definition of a pseudo-dualizing complex in
Section~\ref{auslander-and-bass-secn} above and condition~(iii)
from the definition of a dedualizing complex in~\cite[Section~4]{Pmgm}.

 Let $\sqrt{J}\subset S$ denote the radical of the ideal~$J$.
 Notice that the quotient ring $S/\sqrt{J}$ is a semisimple Artinian
commutative ring, i.~e., a finite direct sum of fields.

 Given a $J$\+torsion $S$\+module $M$, denote by $\soc(M)\subset M$
the socle of $M$, i.~e., the maximal semisimple submodule of~$M$.
 Equivalently, $\soc(M)$ is the maximal $S$\+submodule of $M$ whose
$S$\+module structure comes from an $S/\sqrt{J}$\+module structure.
 It follows that one has $\soc(M)\ne0$ whenever $M\ne0$.

 So $M\longmapsto\soc(M)$ is a functor $S\Modl_{J\tors}\rarrow
(S/\sqrt{J})\Modl$.
 The functor $\soc\:S\Modl_{J\tors}\rarrow(S/\sqrt{J})\Modl$ is right
adjoint to the identity inclusion functor $(S/\sqrt{J})\Modl\rarrow
S\Modl_{J\tors}$.

\begin{lem} \label{socle-injectivity-criterion}
 Let $f\:M\rarrow N$ be a morphism of $J$\+torsion $S$\+modules.
 Then the morphism~$f$ is injective if and only if the morphism
$\soc(f)\:\soc(M)\rarrow\soc(N)$ is injective.
\end{lem}

\begin{proof}
 The functor $\soc$ is a right adjoint, hence it is left exact, i.~e.,
preserves kernels.
 Thus we have $\ker(\soc(f))=\soc(\ker(f))$.
 As $\ker(f)$ is a $J$\+torsion $S$\+module, we have $\ker(f)\ne0$
if and only if $\soc(\ker(f))\ne0$.
\end{proof}

 Recall that a $J$\+torsion $S$\+module is injective in
$S\Modl_{J\tors}$ if and only if it is injective in $S\Modl$ (since
the ring $S$ is Noetherian).
 A complex of $J$\+torsion $S$\+modules $J^\bu$ is said to be
\emph{minimal} if the differential of the complex $\soc(J^\bu)$
vanishes.

\begin{lem} \label{enough-minimal-complexes-of-injectives}
 Any complex of injective $J$\+torsion $S$\+modules decomposes as
a direct sum of a minimal complex of injective $J$\+torsion $S$\+modules
and a contractible complex of injective $J$\+torsion $S$\+modules.
\end{lem}

\begin{proof}
 Let $H^\bu$ be a complex of injective $J$\+torsion $S$\+modules.
 For every integer $n\in\boZ$, let $T_n\subset\soc(H^n)$ be
a complementary submodule to the kernel of the map $\soc(H^n)\rarrow
\soc(H^{n+1})$; so $T_n$ is a maximal submodule among all submodules
$T\subset\soc(H^n)$ such that the composition $T\rarrow\soc(H^n)\rarrow
\soc(H^{n+1})$ is injective.
 In other words, the map from $T_n$ to the image of the morphism
$\soc(H^n)\rarrow\soc(H^{n+1})$ is an isomorphism.
 Denote by $K_n$ the injective envelope of $T_n$ in $S\Modl_{J\tors}$,
or equivalently, in $S\Modl$.
 Then the inclusion $T_n\rarrow\soc(H^n)\rarrow H^n$ can be extended to
an injective map of $J$\+torsion $S$\+modules $K_n\rarrow H^n$.
 We have $\soc(K_n)=T_n$, so it follows from
Lemma~\ref{socle-injectivity-criterion} that the composition
$K_n\rarrow H^n\rarrow H^{n+1}$ is an injective map.

 We have constructed an injective morphism into the complex of
$S$\+modules $H^\bu$ from a contractible two-term complex of injective
$J$\+torsion $S$\+modules $\dotsb\rarrow0\rarrow K_n\overset{\id}\rarrow
K_n\rarrow0\rarrow\dotsb$ situated in the cohomological degrees~$n$
and~$n+1$.
 Now the composition $K_n\rarrow H^{n+1}\rarrow H^{n+2}$ vanishes,
while the composition $K_{n+1}\rarrow H^{n+1}\rarrow H^{n+2}$ is
injective, too.
 It follows that the images of $K_n$ and $K_{n+1}$ do not intersect in
$H^{n+1}$, so the map $K_n\oplus K_{n+1}\rarrow H^{n+1}$ is injective.
 We have arrived to an injective morphism of complexes
$K^\bu=\bigoplus_{n\in\boZ}(K_n\overset{\id}\to K_n)\rarrow H^\bu$.
 Contractible complexes of injective objects are injective objects of
the abelian category of complexes $\sC(S\Modl_{J\tors})$;
hence the complex $H^\bu$ decomposes into a direct sum of the complex
$K^\bu$ and some complex of injective $J$\+torsion $S$\+modules~$G^\bu$.
 One can easily see that the morphism of complexes $\soc(K^\bu)\rarrow
\soc(H^\bu)$ induces an isomorphism on the images of the differentials;
so the complex $G^\bu$ is minimal.
\end{proof}

 The following theorem is the main result of this section.

\begin{thm} \label{Artinian-quotient-ring-finitness-conditions-theorem}
 Let $S$ be a Noetherian commutative ring and $J\subset S$ be
an ideal such that the quotient ring $S/J$ is Artinian.
 Let $N^\bu$ be a finite complex of $J$\+torsion $S$\+modules.
 Then the following two conditions are equivalent:
\begin{enumerate}
\item the cohomology $S$\+modules of the complex $N^\bu$ are Artinian;
\item for every finite complex of finitely generated projective
$S$\+modules $K^\bu$ with $J$\+torsion cohomology modules,
the complex of $S$\+modules\/ $\Hom_S(K^\bu,N^\bu)$ is quasi-isomorphic
to a bounded above complex of finitely generated projective
$S$\+modules.
\end{enumerate}
\end{thm}

\begin{proof}
 (1)~$\Longrightarrow$~(2)
 It follows from~\cite[Lemma~4.3]{Pmgm} that $N^\bu$ is quasi-isomorphic
to a finite complex of Artinian $J$\+torsion $S$\+modules~$M^\bu$.
 Then the complex of $S$\+modules $\Hom_S(K^\bu,N^\bu)$ is
quasi-isomorphic to the complex $\Hom_S(K^\bu,M^\bu)\simeq
\Hom_S(K^\bu,S)\ot_S M^\bu$, which is also a finite complex of
Artinian $J$\+torsion $S$\+modules.

 Clearly, there exists an integer $n\ge1$ such that all the elements
of $J^n\subset S$ act on the complex of $S$\+modules $K^\bu$ by
endomorphisms homotopic to zero.
 Then the cohomology modules of the complex $\Hom_S(K^\bu,M^\bu)$ are
annihilated by~$J^n$.

 Any Artinian $S$\+module $H$ annihilated by $J^n$ is an Artinian module
over the Artinian ring $S/J^n$, and it follows that the $S$\+module $H$
is finitely generated.
 Any finite complex of modules $C^\bu$ over a Noetherian ring $S$ with
finitely generated cohomology modules $H^i(C^\bu)$, \,$i\in\boZ$,
is quasi-isomorphic to a bounded above complex of finitely generated
projective $S$\+modules.

 (2)~$\Longrightarrow$~(1)
 As any bounded below complex of $J$\+torsion $S$\+modules, the complex
$N^\bu$ is quasi-isomorphic to some bounded below complex of injective
$J$\+torsion $S$\+modules~$H^\bu$.
 By Lemma~\ref{enough-minimal-complexes-of-injectives}, we can assume
without loss of generality that the complex $H^\bu$ is minimal.
 Using the canonical truncation, we construct from $H^\bu$ a finite
minimal complex of $J$\+torsion $S$\+modules $M^\bu$ quasi-isomorphic
to~$N^\bu$.

 Let $\bs=(s_1,\dotsc,s_m)$ be a finite sequence of generators of
the ideal $\sqrt{J}\subset S$.
 Then the dual Koszul complex $K^\bu=K^\bu(S,\bs)$ is a finite complex
of finitely generated free $S$\+modules with $J$\+torsion cohomology
modules.
 In fact, every element of $\sqrt{J}\subset S$ acts on $K^\bu$ by
an endomorphism homotopic to zero.

 The finite complex of $J$\+torsion $S$\+modules
$C^\bu=\Hom_S(K^\bu,M^\bu)\simeq K_\bu(S,\bs)\ot_SM^\bu$ is minimal,
since $\soc(C^\bu)\simeq K_\bu(S,\bs)\ot_S\soc(M^\bu)$ is a complex
with zero differential.
 Every element of $\sqrt{J}$ acts on $C^\bu$ by an endomorphism
homotopic to zero, so the cohomology modules of $C^\bu$ are
$S/\sqrt{J}$\+modules (i.~e., semisimple $J$\+torsion $S$\+modules).
 Furthermore, the complex of $S$\+modules $C^\bu$ is quasi-isomorphic
to $\Hom_S(K^\bu,N^\bu)$.
 By~(2), the complex $\Hom_S(K^\bu,N^\bu)$ is quasi-isomorphic to
a bounded above complex of finitely generated projective $S$\+modules.
 Thus the cohomology modules of $C^\bu$ are finitely generated
semisimple $J$\+torsion $S$\+modules.

 Let $n\in\boZ$ be the minimal integer such that the term $C^n$ of
the complex $C^\bu$ is \emph{not} an Artinian $S$\+module.
 Then, by~\cite[Lemma~4.1]{Pmgm}, the $S$\+module $\soc(C^n)$ is
\emph{not} finitely generated.
 The complex $C^\bu$ is minimal, so the composition $\soc(C^n)\rarrow
C^n\rarrow C^{n+1}$ vanishes.
 Hence $\soc(C^n)$ is an infinitely generated semisimple submodule
of the kernel $Z^n$ of the differential $C^n\rarrow C^{n+1}$.
 Thus the $S$\+module $Z^n$ is \emph{not} Artinian.
 On the other hand, by the choice of~$n$, the $S$\+module $C^{n-1}$
is Artinian.
 It follows that the cokernel of the differential $C^{n-1}\rarrow Z^n$
is not Artinian.
 This cokernel is the degree~$n$ cohomology module $H^n(C^\bu)$ of
the complex $C^\bu$, and we have seen in the previous paragraph
that the $S$\+module $H^n(C^\bu)$ is finitely generated and semisimple.
 The contradiction proves that an integer~$n$ does not exist, i.~e.,
all the terms of the complex $C^\bu$ are Artinian $S$\+modules.

 As $C^\bu\simeq K_\bu(S,\bs)\ot_SM^\bu$ and $K^\bu$ is a nonzero
finite complex of finitely generated free $S$\+modules, we arrive
to the conclusion that all the terms of the complex $M^\bu$ are
Artinian $J$\+torsion $S$\+modules, implying~(1).
\end{proof}

\Section{Dedualizing Complexes}  \label{dedualizing-secn}

 In this section we establish a comparsion of the definitions of
dedualizing complexes from~\cite[Section~4]{Pmgm}
and~\cite[Section~5]{Pmgm}, thus answering a question that was left
open in the paper~\cite{Pmgm}.
 We also deduce the triangulated equivalences of~\cite[Theorems~4.9
and~5.10]{Pmgm} as particular cases of
Theorem~\ref{bass-auslander-derived-equivalence} above.

 Let $\sE$ be an exact category.
 A finite complex $E^\bu$ in $\sE$ is said to have \emph{projective
dimension\/~$\le\nobreak d$} if $\Hom_{\sD^\bb(\sE)}(E^\bu,M[n])=0$ for
all objects $M\in\sE$ and all integers $n>d$.
 Dually, the complex $E^\bu$ is said to have \emph{injective
dimension\/~$\le\nobreak d$} if $\Hom_{\sD^\bb(\sE)}(M,E^\bu[n])=0$ for
all objects $M\in\sE$ and all integers $n>d$.
 Let us denote the projective dimension of $E^\bu$ as a complex
in $\sE$ by $\ppd_\sE(E^\bu)$ and the injective dimension of $E^\bu$
as a complex in $\sE$ by $\iid_\sE(E^\bu)$.

 Let $J$ be a weakly proregular finitely generated ideal in
a commutative ring~$S$.
 A finite complex of $J$\+torsion $S$\+modules $N^\bu$ is said to have
\emph{projective dimension\/~$\le\nobreak d$} if $\Ext^n_S(N^\bu,M)=0$
for all $J$\+torsion $S$\+modules $M$ and all integers $n>d$.
 Following~\cite[Section~4]{Pmgm}, we denote the projective dimension of
$N^\bu$ as a complex of $J$\+torsion $S$\+modules by
$\ppd_{(S,J)}N^\bu$.
 In view of
Theorem~\ref{torsion-modules-inclusion-derived-fully-faithful} or
Lemma~\ref{Hom-tensor-underived=derived-lemma}(a), the projective
dimension of $N^\bu$ as per the definition above is equal to its
projective dimension as a complex in the abelian category
$\sE=S\Modl_{J\tors}$ (which was the definition of the projective
dimension of a finite complex of torsion modules
in~\cite[Section~5]{Pmgm}).
 So we have $\ppd_{(S,J)}N^\bu=\ppd_{S\Modl_{J\tors}}N^\bu$.

 We will say that a finite complex of $J$\+contramodule $S$\+modules
$Q^\bu$ has \emph{injective dimension\/~$\le\nobreak d$} if
$\Ext^n_S(P,Q^\bu)=0$ for all $J$\+contramodule $S$\+modules $P$ and
all integers $n>d$.
 We denote the injective dimension of $Q^\bu$ as a complex of
$J$\+contramodule $S$\+modules by $\iid_{(S,J)}Q^\bu$.
 In view of
Theorem~\ref{contramodules-inclusion-derived-fully-faithful} or
Lemma~\ref{Hom-tensor-underived=derived-lemma}(b), the injective
dimension of $Q^\bu$ as per the definition above is equal to its
injective dimension as a complex in the abelian category
$\sE=S\Modl_{J\ctra}$.
 So we have $\iid_{(S,J)}Q^\bu=\iid_{S\Modl_{J\ctra}}Q^\bu$.

 A finite complex of $J$\+torsion $S$\+modules $N^\bu$ is said to have
\emph{contraflat dimension\/~$\le\nobreak d$} if $\Tor_n^S(N^\bu,P)=0$
for all $J$\+contramodule $S$\+modules $P$ and all integers $n>d$.
 In view of Lemma~\ref{Hom-tensor-underived=derived-lemma}(c) (see also
Corollary~\ref{contraflat-contramods-tensor-underived=derived-cor}),
this definition of the contraflat dimension of a finite complex of
torsion modules agrees with the one in~\cite[Section~5]{Pmgm}.
 Following~\cite[Section~4]{Pmgm}, we denote the contraflat dimension of
$N^\bu$ by $\cfd_{(S,J)}N^\bu$.
 It is clear from the formula $\Hom_\boZ(\Tor_n^S(N^\bu,P),\boQ/\boZ)
\simeq\Ext^n_S(P,\Hom_\boZ(N^\bu,\boQ/\boZ))$ that the contraflat
dimension of $N^\bu$ is equal to the injective dimension of the finite
complex of $J$\+contramodule $S$\+modules $\Hom_\boZ(N^\bu,\boQ/\boZ)$.
 So we have $\cfd_{(S,J)}N^\bu=\iid_{(S,J)}\Hom_\boZ(N^\bu,\boQ/\boZ)$.

 A finite complex of $S$\+modules $N^\bu$ is said to have
\emph{flat dimension\/~$\le\nobreak d$} if $\Tor_n^S(N^\bu,M)=0$
for all $S$\+modules $M$ and all integers $n>d$.
 We denote the flat dimension of $N^\bu$ by $\fd_SN^\bu$.
 Using the formula $\Hom_\boZ(\Tor_n^S(N^\bu,M),\boQ/\boZ)\simeq
\Ext^n_S(N^\bu,\Hom_\boZ(M,\boQ/\boZ))$, one can easily show
that the flat dimension of $N^\bu$ does not exceed its projective
dimension as a complex in the abelian category $\sE=S\Modl$.
 We also put $\ppd_SN^\bu=\ppd_{S\Modl}N^\bu$ and $\iid_SN^\bu=
\iid_{S\Modl}N^\bu$.
 So we have $\fd_SN^\bu\le\ppd_SN^\bu$.
 One can also easily see that $\fd_SN^\bu=
\iid_S\Hom_\boZ(N^\bu,\boQ/\boZ)$.

 Let $\fS=\varprojlim_{n\ge1}S/J^n$ be the $J$\+adic completion of
the ring~$S$.
 In the case of a Noetherian ring $S$, the arguments
from~\cite[Proposition~4.7 and its proof]{Pmgm} are applicable,
and one obtains the equations and inequalities
$$
 \fd_\fS N^\bu=\fd_SN^\bu=\cfd_{(S,J)}N^\bu\le\ppd_{(S,J)}N^\bu
 \le\ppd_S N^\bu
$$
for any finite complex of $J$\+torsion $S$\+modules~$N^\bu$
(use an injective cogenerator of the category $S\Modl_{J\tors}$
in place of the module $C$ in the context of~\cite[proof of
Proposition~4.7]{Pmgm}).
 The following lemma provides somewhat weaker inequalities for
a non-Noetherian ring~$S$.

\begin{lem} \label{torsion-and-abstract-homological-dimensions}
 Let $J$ be a weakly proregular finitely generated ideal in
a commutative ring $S$, and let $s_1$,~\dots, $s_m$ be a finite set
of generators of the ideal $J\subset S$.
 Let $N^\bu$ be a finite complex of $J$\+torsion $S$\+modules and
$Q^\bu$ be a finite complex of $J$\+contramodule $S$\+modules.
 Then one has \par
\textup{(a)} $\ppd_{(S,J)}N^\bu\le\ppd_SN^\bu\le\ppd_{(S,J)}N^\bu+m$;
\par
\textup{(b)} $\iid_{(S,J)}Q^\bu\le\iid_SQ^\bu\le\iid_{(S,J)}Q^\bu+m$;
\par
\textup{(c)} $\cfd_{(S,J)}N^\bu\le\fd_SN^\bu\le\cfd_{(S,J)}N^\bu+m$;
\par
\textup{(d)} $\cfd_{(S,J)}N^\bu\le\ppd_{(S,J)}N^\bu+m$.
\end{lem}

\begin{proof}
 Part~(a): the inequality $\ppd_{(S,J)}N^\bu\le\ppd_SN^\bu$ follows
immediately from the definitions.
 To prove the inequality $\ppd_SN^\bu\le\ppd_{(S,J)}N^\bu+m$, put
$\bs=(s_1,\dotsc,s_m)$, and let $K^\bu_\infty(S,\bs)$ be the infinite
dual Koszul complex from Section~\ref{prelims-on-wpr-secn}.
 Let $\sD_{J\tors}(S\Modl)\subset\sD(S\Modl)$ be the full subcategory of
all complexes with $J$\+torsion cohomology modules in $\sD(S\Modl)$.
 Following, e.~g., the discussion in~\cite[Section~2]{Pdc},
the functor $K^\bu_\infty(S,\bs)\ot_S{-}\,\:\sD(S\Modl)\rarrow
\sD_{J\tors}(S\Modl)$ is right adjoint to the inclusion functor
$\sD_{J\tors}(S\Modl)\rarrow\sD(S\Modl)$.
 So for any $S$\+module $M$ we have $\Hom_{\sD^\bb(S\Modl)}(N^\bu,M)
\simeq\Hom_{\sD^\bb(S\Modl)}(N^\bu,\>K^\bu_\infty(S,\bs)\ot_SM)$.
 It remains to point out that $K^\bu_\infty(S,\bs)\ot_SM$ is
a finite complex of $S$\+modules with $J$\+torsion cohomology modules
concentrated in the cohomological degrees~$\le m$.

 Part~(b): the inequality $\iid_{(S,J)}Q^\bu\le\iid_SQ^\bu$ follows
immediately from the definitions.
 To prove the inequality $\iid_SQ^\bu\le\iid_{(S,J)}Q^\bu+m$, keep
the notation~$\bs$ from the proof of part~(a), and let $T^\bu(S,\bs)$
be the finite complex of countably generated projective $S$\+modules
from Section~\ref{prelims-on-wpr-secn}.
 Let $\sD_{J\ctra}(S\Modl)\subset\sD(S\Modl)$ be the full subcategory
of all complexes with $J$\+contramodule cohomology modules in
$\sD(S\Modl)$.
 Following the discussion in~\cite[Section~2]{Pdc}, the functor
$\Hom_S(T^\bu(S,\bs),{-})\:\sD(S\Modl)\rarrow\sD_{J\ctra}(S\Modl)$
is left adjoint to the inclusion functor $\sD_{J\ctra}(S\Modl)
\rarrow\sD(S\Modl)$.
 So for any $S$\+module $M$ we have $\Hom_{\sD^\bb(S\Modl)}(M,Q^\bu)
\simeq\Hom_{\sD^\bb(S\Modl)}(\Hom_S(T^\bu(S,\bs),M),Q^\bu)$.
 It remains to point out that $\Hom_S(T^\bu(S,\bs),M)$ is
a finite complex of $S$\+modules with $J$\+contramodule cohomology
modules concentrated in the cohomological degrees~$\ge-m$.

 Part~(c): the inequality $\cfd_{(S,J)}N^\bu\le\fd_SN^\bu$ follows
immediately from the definitions.
 To prove the inequality $\fd_SN^\bu\le\cfd_{(S,J)}N^\bu+m$, we use
the equalities $\fd_SN^\bu=\iid_S\Hom_\boZ(N^\bu,\boQ/\boZ)$ and
$\cfd_{(S,J)}N^\bu=\iid_{(S,J)}\Hom_\boZ(N^\bu,\boQ/\boZ)$.
 Then it remains to apply part~(b) to the complex of $J$\+contramodule
$S$\+modules $Q^\bu=\Hom_\boZ(N^\bu,\boQ/\boZ)$.

 Part~(d) is provable by comparing parts~(a) and~(c).
 One has $\cfd_{(S,J)}N^\bu\le\fd_SN^\bu\le\ppd_SN^\bu
\le\ppd_{(S,J)}N^\bu+m$.
\end{proof}

 A \emph{dedualizing complex of $J$\+torsion $S$\+modules}
$L^\bu=B^\bu$ is a pseudo-dualizing complex (according to
the definition in Section~\ref{auslander-and-bass-secn}) satisfying
the following additional condition:
\begin{enumerate}
\renewcommand{\theenumi}{\roman{enumi}}
\item the complex $B^\bu$ has finite projective dimension as
a complex of $J$\+torsion $S$\+modules, that is,
$\ppd_{(S,J)}N^\bu<\infty$.
\end{enumerate}

\begin{lem} \label{comparison-with-finiteness-condition-in-Pmgm-Sect5}
 Let $J$ be a weakly proregular finitely generated ideal in
a commutative ring $S$, and let $N^\bu$ be a finite complex of
$J$\+torsion $S$\+modules.
 Assume that $N^\bu$ has finite projective dimension as a complex
of $J$\+torsion $S$\+modules.
 Let $K^\bu$ be a finite complex of finitely generated projective
$S$\+modules.
 Assume that the complex of $S$\+modules\/ $\Hom_S(K^\bu,N^\bu)$
is quasi-isomorphic to a bounded above complex of finitely generated
projective $S$\+modules.
 Then the complex\/ $\Hom_S(K^\bu,N^\bu)$ is actually quasi-isomorphic
to a finite complex of finitely generated projective $S$\+modules.
\end{lem}

\begin{proof}
 By Lemma~\ref{torsion-and-abstract-homological-dimensions}(a), it
follows from the first assumption of the present lemma that the complex
$N^\bu$ has finite projective dimension as a complex in the abelian
category $\sE=S\Modl$, that is $\ppd_SN^\bu<\infty$.
 This means that $N^\bu$ is quasi-isomorphic to a finite complex of
(infinitely generated) projective $S$\+modules.
 Therefore, the complex $\Hom_S(K^\bu,N^\bu)$ is also quasi-isomorphic
to a finite complex of projective $S$\+modules.
 As $\Hom_S(K^\bu,N^\bu)$ is also quasi-isomorphic to a bounded above
complex of finitely generated projective $S$\+modules by assumption,
it follows that $\Hom_S(K^\bu,N^\bu)$ is quasi-isomorphic to a finite
complex of finitely generated projective $S$\+modules.
\end{proof}

 Now we can establish the comparisons of our definition of a dedualizing
complex of torsion modules with the definitions in~\cite[Sections~4
and~5]{Pmgm}.

\begin{cor} \label{dedualizing-comparison-with-Pmgm-sect5-cor}
 Let $J$ be a weakly proregular finitely generated ideal in
a commutative ring $S$ and $B^\bu$ be a finite complex of $J$\+torsion
$S$\+modules.
 Then $B^\bu$ is a dedualizing complex of $J$\+torsion $S$\+modules
in the sense of the definition above if and only if $B^\bu$ is
a dedualizing complex for the ideal $J\subset S$ in the sense of
the definition in~\cite[Section~5]{Pmgm}.
\end{cor}

\begin{proof}
 By Lemma~\ref{torsion-and-abstract-homological-dimensions}(d),
finiteness of the projective dimension $\ppd_{(S,J)}B^\bu$ implies
finiteness of the contraflat dimension $\cfd_{(S,J)}B^\bu$.
 So the homological dimension condition~(i) above is equivalent to
the homological dimension condition~(i) from~\cite[Section~5]{Pmgm}.
 The homothety isomorphism condition~(iii) from
Section~\ref{auslander-and-bass-secn} above is equivalent to
the homothety isomorphism condition~(ii) from~\cite[Section~5]{Pmgm} in
view of Theorem~\ref{torsion-modules-inclusion-derived-fully-faithful}.

 Finally, by~\cite[Proposition~5.1]{Pmgm}, a complex of $S$\+modules
$C^\bu$ with $J$\+torsion cohomology modules is a compact object of
$\sD_{J\tors}(S\Modl)$ if and only if $C^\bu$ is quasi-isomorphic to
a finite complex of finitely generated projective $S$\+modules.
 In view of
Lemma~\ref{comparison-with-finiteness-condition-in-Pmgm-Sect5},
it follows that the finiteness condition~(ii) from
Section~\ref{auslander-and-bass-secn} above is equivalent to
the finiteness condition~(iii) from~\cite[Section~5]{Pmgm} under
the assumption of the homological dimension condition~(i).
\end{proof}

\begin{cor} \label{dedualizing-comparison-with-Pmgm-sect4-cor}
 Let $S$ be a Noetherian commutative ring and $J\subset S$ be an ideal
such that the quotient ring $S/J$ is Artinian.
 Let $B^\bu$ be a finite complex of $J$\+torsion $S$\+modules.
 Then $B^\bu$ is a dedualizing complex of $J$\+torsion $S$\+modules
in the sense of the definition above if and only if $B^\bu$ is
a dedualizing complex for the ideal $J\subset S$ in the sense of
the definition in~\cite[Section~4]{Pmgm}.
\end{cor}

\begin{proof}
 The homological dimension condition~(i) above coincides with
the homological dimension condition~(i) from~\cite[Section~4]{Pmgm}.
 The homothety isomorphism condition~(iii) from
Section~\ref{auslander-and-bass-secn} above is equivalent to
the homothety isomorphism condition~(ii) from~\cite[Section~4]{Pmgm} in
view of Theorem~\ref{torsion-modules-inclusion-derived-fully-faithful}
(we recall that all ideals in a Noetherian commutative ring are
weakly proregular; see Section~\ref{prelims-on-wpr-secn}).

 Finally, the finiteness condition~(ii) from
Section~\ref{auslander-and-bass-secn} above is equivalent to
the finiteness condition~(iii) from~\cite[Section~4]{Pmgm} by
Theorem~\ref{Artinian-quotient-ring-finitness-conditions-theorem}.
\end{proof}

 Now we can conclude that the definition of a dedualizing complex
from~\cite[Section~4]{Pmgm} agrees with the one
from~\cite[Section~5]{Pmgm}.
 This question was left open in the paper~\cite{Pmgm};
see~\cite[Remark~5.9]{Pmgm}.

\begin{cor} \label{dedualizing-comparison-of-Pmgm-sect4-and-sect5}
 Let $S$ be a Noetherian commutative ring and $J\subset S$ be an ideal
such that the quotient ring $S/J$ is Artinian.
 Let $B^\bu$ be a finite complex of $J$\+torsion $S$\+modules.
 Then $B^\bu$ is a dedualizing complex for the ideal $J\subset S$ in
the sense of the definition in~\cite[Section~4]{Pmgm} if and only if
$B^\bu$ is a dedualizing complex for the ideal $J\subset S$ in
the sense of the definition in~\cite[Section~5]{Pmgm}.
\end{cor}

\begin{proof}
 Compare Corollary~\ref{dedualizing-comparison-with-Pmgm-sect5-cor}
with Corollary~\ref{dedualizing-comparison-with-Pmgm-sect4-cor}.
\end{proof}

 Finally, we proceed to obtain the triangulated equivalences
of~\cite[Theorems~4.9 and~5.10]{Pmgm} as particular cases of
Theorem~\ref{bass-auslander-derived-equivalence} above.

 Let $B^\bu$ be a dedualizing complex of $J$\+torsion $S$\+modules
concentrated in the cohomological degrees $-d_1\le m\le d_2$.
 Let us choose the parameter~$l_1$ in such a way that both
the projective and the contraflat dimensions of $B^\bu$ do not
exceed~$l_1$, that is $\ppd_{(S,J)}B^\bu\le l_1$ and
$\cfd_{(S,J)}B^\bu\le l_1$ (this is possible by condition~(i)
and Lemma~\ref{torsion-and-abstract-homological-dimensions}(d)).
 One can see that any one of these two conditions implies $l_1\ge d_1$
if $H^{-d_1}(B^\bu)\ne0$ (take $M$ to be an injective cogenerator of
$S\Modl_{J\tors}$ or $P=\fS$ in the definitions of the projective and
contraflat dimensions).

\begin{lem} \label{dedualizing-bass-and-auslander-classes}
 Let $B^\bu$ be a dedualizing complex of $J$\+torsion $S$\+modules,
and let the parameter~$l_1$ be chosen as stated above.
 Then the related Bass and Auslander classes\/
$\sE_{l_1}=\sE_{l_1}(B^\bu)$ and\/ $\sF_{l_1}=\sF_{l_1}(B^\bu)$
coincide with the whole categories of $J$\+torsion $S$\+modules and
$J$\+contramodule $S$\+modules, $\sE_{l_1}=S\Modl_{J\tors}$ and\/
$\sF_{l_1}=S\Modl_{J\ctra}$.
\end{lem}

\begin{proof}
 In view of Lemma~\ref{abstract-corresponding-classes-adjunction-isoms}
and the subsequent discussion, it suffices to check that conditions
(I\+-IV) of Section~\ref{abstract-classes-secn} hold for the classes
$\sE=S\Modl_{J\tors}$ and $\sF=S\Modl_{J\ctra}$ with the given
parameter~$l_1$ and some integer $l_2\ge d_2$.
 Indeed, let us take $l_2=d_2$.
 Then conditions~(I\+-II) are obvious, and conditions~(III\+-IV) follow
from~(i) and Lemma~\ref{torsion-and-abstract-homological-dimensions}(d)
(or from the choice of~$l_1$).
\end{proof}

 It is clear from Lemma~\ref{dedualizing-bass-and-auslander-classes}
that for a dedualizing complex of $J$\+torsion $S$\+modules $B^\bu$
one has
$$
 \sD'_{B^\bu}(S\Modl_{J\tors})=\sD(S\Modl_{J\tors})
 \quad\text{and}\quad
 \sD''_{B^\bu}(S\Modl_{J\ctra})=\sD(S\Modl_{J\ctra}).
$$

\begin{cor}
 Let $J$ be a weakly proregular finitely generated ideal in
a commutative ring $S$, and let $B^\bu$ be a dedualizing complex of
$J$\+torsion $S$\+modules.
 Then, for any conventional or absolute derived category symbol\/
$\st=\bb$, $+$, $-$, $\varnothing$, $\abs+$, $\abs-$, or\/~$\abs$,
there is a triangulated equivalence\/ $\sD^\st(S\Modl_{J\tors})\simeq
\sD^\st(S\Modl_{J\ctra})$ provided by (appropriately defined)
mutually inverse derived functors\/ $\boR\Hom_S(L^\bu,{-})$ and
$L^\bu\ot_S^\boL\nobreak{-}$.
\end{cor}

\begin{proof}
 This is a restatement of~\cite[Theorem~5.10]{Pmgm} (in view of
Corollary~\ref{dedualizing-comparison-with-Pmgm-sect5-cor}),
a generalization of~\cite[Theorem~4.9]{Pmgm} (in view of
Corollary~\ref{dedualizing-comparison-with-Pmgm-sect4-cor}), and
a particular case of Theorem~\ref{bass-auslander-derived-equivalence} 
above (in view of Lemma~\ref{dedualizing-bass-and-auslander-classes}).
\end{proof}

\Section{Adically Coherent Rings and Coherent Torsion Modules}
\label{adically-coherent-secn}

 We start with a general ring-theoretic lemma~\cite[Lemma~1
and Theorem~2]{Harr}.

\begin{lem} \label{quotient-ring-coherence-lemma}
 Let $A$ be an associative ring and $I\subset A$ be a two-sided ideal.
 Assume that $I$ is finitely generated as a left $A$\+module.
 In this context: \par
\textup{(a)} a left $A/I$\+module is finitely presented if and only if
it is finitely presented as a left $A$\+module; \par
\textup{(b)} if the ring $A$ is left coherent, then so is
the ring~$A/I$.
\end{lem}

\begin{proof}
 Part~(a): let $M$ be a finitely presented left $A/I$\+module;
so $M$ is the cokernel of a morphism of finitely generated projective
left $A/I$\+modules $Q_1\rarrow Q_0$.
 Then both $Q_1$ and $Q_0$ are finitely presented left $A$\+modules
(since $A/I$ is a finitely presented left $A$\+module); so $M$ is
the cokernel of a morphism of finitely presented left $A$\+modules.
 Thus $M$ is finitely presented as a left $A$\+module.

 The converse implication does not depend on the assumtion that $I$
is finitely generated.
 Let $M$ be a left $A/I$\+module that is finitely presented as
a left $A$\+module.
 So $M$ is the cokernel of a morphism of finitely generated projective
left $A$\+modules $f\:P_1\rarrow P_0$.
 Then $M$ is also the cokernel of the morphism of finitely generated
projective left $A/I$\+modules $A/I\ot_Af\:A/I\ot_AP_1\rarrow
A/I\ot_AP_0$.

 Part~(b): let $J\subset A/I$ be a finitely generated left ideal.
 Then there is a finitely generated left ideal $K\subset A$ such that
$J=(K+I)/I$ (lift any finite set of generators of $J$ to some elements
of $A$ and generate the ideal $K$ by the resulting elements).
 Since $I$ is a finitely generated left $A$\+module by assumption,
the ideal $K+I\subset A$ is also finitely generated.
 By assumption, it follows that $K+I$ is a finitely presented
left $A$\+module.
 As $I$ is a finitely generated left $A$\+module, it follows that
$J=(K+I)/I$ is a finitely presented left $A$\+module.
 Using part~(a), we can conclude that $J$ is a finitely presented
left $A/I$\+module, as desired.
\end{proof}

\begin{rem} \label{nilpotent-noncoherence-counterex-remark}
 Let $I\subset A$ be a nilpotent two-sided ideal; so $I^n=0$ for some
$n\ge1$.
 Assume that $I$ is finitely generated as a left $A$\+module.
 Then one can easily see that the ring $A$ is left Noetherian
whenever the ring $A/I$ is left Noetherian.
 The analogous assertion for coherent rings is \emph{not} true.
 For example, let $k$~be a field and $k[s;x_1,x_2,\dotsc]$ be
the commutative ring of polynomials in a countable family of variables
$x_1$, $x_2$,~\dots\ and an additional variable~$s$ over~$k$.
 Let $A$ be the quotient ring of $k[s;x_1,x_2,\dotsc,]$ by
the ideal spanned by the elements $s^2$ and $sx_n$, \,$n\ge1$.
 Let $I=(s)\subset A$ be the principal ideal spanned by
the element~$s$.
 Then $I^2=0$ and $A/I=k[x_1,x_2,\dotsc]$ is the ring of polynomials
in the countable family of variables $x_1$, $x_2$,~\dots\ over~$k$;
so the ring $A/I$ is coherent.
 But the ring $A$ is \emph{not} coherent, since the ideal $I=(s)$ is
not finitely presented as an $A$\+module.
\end{rem}

 Let $\sT$ be a category with direct limits.
 We recall that an object $M\in\sT$ is said to be \emph{finitely
presented} if the functor $\Hom_\sT(M,{-})\:\sT\rarrow\mathsf{Sets}$
preserves direct limits~\cite[Definition~1.1]{AR}.
 The category $\sT$ is called \emph{locally finitely presentable} if
all colimits exist in $\sT$ and there is a set of finitely presented
objects $\sS\subset\sT$ such that all the objects of $\sT$ are
direct limits of objects from~$\sS$ \,\cite[Definition~1.9]{AR}.

\begin{lem} \label{torsion-modules-locally-finitely-presentable}
 Let $R$ be a commutative ring and $I\subset R$ be a finitely generated
ideal.
 Then the abelian category $R\Modl_{I\tors}$ is locally finitely
presentable.
 An $I$\+torsion $R$\+module $M$ is finitely presented as an object
of $R\Modl_{I\tors}$ if and only if it is finitely presented as
an object of $R\Modl$.
 Equivalently, $M$ is finitely presented in $R\Modl_{I\tors}$ if and
only if there exists an integer $n\ge1$ such that $I^nM=0$ \emph{and}
the $R/I^n$\+module $M$ is finitely presented.
\end{lem}

\begin{proof}
 Clearly, any finitely presented object of $R\Modl_{I\tors}$ must be
finitely generated as an $R$\+module; hence there exists $n\ge1$
such that $I^nM=0$.
 If this is the case, then Lemma~\ref{quotient-ring-coherence-lemma}(a)
says that $M$ is finitely presented over $R/I^n$ if and only if it is
finitely presented over~$R$.
 Any object $M\in R\Modl_{I\tors}$ that is finitely presented in
$R\Modl$ is also finitely presented in $R\Modl_{I\tors}$, since
the full subcategory $R\Modl_{I\tors}$ is closed under direct limits
in $R\Modl$.
 Similarly, any object of $R/I^n\Modl$ that is finitely presented in
$R\Modl_{I\tors}$ is also finitely presented in $R/I^n\Modl$, since
the full subcategory $R/I^n\Modl$ is closed under direct limits in
$R\Modl_{I\tors}$.
 This proves the second and third assertions of the lemma.
 It follows that representatives of isomorphism classes of finitely
presented objects form a set of generators in the abelian category
$R\Modl_{I\tors}$, hence the category $R\Modl_{I\tors}$ is locally
finitely presentable by~\cite[Theorem~1.11]{AR}.
 (Notice that, for abelian categories, there is no difference between
a set of generators and a set of strong generators.)
\end{proof}

 Let $R$ be a commutative ring and $I\subset R$ be a finitely generated
ideal.
 We will say that the ring $R$ is \emph{$I$\+adically coherent} if
the rings $R/I^n$ are coherent for all integers $n\ge1$.
 Clearly, an $I$\+adically coherent ring $R$ need not be coherent
(take $I=R$).
 Moreover, the counterexample in
Remark~\ref{nilpotent-noncoherence-counterex-remark} shows that
coherence of the ring $R/I$ does \emph{not} imply coherence of
the ring $R/I^2$.
 However, by Lemma~\ref{quotient-ring-coherence-lemma}(b), any
coherent ring $R$ is $I$\+adically coherent with respect to every
finitely generated ideal $I\subset R$.

\begin{cor} \label{radical-inclusion-adic-coherence}
 Let $R$ be a commutative ring and $I$, $J\subset R$ be two finitely
generated ideals such that $\sqrt{I}\subset\sqrt{J}$.
 Assume that the ring $R$ is $I$\+adically coherent.
 Then $R$ is also $J$\+adically coherent.
\end{cor}

\begin{proof}
 We have $I\subset\sqrt{I}\subset\sqrt{J}$.
 Since the ideal $I$ is finitely generated, it follows that there exists
$m\ge1$ for which $I^m\subset J$.
 Hence $I^{mn}\subset J^n$ for all $n\ge1$.
 Since the ring $A=R/I^{mn}$ is coherent by assumption and the ideal
$J^n/I^{mn}\subset A$ is finitely generated, it follows by virtue of
Lemma~\ref{quotient-ring-coherence-lemma}(b) that the ring $A/J^n$ is
coherent, too.
\end{proof}

 Similarly to the definition above, given a finitely generated ideal
$I\subset R$, let us say that the ring $R$ is \emph{$I$\+adically
Noetherian} if the ring $R/I$ is Noetherian.
 If this is the case, then all the rings $R/I^n$, \,$n\ge1$, are
Noetherian, too (see
Remark~\ref{nilpotent-noncoherence-counterex-remark}).
 Similarly to Corollary~\ref{radical-inclusion-adic-coherence},
if $\sqrt{I}\subset\sqrt{J}$ for finitely generated ideals $I$,
$J\subset R$ and $R$ is $I$\+adically Noetherian, then $R$ is
also $J$\+adically Noetherian.

 Weak proregularity of a finitely generated ideal $I\subset R$ does
not imply $I$\+adic coherence of $R$ (for example, the zero ideal in
any commutative ring is weakly proregular).
 The two properties are independent of each other:
the converse implication is not true, either, as the following
remark explains.

\begin{rem}
 All ideals in Noetherian commutative rings are weakly proregular.
 However, if the ring $R$ is $I$\+adically Noetherian, then the ideal
$I\subset R$ \emph{need not} be weakly proregular.
 It suffices to consider the case of a principal ideal $I=(s)\subset R$.

 Given an element $s\in R$, one says that the ring $R$ has \emph{bounded
$s$\+torsion} if there exists an integer $n_0\ge1$ such that $s^nr=0$
for $r\in R$ and $n\ge1$ implies $s^{n_0}r=0$.
 It is easy to see that the principal ideal $I=(s)$ is weakly proregular
in $R$ if and only if the $s$\+torsion in $R$ is bounded.
 Now let $k$~be a field, $S=k[s]$ be the ring of polynomials in one
variable~$s$ over~$k$, and $P=k[s,s^{-1}]/k[s]$ be the Pr\"ufer
$S$\+module.
 Consider the trivial extension $R=S\oplus P$.
 So $S$ is a subring in $R$, the product of any two elements from $S$
and $P$ in $R$ is given by the action of $S$ on $P$, and the product
of any two elements from $P$ in $R$ vanishes.
 Then the $s$\+torsion is not bounded in $R$; hence the ideal $I=(s)$ is
not weakly proregular in~$R$.
 However, the quotient ring $R/I$ is isomorphic to~$k$,
while the quotient ring $R/I^n$ is isomorphic to $k[s]/(s^n)$
for every $n\ge1$; all these quotient rings are Noetherian.

 To give another example, consider the ring $k[s;x_1,x_2,\dotsc,]$ as
in Remark~\ref{nilpotent-noncoherence-counterex-remark}; and let $R$
be the quotient ring of $k[s;x_1,x_2,\dotsc,]$ by the ideal spanned by
the elements $s^nx_n$, \,$n\ge1$.
 Then the $s$\+torsion in $R$ is not bounded; so the ideal $I=(s)$ is
not weakly proregular in~$R$.
 Still, for every $n\ge1$, the quotient ring $R/I^n$ is the ring of
polynomials in a countable set of variables $x_n$, $x_{n+1}$,~\dots\
over a commutative $k$\+algebra with a finite set of generators $s$,
$x_1$,~\dots,~$x_{n-1}$.
 So the ring $R/I^n$ is coherent.
\end{rem}

\begin{lem} \label{torsion-modules-locally-coherent}
 Let $R$ be a commutative ring and $I\subset R$ be a finitely
generated ideal such that the ring $R$ is $I$\+adically coherent.
 Then finitely presented objects form a set of generators of
$R\Modl_{I\tors}$, and the full subcategory of finitely presented
objects is closed under kernels, cokernels, and extensions
in $R\Modl_{I\tors}$.
 In other words, $R\Modl_{I\tors}$ is a locally coherent Grothendieck
category in the sense of\/~\cite[Section~2]{Roos},
\cite[Section~8.2]{PS5}.
\end{lem}

\begin{proof}
 Follows from Lemma~\ref{torsion-modules-locally-finitely-presentable}
together with the fact that the full subcategory of finitely
presented $R/I^n$\+modules is closed under kernels, cokernels, and
extensions in $R/I^n\Modl$ for every $n\ge1$.
\end{proof}

 Let $R$ be a commutative ring and $I\subset R$ be a finitely generated
ideal such that the ring $R$ is $I$\+adically coherent.
 We will say that an $I$\+torsion $R$\+module $M$ is \emph{coherent}
(as an $I$\+torsion $R$\+module) if $M$ is finitely generated and
every finitely generated submodule of $M$ is finitely presented as
a module over $R/I^n$ for some $n\ge1$.
 In view of Lemma~\ref{quotient-ring-coherence-lemma}(a),
\,$M$ is coherent as an $I$\+torsion $R$\+module if and only if it is
coherent as an $R$\+module.
 It follows from Lemma~\ref{torsion-modules-locally-coherent} that
an $I$\+torsion $R$\+module is coherent if and only if it is
finitely presented as an object of $R\Modl_{I\tors}$.

 The following definition is most useful in the $I$\+adically coherent
case, but makes sense for any finitely generated ideal $I$ in
a commutative ring~$R$.
 An $I$\+torsion $R$\+module $K$ is said to be \emph{fp\+injective}
(as an $I$\+torsion $R$\+module) if $\Ext^1_{R\Modl_{I\tors}}(M,K)=0$
for all finitely presented $I$\+torsion $R$\+modules~$M$.
 Clearly, all injective objects of $R\Modl_{I\tors}$ are
fp\+injective.
 Denote the class of fp\+injective $I$\+torsion $R$\+modules by
$R\Modl_{I\tors}^\fpinj\subset R\Modl_{I\tors}$.

 Specializing the previous definition to the case of a ring $R$
with the zero ideal $I=0$ (when all $R$\+modules are $I$\+torsion),
we obtain the classical concept of an \emph{fp\+injective
$R$\+module}~\cite{Sten}.
 Notice that an fp\+injective $I$\+torsion $R$\+module \emph{need not}
be fp\+injective as an $R$\+module.

\begin{lem} \label{fp-injective-torsion-modules-lemma}
 Let $R$ be a commutative ring and $I\subset R$ be a finitely generated
ideal such that the ring $R$ is $I$\+adically coherent.
 In this context: \par
\textup{(a)} An $I$\+torsion $R$\+module $K$ is fp\+injective if and
only if the functor $M\longmapsto\Hom_R(M,K)$ is exact on the abelian
category of finitely presented/coherent $I$\+torsion $R$\+modules~$M$.
\par
\textup{(b)} An $I$\+torsion $R$\+module $K$ is fp\+injective if and
only if\/ $\Ext^n_{R\Modl_{I\tors}}(M,K)=0$ for all finitely presented
$I$\+torsion $R$\+modules~$M$ and all $n\ge1$. \par
\textup{(c)} The full subcategory of fp\+injective objects is closed
under extensions, cokernels of injective morphisms, infinite direct
sums, and direct limits in $R\Modl_{I\tors}$. \par
\textup{(d)} For any finitely presented $I$\+torsion $R$\+module $M$,
the functor\/ $\Hom_R(M,{-})$ is exact on the exact category of
fp\+injective $I$\+torsion $R$\+modules. \par
\textup{(e)} An $I$\+torsion $R$\+module $K$ is fp\+injective if and
only if the $R/I^n$\+module\/ $\Hom_R(R/I^n,K)$ is fp\+injective
for every $n\ge1$.
\end{lem}

\begin{proof}
 The assertions~(a\+-c) hold for an arbitrary locally coherent
Grothendieck category $\sA$ in place of $\sA=R\Modl_{I\tors}$.
 See, e.~g., \cite[Appendix~B]{Sto2}.
 Part~(e) follows from part~(a), and part~(d) follows from
the definitions (while the existence of the inherited exact structure
on the full subcategory of fp\+injective $I$\+torsion $R$\+modules
follows from part~(c)).
 Another relevant reference is~\cite[Lemma~E.2.1]{Pcosh}.
\end{proof}

 In the case of a coherent ring $R$, it is clear from
Lemma~\ref{fp-injective-torsion-modules-lemma}(a) that the functor
$\Gamma_I\:R\Modl\rarrow R\Modl_{I\tors}$ takes fp\+injective
$R$\+modules to fp\+injective $I$\+torsion $R$\+modules.

 If the ring $R$ is $I$\+adically Noetherian, then the abelian
category $R\Modl_{I\tors}$ is locally Noetherian.
 In this case, the Noetherian objects of $R\Modl_{I\tors}$ are simply
the $I$\+torsion $R$\+modules that are finitely generated as
$R$\+modules.
 All finitely generated $I$\+torsion $R$\+modules are finitely
presented in this case, and all fp\+injective $I$\+torsion $R$\+modules
are injective (as objects of $R\Modl_{I\tors}$).

 The following lemma complements
Lemma~\ref{contraflat-contramodules-lemma}(a).
 Taken together, these two lemmas form a dual-analogous version of
Lemma~\ref{fp-injective-torsion-modules-lemma}(c).

\begin{lem} \label{products-of-contraflat-contramodules-lemma}
 Let $R$ be a commutative ring and $I\subset R$ be a finitely generated
ideal such that the ring $R$ is $I$\+adically coherent.
 Then the full subcategory of contraflat $I$\+contramodule $R$\+modules
is closed under infinite products in $R\Modl_{I\ctra}$.
\end{lem}

\begin{proof}
 The assertion holds because the functor $P\longmapsto P/I^nP\:
R\Modl\rarrow R/I^n\Modl$ preserves infinite products for all
$n\ge1$ (as the ideal $I^n\subset R$ is finitely generated), and
infinite products of flat $R/I^n$\+modules are flat $R/I^n$\+modules
(as the ring $R/I^n$ is coherent).
\end{proof}

\Section{Dualizing Complexes}  \label{dualizing-complexes-secn}

 We refer to Section~\ref{cors-of-derived-fullyf-thms-secn} for
the definition of \emph{contraflat $I$\+contramodule $R$\+modules}.
 The discussion of \emph{finitely presented $I$\+torsion $R$\+modules}
can be found in Section~\ref{adically-coherent-secn}.
 The following lemma is very general.

\begin{lem} \label{finitely-presented-hom-tensor-isomorphisms}
 Let $I$ be a finitely generated ideal in a commutative ring $R$,
and let $M$ be a finitely presented $I$\+torsion $R$\+module.
 In this context: \par
\textup{(a)} For any contraflat $I$\+contramodule $R$\+module $P$ and
any $I$\+torsion $R$\+module $K$, the natural map
$$
 \Hom_R(M,K)\ot_RP\lrarrow\Hom_R(M,\>K\ot_RP)
$$
is an isomorphism. \par
\textup{(b)} For any injective object $H$ of the category
$R\Modl_{I\tors}$ and any $I$\+torsion $R$\+module $K$, the natural map
$$
 M\ot_R\Hom_R(K,H)\lrarrow\Hom_R(\Hom_R(M,K),H)
$$
is an isomorphism.
\end{lem}

\begin{proof}
 Part~(a): let $m\ge1$ be an integer such that $M$ is
an $R/I^m$\+module.
 For every $n\ge m$, denote by $K_n\subset K$ the submodule of all
elements annihilated by $I^n$ in~$K$.
 Then we have $K\ot_RP=\varinjlim_{n\ge m}K_n\ot_RP$, and therefore
\begin{multline*}
 \Hom_R(M,\>K\ot_RP)\simeq
 \varinjlim\nolimits_{n\ge m}\Hom_R(M,\>K_n\ot_RP) \\ \simeq
 \varinjlim\nolimits_{n\ge m}\Hom_{R/I^n}(M,\>K_n\ot_{R/I^n} P/I^nP) \\
 \simeq\varinjlim\nolimits_{n\ge m}
 \bigl(\Hom_{R/I^n}(M,K_n)\ot_{R/I^n}P/I^nP\bigr) \\ \simeq
 \varinjlim\nolimits_{n\ge m}\bigl(\Hom_{R/I^n}(M,K_n)\ot_R P\bigr)
 \simeq\Hom_R(M,K)\ot_RP,
\end{multline*}
since $M$ is a finitely presented object of $R\Modl_{I\tors}$ and
$P/I^nP$ is a flat $R/I^n$\+module.

 Part~(b):  Following the discussion in
Section~\ref{prelims-on-wpr-secn}, there exists an injective $R$\+module
$G$ such that $H$ is a direct summand of $\Gamma_I(G)$.
 Notice that both $K$ and $\Hom_R(M,K)$ are $I$\+torsion
$R$\+modules; hence we have $\Hom_R(K,\Gamma_I(G))=\Hom_R(K,G)$
and $\Hom_R(\Hom_R(M,K),\Gamma_I(G))\simeq\Hom_R(\Hom_R(M,K),G)$.
 It remains to recall that $M$ is a finitely presented $R$\+module
by Lemma~\ref{torsion-modules-locally-finitely-presentable};
so the natural map
$$
 M\ot_R\Hom_R(K,G)\lrarrow\Hom_R(\Hom_R(M,K),G)
$$
is an isomorphism.
\end{proof}

 Let $R$ be a commutative ring and $I\subset R$ be a finitely generated
ideal such that the ring $R$ is $I$\+adically coherent (in the sense
of the definition in Section~\ref{adically-coherent-secn}).
 The definition of \emph{fp\+injective $I$\+torsion $R$\+modules} was
also given in Section~\ref{adically-coherent-secn}.

\begin{lem} \label{tensor-hom-fp-injective-contraflat-lemma}
\textup{(a)} Let $P$ be a contraflat $I$\+contramodule $R$\+module
and $K$ be an fp\+injective $I$\+torsion $R$\+module.
 Then the tensor product $K\ot_RP$ is an fp\+injective $I$\+torsion
$R$\+module. \par
\textup{(b)} Let $H$ be an injective object of $R\Modl_{I\tors}$ and
$K$ be an fp\+injective $I$\+torsion $R$\+module.
 Then the\/ $\Hom$ module\/ $\Hom_R(K,H)$ is a contraflat
$I$\+contramodule $R$\+module.
\end{lem}

\begin{proof}
 Part~(a): by Lemma~\ref{fp-injective-torsion-modules-lemma}(a),
we need to prove that the functor $M\longmapsto\Hom_R(M,\>K\ot_RP)$
is exact on the abelian category of finitely presented $I$\+torsion
$R$\+modules~$M$.
 By Lemma~\ref{finitely-presented-hom-tensor-isomorphisms}(a), we
have $\Hom_R(M,\>K\ot_RP)\simeq\Hom_R(M,K)\ot_RP$.
 It remains to point out that the functor $M\longmapsto\Hom_R(M,K)$
is exact on the category of finitely presented $I$\+torsion
$R$\+modules $M$, the $R$\+module $\Hom_R(M,K)$ is $I$\+torsion for
all such $R$\+modules $M$, and the functor ${-}\ot_RP$ is exact on
the category of $I$\+torsion $R$\+modules $R\Modl_{I\tors}$.

 Part~(b): All $I$\+torsion $R$\+modules are direct limits of
finitely presented $I$\+torsion $R$\+modules by
Lemma~\ref{torsion-modules-locally-finitely-presentable}.
 As the class of finitely presented $I$\+torsion $R$\+modules is closed
under kernels by Lemma~\ref{torsion-modules-locally-coherent},
it follows easily that all short exact sequences of $I$\+torsion
$R$\+modules are direct limits of short exact sequences of finitely
presented $I$\+torsion $R$\+modules.
 Thus it suffices to prove that the functor $M\longmapsto
M\ot_R\Hom_R(K,H)$ is exact on the abelian category of finitely
presented $I$\+torsion $R$\+modules~$M$.
 Alternatively, one can say that in order to prove that
the $R/I^n$\+module $\Hom_R(K,H)/I^n\Hom_R(K,H)$ is flat, it suffices
to show that the functor $M\longmapsto M\ot_R\Hom_R(K,H)$ is exact on
the abelian category of finitely presented $R/I^n$\+modules~$M$.

 By Lemma~\ref{finitely-presented-hom-tensor-isomorphisms}(b), we
have $M\ot_R\Hom_R(K,H)\simeq\Hom_R(\Hom_R(M,K),H)$.
 It remains to point out that the functor $M\longmapsto\Hom_R(M,K)$
is exact on the category of finitely presented $I$\+torsion
$R$\+modules $M$, the $R$\+module $\Hom_R(M,K)$ is $I$\+torsion for
all such $R$\+modules $M$, and the functor $\Hom_R({-},H)$ is exact on
the category of $I$\+torsion $R$\+modules $R\Modl_{I\tors}$.
\end{proof}

 We start with the definition of a dualizing complex of modules over
a coherent commutative ring~$A$.
 A complex of $A$\+modules $D^\bu_A$ is said to be a \emph{dualizing
complex} if the following three conditions are satisfied:
\begin{enumerate}
\renewcommand{\theenumi}{\roman{enumi}${}^\circ$}
\item the complex $D_A^\bu$ is quasi-isomorphic to a finite complex of
fp\+injective $A$\+modules;
\item the cohomology modules of the complex $D_A^\bu$ are finitely
presented $A$\+modules;
\item the homothety map $A\rarrow\Hom_{\sD(A\Modl)}(D^\bu,D^\bu[*])$
is an isomorphism of graded rings.
\end{enumerate}
 Here the complex $D^\bu_A$ is viewed as an object of the derived
category $\sD(A\Modl)$.

 Let $A$ be a commutative ring and $I\subset A$ be an ideal.
 The derived functor
$$
 \boR\Hom_A(A/I,{-})\:\sD(A\Modl)\lrarrow\sD(A/I\Modl)
$$
is constructed by applying the functor $\Hom_A(A/I,{-})\:\sK(A\Modl)
\rarrow\sK(A/I\Modl)$ to homotopy injective complexes of $A$\+modules.
 Similarly, the derived functor
$$
 A/I\ot_A^\boL{-}\,\:\sD(A\Modl)\lrarrow\sD(A/I\Modl)
$$
is constructed by applying the functor $A/I\ot_A{-}\,\:\sK(A\Modl)
\rarrow\sD(A/I\Modl)$ to homotopy flat complexes of $A$\+modules.

 Let $R$ be a commutative ring and $I\subset R$ be a weakly proregular
finitely generated ideal such that the ring $R$ is $I$\+adically
coherent.
 A \emph{dualizing complex of $I$\+torsion $R$\+modules} $L^\bu=D^\bu$
is a pseudo-dualizing complex (according to the definition in
Section~\ref{auslander-and-bass-secn}) satisfying the following
additional condition:
\begin{enumerate}
\renewcommand{\theenumi}{\roman{enumi}}
\item the complex $D^\bu$ is quasi-isomorphic to a finite complex of
fp\+injective $I$\+torsion $R$\+modules.
\end{enumerate}

 In order to prove the results below, we will have to assume that
all fp\+injective $I$\+torsion $R$\+modules have finite injective
dimensions as objects of $R\Modl_{I\tors}$.
 This assumption holds whenever there exists an integer $m\ge0$
such that every ideal in $R/I$ is generated by at most $\aleph_m$
elements~\cite[Lemma~E.2.2(a\+-b)]{Pcosh}.

 In some results, we will also have to assume that all contraflat
$I$\+contramodule $R$\+modules have finite projective dimensions
as objects of $R\Modl_{I\ctra}=R\Modl_{I\ctra}^\qs$.
 This assumption holds whenever all flat $R/I$\+modules have finite
projective dimensions~\cite[Lemma~E.2.2(c)]{Pcosh} (notice that
the projective dimensions of flat $R/I^n$\+modules do not exceed
the projective dimensions of flat $R/I$\+modules, since a flat
$R/I^n$\+module $F$ is projective whenever the $R/I$\+module $F/IF$
is projective).

 The following theorem establishes a comparison between the two
preceding definitions in the case when the ring $R$ is coherent
(rather than merely $I$\+adically coherent).

\begin{thm} \label{definitions-of-dualizing-complex-comparison}
 Let $R$ be a coherent commutative ring and $I\subset R$ be
a weakly proregular finitely generated ideal.
 Let $D^\bu$ be finite complex of $I$\+torsion $R$\+modules.
 Assume that all fp\+injective $I$\+torsion $R$\+modules have finite
injective dimensions in $R\Modl_{I\tors}$.
 Then the following conditions are equivalent:
\begin{enumerate}
\item $D^\bu$ is a dualizing complex of $I$\+torsion $R$\+modules;
\item for every integer $n\ge1$, the complex\/ $D_{R/I^n}^\bu=
\boR\Hom_R(R/I^n,D^\bu)$ is a dualizing complex of $R/I^n$\+modules;
\item for some integer $n\ge1$, the complex\/ $D_{R/I^n}^\bu=
\boR\Hom_R(R/I^n,D^\bu)$ is a dualizing complex of $R/I^n$\+modules.
\end{enumerate}
\end{thm}

 In particular, in the case of a Noetherian commutative ring $R$,
Theorem~\ref{definitions-of-dualizing-complex-comparison} tells us that
the definition of a dualizing complex of $I$\+torsion $R$\+modules
above in this section agrees with the one in~\cite[Section~D.1]{Pcosh}
(see also~\cite[Sections~D.5 and~E.2]{Pcosh} for generalizations).
 The proof of Theorem~\ref{definitions-of-dualizing-complex-comparison}
occupies most of the remaining part of
Section~\ref{dualizing-complexes-secn}.

\begin{prop} \label{definitions-of-dualizing-complex-homoldim}
 Let $R$ be a commutative ring and $I\subset R$ be a finitely generated
ideal such that the ring $R$ is $I$\+adically coherent.
 Let $G^\bu$ be a bounded below complex of fp\+injective $I$\+torsion
$R$\+modules.
 Then the following conditions are equivalent:
\begin{enumerate}
\item the complex $G^\bu$ is quasi-isomorphic to a finite complex of
fp\+injective $I$\+torsion $R$\+modules;
\item for every integer $n\ge1$ the complex of $R/I^n$\+modules\/
$\Hom_R(R/I^n,G^\bu)$ is quasi-isomorphic to a finite complex
of fp\+injective $R/I^n$\+modules concentrated in the cohomological
degrees\/~$\le j$, where a fixed integer~$j$ does not depend on~$n$;
\item the complex $G^\bu$ is cohomologically bounded and there exists
an integer $n\ge1$ for which the complex of $R/I^n$\+modules\/
$\Hom_R(R/I^n,G^\bu)$ is quasi-isomorphic to a finite complex
of fp\+injective $R/I^n$\+modules.
\end{enumerate}
\end{prop}

\begin{proof}
 The implications (1)~$\Longrightarrow$~(2) and
(1)~$\Longrightarrow$~(3) hold in view of
Lemma~\ref{fp-injective-torsion-modules-lemma}(c\+-e).

 (2)~$\Longrightarrow$~(3) It suffices to prove that $H^i(G^\bu)=0$
for $i>j$.
 Indeed, we have $H^i\Hom_R(R/I^n,G^\bu)=0$ for all $n\ge1$, and
$G^\bu\simeq\varinjlim_{n\ge1}\Hom_R(R/I^n,G^\bu)$.

 (3)~$\Longrightarrow$~(1)
 Let $j$~be an integer such that $H^i(G^\bu)=0$ for $i\ge j$ and
the complex of $R/I^n$\+modules $\Hom_R(R/I^n,G^\bu)$ is
quasi-isomorphic to a bounded below complex of fp\+injective
$R/I^n$\+modules concentrated in the cohomological degrees~$\le j$.
 Put $Z=\ker(G^j\to G^{j+1})$ and $N^t=G^{j+t}$ for all integers
$t\ge0$; so $0\rarrow Z\rarrow N^0\rarrow N^1\rarrow N^2\rarrow\dotsb$
is the shifted canonical truncation $(\tau_{\ge j}G^\bu)[j]$ of
the complex~$G^\bu$.
 Then $N^\bu$ is a coresolution of the $R$\+module $Z$ by fp\+injective
$I$\+torsion $R$\+modules $N^t$, and we need to prove that
the $I$\+torsion $R$\+module $Z$ is fp\+injective, too.
 After this is established, we will have a finite complex of
fp\+injective $I$\+torsion $R$\+modules $\dotsb\rarrow G^{j-2}\rarrow
G^{j-1}\rarrow Z\rarrow0$ quasi-isomorphic to~$G^\bu$.

 For any finitely presented $R/I^n$\+module $M$ and every $t\ge1$,
we have
\begin{multline*}
 \Ext^t_{R\Modl_{I\tors}}(M,Z)=
 H^t\Hom_R(M,N^\bu)=H^{j+t}\Hom_R(M,G^\bu) \\
 \simeq H^{j+t}\Hom_{R/I^n}(M,\Hom_R(R/I^n,G^\bu))=0,
\end{multline*}
since a quasi-isomorphism of the bounded below complex of
fp\+injective $R/I^n$\+mod\-ules $\Hom_R(R/I^n,G^\bu)$ with a bounded
below complex of fp\+injective $R/I^n$\+mod\-ules concentrated in
the cohomological degrees~$\le j$ is preserved by the functor
$\Hom_{R/I^n}(M,{-})$.
 The latter assertion holds by
Lemma~\ref{fp-injective-torsion-modules-lemma}(c\+-d) (applied to
the zero ideal in the coherent commutative ring~$R/I^n$).

 It remains to point out that every finitely presented $I$\+torsion
$R$\+module $L$ is a finitely iterated extension of finitely
presented $R/I$\+modules in the abelian category $R\Modl_{I\tors}$.
 Indeed, we have a short exact sequence $0\rarrow IL\rarrow L
\rarrow L/IL\rarrow0$ in $R\Modl_{I\tors}$, and the $R/I$\+module
$L/IL$ is finitely presented as the cokernel of a morphism from
a direct sum of a finite number of copies of $L$ into~$L$.
 In view of the coherence assumption on $R$, the kernel $IL$ of
the morphism $L\rarrow L/IL$ is a finitely presented $I$\+torsion
$R$\+module, too; and we can proceed by induction.
\end{proof}

 The following lemma is obvious.

\begin{lem} \label{finitely-presented-ext-lemma}
 Let $R$ be a coherent commutative ring and $I\subset R$ be
a finitely generated ideal.
 Let $V$ be an $R/I$\+module.
 Then the following conditions are equivalent:
\begin{enumerate}
\item the $R/I$\+module $V$ is finitely presented;
\item the $R/I$\+module\/ $\Ext_R^i(R/I,V)$ is finitely presented
for every $i\ge0$.
\end{enumerate}
\end{lem}

\begin{proof}
 The implication (2)~$\Longrightarrow$~(1) holds because one can take
$i=0$.
 The implication (1)~$\Longrightarrow$~(2) holds for any $R$\+module
$V$, not necessarily annihilated by~$I$, because the $R$\+module $R/I$
has a resolution by finitely generated projective $R$\+modules.
 (It is helpful to keep Lemma~\ref{quotient-ring-coherence-lemma}(a)
in mind.)
\end{proof}

\begin{prop} \label{definitions-of-dualizing-complex-finiteness}
 Let $R$ be a coherent commutative ring and $I\subset R$ be a finitely
generated ideal.
 Let $G^\bu$ be a finite complex of $R$\+modules.
 Then the following two conditions are equivalent:
\begin{enumerate}
\item the cohomology $R/I$\+modules of the complex of $R/I$\+modules\/
$\boR\Hom_R(R/I,G^\bu)$ are finitely presented; {\hfuzz=2pt\par}
\item for every finite complex of finitely generated projective
$R$\+modules $K^\bu$ with $I$\+torsion cohomology $R$\+modules,
the complex of $R$\+modules\/ $\Hom_R(K^\bu,G^\bu)$ is
quasi-isomorphic to a bounded above complex of finitely generated
projective $R$\+modules.
\end{enumerate}
\end{prop}

\begin{proof}
 Notice first of all that, over a coherent ring $R$, a bounded above
complex of $R$\+modules $V^\bu$ is quasi-isomorphic to a bounded
above complex of finitely generated projective $R$\+modules if and
only if the cohomology $R$\+modules of $V^\bu$ are finitely presented.
 So condition~(2) means simply that the cohomology $R$\+modules of
the complex $\Hom_R(K^\bu,G^\bu)$ are finitely presented.

 (2)~$\Longrightarrow$~(1)
 Let $\bs=(s_1,\dotsc,s_m)$ be a finite sequence of generators of
the ideal $I\subset R$ and $K^\bu=K^\bu(R,\bs)$ be the related dual
Koszul complex from Section~\ref{prelims-on-wpr-secn}.
 Then the complex of $R/I$\+modules $R/I\ot_RK^\bu$ has zero
differential, and the one-term complex $R/I$ is a direct summand of
$R/I\ot_RK^\bu$.
 Hence it suffices to show that the complex
$\boR\Hom_R(R/I\ot_RK^\bu,\>G^\bu)\simeq
\boR\Hom_R(R/I,\Hom_R(K^\bu,G^\bu))$ has finitely presented
$R/I$\+modules of cohomology.

 Now $V^\bu=\Hom_R(K^\bu,G^\bu)$ is a finite complex of $R$\+modules
with finitely presented cohomology $R$\+modules (by~(2)).
 Pick a resolution $P_\bu$ of the $R$\+module $R/I$ by finitely
generated projective $R$\+modules.
 Then the derived category object $\boR\Hom_R(R/I,V^\bu)\in
\sD^+(R/I\Modl)$, viewed as an object of the derived category of
$R$\+modules, is represented by the complex $\Hom_R(P_\bu,V^\bu)$.
 It follows that the cohomology $R/I$\+modules of the complex
$\Hom_R(R/I,V^\bu)$ are finitely presented as $R$\+modules.
 Using Lemma~\ref{quotient-ring-coherence-lemma}(a), we see that
these cohomology modules are also finitely presented over~$R/I$.

 (1)~$\Longrightarrow$~(2)
 By~\cite[Proposition~6.1]{BN} or~\cite[Proposition~6.6]{Rou}
(see also~\cite[Proposition~5.1 and proof of Lemma~5.4(a)]{Pmgm}),
it suffices to consider the case of the dual Koszul complex
$K^\bu(R,\bs)$ in the role of~$K^\bu$, as in the previous argument.
 Then all the elements of $I$ act on the complex $K^\bu$ by
endomorphisms homotopic to zero, so the cohomology $R$\+modules
of the complex $V^\bu=\Hom_R(K^\bu,G^\bu)$ are annihilated by~$I$.
 Furthermore, it follows from~(1) that the cohomology $R/I$\+modules
of the complex $\boR\Hom_R(R/I,\Hom_R(K^\bu,G^\bu))\simeq
\Hom_R(K^\bu,\boR\Hom_R(R/I,G^\bu))$ are finitely presented.
 We need to prove that the cohomology $R$\+modules of the complex
$V^\bu=\Hom_R(K^\bu,G^\bu)$ are finitely presented.

 For any bounded below complex of $R$\+modules $V^\bu$ with
the cohomology $R$\+modules annihilated by~$I$, the claim is that
the cohomology $R$\+modules of $V^\bu$ are finitely presented whenever
the cohomology $R/I$\+modules of $\boR\Hom_R(R/I,V^\bu)$ are
finitely presented.
 This is provable by increasing induction on the cohomological degree
using Lemma~\ref{finitely-presented-ext-lemma} and taking
Lemma~\ref{quotient-ring-coherence-lemma}(a) into account.
\end{proof}

 Given a finitely generated ideal $I$ in a commutative ring $R$,
we denote by $\fR=\varprojlim_{n\ge1}R/I^n$ the $I$\+adic completion
of the ring~$R$.

\begin{lem} \label{contraflat-contramodule-complex-nakayama}
 Let $I$ be a finitely generated ideal in a commutative ring $R$
and $F^\bu$ be a bounded above complex of contraflat quotseparated
$I$\+contramodule $R$\+modules.
 Then the complex $F^\bu$ is acyclic if and only if the complex of
$R/I$\+modules $F^\bu/IF^\bu$ is acyclic.
\end{lem}

\begin{proof}
 The ``only if'' implication follows from
Lemma~\ref{contraflat-contramodules-lemma}(a) and its proof.
 The ``if'' is essentially the result of~\cite[Corollary~0.3]{PSY2}.
 Specifically, let $i$~be the largest integer such that $F^i\ne0$.
 Then the map $F^{i-1}/IF^{i-1}\rarrow F^i/IF^i$ is surjective
by assumption.
 So, denoting by $P=H^i(F^\bu)$ the cokernel of the differential
$F^{i-1}\rarrow F^i$, we obtain a (quotseparated) $I$\+contramodule
$R$\+module $P$ such that $P=IP$.
 Using the fact that $sQ=Q$ implies $Q=0$ for an $s$\+contramodule
$R$\+module $Q$, and arguing by induction on the number of generators
of the ideal $I$, one proves that $P=0$.
 Hence the differential $F^{i-1}\rarrow F^i$ is surjective; denote its
kernel by ${}'\!F^{i-1}$.
 By Lemma~\ref{contraflat-contramodules-lemma}(a) and its proof,
${}'\!F^{i-1}$ is a contraflat quotseparated $I$\+contramodule
$R$\+module and the short sequence $0\rarrow R/I\ot_R{}'\!F^{i-1}
\rarrow R/I\ot_RF^{i-1}\rarrow R/I\ot_RF^i\rarrow0$ is exact.
 Replacing the complex $F^\bu$ with the complex $\dotsb\rarrow
F^{i-3}\rarrow F^{i-2}\rarrow{}'\!F^{i-1}\rarrow0$ and proceeding in
the same way, one proves the desired assertion by decreasing induction
on the cohomological degree.
\end{proof}

\begin{prop} \label{definitions-of-dualizing-complex-homothety}
 Let $R$ be a commutative ring and $I\subset R$ be a finitely generated
ideal such that the ring $R$ is $I$\+adically coherent.
 Let $G^\bu$ be a finite complex of fp\+injective $I$\+torsion
$R$\+modules, $H^\bu$ be a finite complex of injective $I$\+torsion
$R$\+modules, and $G^\bu\rarrow H^\bu$ be a morphism of complexes.
 Then the homothety map\/ $\fR\rarrow\Hom_R(G^\bu,H^\bu)$ is
a quasi-isomorphism of complexes of $R$\+modules if and only if
the homothety map
$R/I\rarrow\Hom_{R/I}(\Hom_R(R/I,G^\bu),\Hom_R(R/I,H^\bu))$
is a quasi-isomorphism of complexes of $R/I$\+modules.
\end{prop}

\begin{proof}
 By Lemmas~\ref{tors-contra-tensor-hom-lemma}(b)
and~\ref{tensor-hom-fp-injective-contraflat-lemma}(b),
$\Hom_R(G^\bu,H^\bu)$ is a finite complex of contraflat quotseparated
$I$\+contramodule $R$\+modules.
 So is the one-term complex~$\fR$.
 Furthermore, by
Lemma~\ref{finitely-presented-hom-tensor-isomorphisms}(b) we have
$\Hom_{R/I}(\Hom_R(R/I,G^\bu),\Hom_R(R/I,H^\bu))\simeq
R/I\ot_R\Hom_R(G^\bu,H^\bu)$.
 It remains to apply the result of
Lemma~\ref{contraflat-contramodule-complex-nakayama} to the cone of
the morphism of complexes $\fR\rarrow\Hom_R(G^\bu,H^\bu)$.
\end{proof}

\begin{proof}[Proof of
Theorem~\ref{definitions-of-dualizing-complex-comparison}]
 The coherence assumption on the ring $R$ obviously implies
the $I$\+adic coherence
(see Lemma~\ref{quotient-ring-coherence-lemma}(b)).
 For any weakly proregular finitely generated ideal $I$ in
a commutative ring $R$ such that the ring $R$ is $I$\+adically coherent,
and for any bounded below complex of fp\+injective $I$\+torsion
$R$\+modules $G^\bu$, the complex of $R/I^n$\+modules
$\Hom_R(R/I^n,G^\bu)$ represents the derived category object
$\boR\Hom_R(R/I^n,G^\bu)$.

 Indeed, following the definition above, in order to compute the derived
category object $\boR\Hom_R(R/I^n,G^\bu)\in\sD^+(R/I^n\Modl)$, one
needs to apply the functor $\Hom_R(R/I^n,{-})$ to a bounded below
complex of injective $R$\+modules $J^\bu$ quasi-isomorphic to~$G^\bu$.
 Let $H^\bu$ be a bounded below complex of injective objects in
$R\Modl_{I\tors}$ quasi-isomorphic to~$G^\bu$.
 Then we have quasi-isomorphisms of complexes of $R$\+modules
$G^\bu\rarrow H^\bu\rarrow J^\bu$.
 It follows from
Lemma~\ref{fp-injective-torsion-modules-lemma}(c\+-d) that
the induced map of complexes of $R/I^n$\+modules $\Hom_R(R/I^n,G^\bu)\rarrow\Hom_R(R/I^n,H^\bu)$ is a quasi-isomorphism.
 By Lemma~\ref{Hom-tensor-underived=derived-lemma}(a), the induced map
of complexes of $R/I^n$\+modules $\Hom_R(R/I^n,H^\bu)\rarrow
\Hom_R(R/I^n,J^\bu)$ is a quasi-isomorphism, too.

 Now let $G^\bu$ be a bounded below complex of fp\+injective
$I$\+torsion $R$\+modules quasi-isomorphic to~$D^\bu$.
 Then it is clear from
Proposition~\ref{definitions-of-dualizing-complex-homoldim} that
the complex $D^\bu$ is quasi-isomorphic to a finite complex of
fp\+injective $I$\+torsion $R$\+modules if and only if the complex
of $R/I^n$\+modules $\boR\Hom_R(R/I^n,D^\bu)$ is quasi-isomorphic to
a finite complex of fp\+injective $R/I^n$\+modules for some
(or equivalently, for all) $n\ge1$.
 So condition~(i) holds for $D^\bu$ if and only if
condition~(i${}^\circ$) holds for $\boR\Hom_R(R/I^n,D^\bu)$.

 Now assume that condition~(i) is satisfied for~$D^\bu$.
 Let $G^\bu$ be a finite complex of fp\+injective $I$\+torsion
$R$\+modules quasi-isomorphic to~$D^\bu$.
 Applying Proposition~\ref{definitions-of-dualizing-complex-finiteness}
(and replacing $I$ by $I^n$ if needed), we see that the complex $D^\bu$
satisfies condition~(ii) from Section~\ref{auslander-and-bass-secn}
if and only if the complex of $R/I^n$\+modules $\boR\Hom_R(R/I^n,D^\bu)$
satisfies condition~(ii${}^\circ$) for some (or equivalently, for all)
$n\ge1$.

 Finally, in order to compare conditions~(iii) and~(iii${}^\circ$),
we need the assumption that all fp\+injective $I$\+torsion
$R$\+modules have finite injective dimensions in $R\Modl_{I\tors}$.
 Let $H^\bu$ be a finite complex of injective $I$\+torsion $R$\+modules
endowed with a quasi-isomorphism $G^\bu\rarrow H^\bu$.
 Then Proposition~\ref{definitions-of-dualizing-complex-homothety}
(for the ideal $I^n$ in the role of~$I$) tells us that the complex
$D^\bu$ satisfies condition~(iii) from
Section~\ref{auslander-and-bass-secn} if and only if the complex
of $R/I^n$\+modules $\boR\Hom_R(R/I^n,D^\bu)$ satisfies
condition~(iii${}^\circ$) for some (or equivalently, for all) $n\ge1$.
\end{proof}

\begin{ex}
 Let $R$ be a coherent commutative ring and $I\subset R$ be a weakly
proregular finitely generated ideal.
 Let $D_R^\bu$ be a dualizing complex of $R$\+modules, and let
$G_R^\bu$ be a finite complex of fp\+injective $R$\+modules
quasi-isomorphic to $D_R^\bu$, as per condition~(i${}^\circ$).

 Then $D_{R/I}^\bu=G_{R/I}^\bu=\Hom_R(R/I,G_R^\bu)$ is a finite complex
of fp\+injective $R/I$\+modules.
 Condition~(ii${}^\circ$) is satisfied for $D_{R/I}^\bu$, since
the complex $D_{R/I}^\bu$ represents the derived category
$\boR\Hom_R(R/I,D_R^\bu)\in\sD^+(R/I\Modl)$, which, viewed as an object
of the derived category of $R$\+modules, can be also computed using
a resolution of $R/I$ by finitely generated projective $R$\+modules.
 Lemma~\ref{quotient-ring-coherence-lemma}(a) is helpful here.

 Assuming further that all fp\+injective $R$\+modules have finite
injective dimensions, condition~(iii${}^\circ$) is also satisfied
for~$D_{R/I}^\bu$.
 Indeed, let $H_R^\bu$ be a finite complex of injective $R$\+modules
endowed with a quasi-isomorphism $G_R^\bu\rarrow H_R^\bu$.
 Then $R\rarrow\Hom_R(G_R^\bu,H_R^\bu)$ is a quasi-isomorphism of
finite complexes of flat $R$\+modules (by
Lemma~\ref{tensor-hom-fp-injective-contraflat-lemma}(b) applied to
the zero ideal in~$R$), so it remains a quasi-isomorphism after
the functor $R/I\ot_R{-}$ is applied.
 It remains to point out the natural isomorphisms of complexes of
$R/I$\+modules $R/I\ot_R\Hom_R(G_R^\bu,H_R^\bu)\simeq
\Hom_{R/I}(\Hom_R(R/I,G_R^\bu),H_R^\bu)\simeq
\Hom_{R/I}(\Hom_R(R/I,G_R^\bu),\Hom_R(R/I,H_R^\bu))$.
 Thus $D^\bu_{R/I}$ is a dualizing complex of $R/I$\+modules.

 Now consider the finite complex of $I$\+torsion $R$\+modules
$D^\bu=\boR\Gamma_I(D_R^\bu)=\Gamma_I(G_R^\bu)$.
 By~\cite[Theorem~3.2(iii)]{Sch}, \cite[Corollary~4.26]{PSY},
or~\cite[Lemma~1.2(a)]{Pmgm}, the complex $D^\bu$ is quasi-isomorphic
to the tensor product $K^\bu_\infty(R,\bs)\ot_RD_R^\bu$, where $\bs$~is
any finite sequence of generators of the ideal $I\subset R$ and
$K^\bu_\infty(R,\bs)$ is the infinite dual Koszul complex from
Section~\ref{prelims-on-wpr-secn}.

 We claim that $D^\bu$ is a dualizing complex of $I$\+torsion
$R$\+modules, because it satisfies the condition of
Theorem~\ref{definitions-of-dualizing-complex-comparison}(2).
 In fact, $\Gamma_I(G_R^\bu)$ is a finite complex of fp\+injective
$I$\+torsion $R$\+modules as per the paragraph after
Lemma~\ref{fp-injective-torsion-modules-lemma}, so one has
$\boR\Hom_R(R/I^n,D^\bu)=\Hom_R(R/I^n,\Gamma_I(G_R^\bu))=
\Hom_R(R/I^n,G_R^\bu)=D_{R/I^n}^\bu$ in view of the argument in
the beginning of the proof of
Theorem~\ref{definitions-of-dualizing-complex-comparison}.
\end{ex}

\Section{Co-Contra Correspondence for a Dualizing Complex}

 Let $R$ be a commutative ring and $I\subset R$ be a weakly proregular
finitely generated ideal such that the ring $R$ is $I$\+adically
coherent, and let $D^\bu$ be a dualizing complex of $I$\+torsion
$R$\+modules.
 Let us choose the parameter~$l_2$ in such a way that $D^\bu$ is
quasi-isomorphic to a complex of fp\+injective $I$\+torsion $R$\+modules
concentrated in the cohomological degrees $-d_1\le m\le l_2$.

\begin{prop} \label{dualizing-complex-minimal-corresponding-classes}
 Let $R$ be a commutative ring and $I\subset R$ be a weakly proregular
finitely generated ideal such that the ring $R$ is $I$\+adically
coherent.
 Let $n\ge0$ be an integer such that the injective dimensions of
fp\+injective $I$\+torsion $R$\+modules (as objects of
$R\Modl_{I\tors}$) do not exceed~$n$.

 Let $L^\bu=D^\bu$ be a dualizing complex of $I$\+torsion $R$\+modules,
and let the parameter~$l_2$ be chosen as stated above.
 Then the related minimal corresponding classes\/
$\sE^{l_2}=\sE^{l_2}(D^\bu)$ and\/ $\sF^{l_2}=\sF^{l_2}(D^\bu)$ 
(defined in Section~\ref{minimal-classes-secn}) are contained
in the classes of fp\+injective $I$\+torsion $R$\+modules
and contraflat $I$\+contramodule $R$\+modules,
$\sE^{l_2}\subset R\Modl_{I\tors}^\fpinj$ and\/
$\sF^{l_2}\subset R\Modl_{I\ctra}^\ctrfl$.

 Moreover, the classes\/ $\sE=R\Modl_{I\tors}^\fpinj$ and\/
$\sF=R\Modl_{I\ctra}^\ctrfl$ satisfy conditions (I\+-IV)
from Section~\ref{abstract-classes-secn} with the parameters
$l_1=n+d_1$ and~$l_2$.
\end{prop}

\begin{proof}
 In view of Remark~\ref{minimal-classes-remark}, it suffices to prove
the moreover clause.
 Indeed, condition~(I) holds for $\sE=R\Modl_{I\tors}^\fpinj$ by
Lemma~\ref{fp-injective-torsion-modules-lemma}(c), and condition~(II)
holds for $\sF=R\Modl_{I\ctra}^\ctrfl$ by
Lemma~\ref{contraflat-contramodules-lemma}(a) and the paragraph
preceding it.

 To prove condition~(III), let $E\in\sE$ be an fp\+injective
$I$\+torsion $R$\+module.
 By assumption, there exists a finite injective coresolution
$0\rarrow E\rarrow H^0\rarrow H^1\rarrow\dotsb\rarrow H^n\rarrow0$
of $E$ in $R\Modl_{I\tors}$.
 Let ${}'\!D^\bu$ be a complex of fp\+injective $I$\+torsion
$R$\+modules concentrated in the cohomological degrees between
$-d_1$ and~$l_2$ and quasi-isomorphic to~$D^\bu$.
 By Lemma~\ref{Hom-tensor-underived=derived-lemma}(a), the complex
of $R$\+modules $\Hom_R({}'\!D^\bu,H^\bu)$ represents the derived
category object $\boR\Hom_R(D^\bu,E)\in\sD^+(R\Modl)$.
 Clearly, the complex $\Hom_R({}'\!D^\bu,H^\bu)$ is concentrated
in the cohomological degrees from $-l_2$ to $n+d_1$.
 By Lemma~\ref{tensor-hom-fp-injective-contraflat-lemma}(b),
\,$\Hom_R({}'\!D^\bu,H^\bu)$ is a complex of contraflat
$I$\+contramodule $R$\+modules.

 To prove condition~(IV), let $F\in\sF$ be a contraflat
$I$\+contramodule $R$\+module.
 By Corollary~\ref{contraflat-contramods-tensor-underived=derived-cor},
the complex of $R$\+modules ${}'\!D^\bu\ot_RF$ represents the derived
category object $D^\bu\ot_R^\boL F\in\sD^-(R\Modl)$.
 Clearly, the complex ${}'\!D^\bu\ot_RF$ is concentrated in
the cohomological degrees from $-d_1$ to~$l_2$.
 By Lemma~\ref{tensor-hom-fp-injective-contraflat-lemma}(a),
\,${}'\!D^\bu\ot_RF$ is a complex of fp\+injective $I$\+torsion
$R$\+modules.
\end{proof}

\begin{prop} \label{torsion-modules-coderived-categories-prop}
 Let $I$ be a finitely generated ideal in a commutative ring $R$ such
that the ring $R$ is $I$\+adically coherent.
 Then \par
\textup{(a)} For any coderived category symbol\/ $\st=\co$ or\/~$\bco$,
the inclusion of exact categories $R\Modl_{I\tors}^\fpinj\rarrow
R\Modl_{I\tors}$ induces a triangulated equivalence of the coderived
categories
$$
 \sD^\st(R\Modl_{I\tors}^\fpinj)\simeq\sD^\st(R\Modl_{I\tors}).
$$ \par
\textup{(b)} A complex in the exact category $R\Modl_{I\tors}^\fpinj$
is Becker-coacyclic if and only if it is acyclic, that is\/
$\Ac^\bco(R\Modl_{I\tors}^\fpinj)=\Ac(R\Modl_{I\tors}^\fpinj)$.
 So, for the exact category $R\Modl_{I\tors}^\fpinj$, the Becker
coderived category coincides with the derived category,
$$
 \sD^\bco(R\Modl_{I\tors}^\fpinj)=\sD(R\Modl_{I\tors}^\fpinj).
$$
 The inclusions of additive/exact/abelian categories
$R\Modl_{I\tors}^\inj\rarrow R\Modl_{I\tors}^\fpinj\rarrow
R\Modl_{I\tors}$ induce triangulated equivalences
$$
 \sK(R\Modl_{I\tors}^\inj)\simeq\sD(R\Modl_{I\tors}^\fpinj)
 =\sD^\bco(R\Modl_{I\tors}^\fpinj)\simeq\sD^\bco(R\Modl_{I\tors}).
$$ \par
\textup{(c)} If all fp\+injective $I$\+torsion $R$\+modules have finite
injective dimensions, then the classes of Positselski-coacyclic and
Becker-coacyclic complexes in the abelian category $R\Modl_{I\tors}$
coincide, that is\/
$\Ac^\co(R\Modl_{I\tors})=\Ac^\bco(R\Modl_{I\tors})$.
 The inclusion of additive/abelian categories $R\Modl_{I\tors}^\inj
\rarrow R\Modl_{I\tors}$ induces a triangulated equivalence between
the homotopy category and the coderived category
$$
 \sK(R\Modl_{I\tors}^\inj)\simeq
 \sD^\co(R\Modl_{I\tors})=\sD^\bco(R\Modl_{I\tors})
$$
in this case.
\end{prop}

\begin{proof}
 Part~(a): the case of the Positselski coderived category, $\st=\co$,
is an application of the dual version
of~\cite[Proposition~A.3.1(b)]{Pcosh}.
 (Notice that the full subcategory $R\Modl_{I\tors}^\fpinj$ is
closed under infinite direct sums in $R\Modl_{I\tors}$ by
Lemma~\ref{fp-injective-torsion-modules-lemma}(c).)
 The case of the Becker coderived category, $\st=\bco$, follows from
Theorem~\ref{becker-co-contra-derived-of-loc-pres-abelian}(a)
(for $\sA=R\Modl_{I\tors}$) together with the fact that the classes of
injective objects in the abelian category $R\Modl_{I\tors}$ and in
the exact category $R\Modl_{I\tors}^\fpinj$ coincide.

 The first assertion of part~(b) is a particular case (for
$\sA=R\Modl_{I\tors}$) of a result applicable to all locally coherent
Grothendieck categories~$\sA$; see~\cite[Proposition~6.11]{Sto2}.
 One only needs to point out that the classes of injective objects
in $R\Modl_{I\tors}^\fpinj$ and in $R\Modl_{I\tors}$ coincide
(as $R\Modl_{I\tors}^\fpinj$ is a coresolving subcategory closed
under direct summands in $R\Modl_{I\tors}$, by
Lemma~\ref{fp-injective-torsion-modules-lemma}(c)); so a complex
in $R\Modl_{I\tors}^\fpinj$ is Becker-coacyclic in
$R\Modl_{I\tors}^\fpinj$ if and only if it is Becker-coacyclic
in $R\Modl_{I\tors}$.
 The second assertion of part~(b) follows by part~(a) and
Theorem~\ref{becker-co-contra-derived-of-loc-pres-abelian}(a).
 Part~(c) is a particular case of~\cite[Theorem~B.7.7(a)]{Pcosh}.
\end{proof}

\begin{prop} \label{contramodules-contraderived-categories-prop}
 Let $I$ be a weakly proregular finitely generated ideal in
a commutative ring $R$ such that the ring $R$ is $I$\+adically coherent.
 Then \par
\textup{(a)} For any contraderived category symbol\/ $\st=\ctr$
or\/~$\bctr$, the inclusion of exact categories $R\Modl_{I\ctra}^\ctrfl
\rarrow R\Modl_{I\ctra}$ induces a triangulated equivalence of
the contraderived categories
$$
 \sD^\st(R\Modl_{I\ctra}^\ctrfl)\simeq\sD^\st(R\Modl_{I\ctra}).
$$ \par
\textup{(b)} A complex in the exact category $R\Modl_{I\ctra}^\ctrfl$
is Becker-contraacyclic if and only if it is acyclic, that is\/
$\Ac^\bctr(R\Modl_{I\ctra}^\ctrfl)=\Ac(R\Modl_{I\ctra}^\ctrfl)$.
 So, for the exact category $R\Modl_{I\ctra}^\ctrfl$, the Becker
contraderived category coincides with the derived category,
$$
 \sD^\bctr(R\Modl_{I\ctra}^\ctrfl)=\sD(R\Modl_{I\ctra}^\ctrfl).
$$
The inclusions of additive/exact/abelian categories
$R\Modl_{I\ctra}^\proj\rarrow R\Modl_{I\ctra}^\ctrfl\rarrow
R\Modl_{I\ctra}$ induce triangulated equivalences
$$
 \sK(R\Modl_{I\ctra}^\proj)\simeq\sD(R\Modl_{I\ctra}^\ctrfl)
 =\sD^\bctr(R\Modl_{I\ctra}^\ctrfl)\simeq\sD^\bctr(R\Modl_{I\ctra}).
$$ \par
\textup{(c)} If all contraflat $I$\+contramodule $R$\+modules have
finite projective dimensions, then the classes of
Positselski-contraacyclic and Becker-contraacyclic complexes in
the abelian category $R\Modl_{I\ctra}$ coincide, that is\/
$\Ac^\ctr(R\Modl_{I\ctra})=\Ac^\bctr(R\Modl_{I\ctra})$.
 The inclusion of additive/abelian categories $R\Modl_{I\ctra}^\proj
\rarrow R\Modl_{I\ctra}$ induces a triangulated equivalence between
the homotopy category and the contraderived category
$$
 \sK(R\Modl_{I\ctra}^\proj)\simeq
 \sD^\ctr(R\Modl_{I\ctra})=\sD^\bctr(R\Modl_{I\ctra})
$$
in this case.
\end{prop}

\begin{proof}
 The assumption of weak proregularity of the ideal $I$ is only used
in parts~(a\+-b), and only in order to claim that all $I$\+contramodule
$R$\+modules are quotseparated, $R\Modl_{I\ctra}^\qs=R\Modl_{I\ctra}$
(cf.~\cite[Corollary~3.7 and Remark~3.8]{Pdc}).
 Without the weak proregularity assumption, part~(c) holds
in both the contexts of $R\Modl_{I\ctra}$ and $R\Modl_{I\ctra}^\qs$,
while parts~(a) and~(b) hold for $R\Modl_{I\ctra}^\qs$.

 Part~(a): the case of the Positselski contraderived category,
$\st=\ctr$, is an application of~\cite[Proposition~A.3.1(b)]{Pcosh}.
 (Notice that the full subcategory $R\Modl_{I\ctra}^\ctrfl$ is
closed under infinite products in $R\Modl_{I\ctra}$ by
Lemma~\ref{products-of-contraflat-contramodules-lemma}.)
 The case of the Becker contraderived category, $\st=\bctr$, follows
from Theorem~\ref{becker-co-contra-derived-of-loc-pres-abelian}(b)
(for $\sB=R\Modl_{I\ctra}$) together with the fact that the classes of
projective objects in the abelian category $R\Modl_{I\ctra}$ and in
the exact category $R\Modl_{I\ctra}^\ctrfl$ coincide.

 The first assertion of part~(b) is a particular case
of~\cite[Theorems~5.1 and~6.1]{Pbc}, which are applicable in view
of~\cite[Proposition~1.5]{Pdc}.
 One only needs to point out that the classes of projective objects
in $R\Modl_{I\ctra}^\ctrfl$ and in $R\Modl_{I\ctra}$ coincide
(as $R\Modl_{I\ctra}^\ctrfl$ is a resolving subcategory closed
under direct summands in $R\Modl_{I\ctra}$, by
Lemma~\ref{contraflat-contramodules-lemma}(a)); so a complex
in $R\Modl_{I\ctra}^\ctrfl$ is Becker-contraacyclic in
$R\Modl_{I\ctra}^\ctrfl$ if and only if it is Becker-coacyclic
in $R\Modl_{I\ctra}$.
 The second assertion of part~(b) follows by part~(a) and
Theorem~\ref{becker-co-contra-derived-of-loc-pres-abelian}(b).
 Part~(c) is a particular case of~\cite[Theorem~B.7.7(b)]{Pcosh}.
\end{proof}

\begin{cor} \label{dualizing-complex-becker-co-contra-derived-equiv}
 Let $I$ be a weakly proregular finitely generated ideal in
a commutative ring $R$ such that the ring $R$ is $I$\+adically coherent.
 Assume that the injective dimensions of fp\+injective $I$\+torsion
$R$\+modules (as objects of $R\Modl_{I\tors}$) are finite.
 Let $L^\bu=D^\bu$ be a dualizing complex of $I$\+torsion $R$\+modules.
 Then there is a triangulated equivalence between the Becker coderived
and contraderived categories
$$
 \sD^\bco(R\Modl_{I\tors})\simeq\sD^\bctr(R\Modl_{I\ctra})
$$
provided by (appropriately defined) mutually inverse derived functors\/
$\boR\Hom_R(D^\bu,{-})$ and $D^\bu\ot_R^\boL{-}$.
\end{cor}

\begin{proof}
 By Proposition~\ref{dualizing-complex-minimal-corresponding-classes},
the pair of classes $\sE=R\Modl_{I\tors}^\fpinj$ and
$\sF=R\Modl_{I\ctra}^\ctrfl$ satisfies conditions (I\+-IV)
from Section~\ref{abstract-classes-secn} with suitable parameters $l_1$
and~$l_2$.
 By Proposition~\ref{torsion-modules-coderived-categories-prop}(a\+-b),
we have $\sD^\bco(R\Modl_{I\tors})\simeq\sD(R\Modl_{I\tors}^\fpinj)$.
 By
Proposition~\ref{contramodules-contraderived-categories-prop}(a\+-b),
we have $\sD^\bctr(R\Modl_{I\ctra})\simeq\sD(R\Modl_{I\ctra}^\ctrfl)$.
 Now the desired triangulated equivalence $\sD(\sE)\simeq\sD(\sF)$
is provided by Theorem~\ref{abstract-classes-derived-equivalence}.
\end{proof}

\begin{cor} \label{dualizing-complex-all-derived-equivs}
 Let $I$ be a weakly proregular finitely generated ideal in
a commutative ring $R$ such that the ring $R$ is $I$\+adically coherent.
 Assume that the injective dimensions of fp\+injective $I$\+torsion
$R$\+modules (as objects of $R\Modl_{I\tors}$) are finite, and
the projective dimensions of contraflat $I$\+contramodule $R$\+modules
(as objects of $R\Modl_{I\ctra}$) are finite.
 Let $L^\bu=D^\bu$ be a dualizing complex of $I$\+torsion $R$\+modules.
 Then there is a triangualted equivalence between the coderived and
contraderived categories
$$
 \sD^{\co=\bco}(R\Modl_{I\tors})\simeq\sD^{\ctr=\bctr}(R\Modl_{I\ctra})
$$
provided by (appropriately defined) mutually inverse derived functors\/
$\boR\Hom_R(D^\bu,{-})$ and $D^\bu\ot_R^\boL{-}$.
 Here the notation\/ $\co=\bco$ and\/ $\ctr=\bctr$ means that
the Positselski co/contraderived categories coincide with
the Becker ones in this case.

 Furthermore, there is a chain of triangulated equivalences
\begin{multline*}
 \sD^{\co=\bco}(R\Modl_{I\tors})\simeq
 \sD^{\abs=\varnothing}(R\Modl_{I\tors}^\fpinj)\simeq
 \sD^{\abs=\varnothing}(\sE^{l_2})\simeq\sK(R\Modl_{I\tors}^\inj) \\
 \simeq\sK(R\Modl_{I\ctra}^\proj)\simeq\sD^{\abs=\varnothing}(\sF^{l_2})
 \simeq\sD^{\abs=\varnothing}(R\Modl_{I\ctra}^\ctrfl)\simeq
 \sD^{\ctr=\bctr}(R\Modl_{I\ctra}).
\end{multline*}
 Here the notation\/ $\sD^{\abs=\varnothing}(\sT)$ means that
the absolute derived category concides with the conventional derived
category for an exact category\/~$\sT$.
 Moreover, for any conventional derived category symbol\/ $\st=\bb$,
$+$, $-$, or\/~$\varnothing$, there are triangulated equivalences
\begin{multline*}
 \sD^\st(R\Modl_{I\tors}^\fpinj)\simeq\sD^\st(\sE^{l_2})\simeq
 \sK^\st(R\Modl_{I\tors}^\inj) \\ \simeq\sK^\st(R\Modl_{I\ctra}^\proj)
 \simeq\sD^\st(\sF^{l_2})\simeq\sD^\st(R\Modl_{I\ctra}^\ctrfl).
\end{multline*}
\end{cor}

\begin{proof}
 Under the assumptions of the corollary, one has
$\sD^\co(R\Modl_{I\tors})=\sD^\bco(R\Modl_{I\tors})$ by
Proposition~\ref{torsion-modules-coderived-categories-prop}(c)
and $\sD^\ctr(R\Modl_{I\ctra})=\sD^\bctr(R\Modl_{I\ctra})$ by
Proposition~\ref{contramodules-contraderived-categories-prop}(c).
 So the first assertion follows from
Corollary~\ref{dualizing-complex-becker-co-contra-derived-equiv}.
{\hbadness=1750\par}

 The rest of the proof is very similar to that
of~\cite[Corollary~7.6]{Pps}.
 The exact categories $R\Modl_{I\tors}^\fpinj$ and
$R\Modl_{I\ctra}^\ctrfl$ have finite homological dimensions
by assumption.
 Hence so do their full subcategories $\sE^{l_2}$ and $\sF^{l_2}$
satisfying condition (I) or~(II).
 It follows easily (see, e.~g., \cite[Remark~2.1]{Psemi}
or~\cite[Theorem~B.7.6]{Pcosh}) that a complex in any one of these
exact categories is acyclic if and only if it is absolutely acyclic,
and that their conventional/absolute derived categories are
equivalent to the homotopy categories of complexes of injective
or projective objects.
 The same applies to the Becker coderived/contraderived categories,
and also to the Positselski coderived/contraderived categories of
those of these exact categories that happen to be closed under
infinite direct sums/products in their respective abelian categories
of torsion modules/contramodules.
 The same applies also to the bounded versions of the derived
and homotopy categories.

 Propositions~\ref{torsion-modules-coderived-categories-prop}(a,c)
and~\ref{contramodules-contraderived-categories-prop}(a,c)
provide the equivalences of the categories mentioned in the previous
paragraph with the coderived category $\sD^{\co=\bco}(R\Modl_{I\tors})$
or the contraderived category $\sD^{\ctr=\bctr}(R\Modl_{I\ctra})$.
 The equivalence $\sD^\st(\sE^{l_2})\simeq\sD^\st(\sF^{l_2})$ can be
obtained as a particular case of
Theorem~\ref{minimal-classes-derived-equivalence}.
\end{proof}

\Section{Quotflat Morphisms of Ring-Ideal Pairs}  \label{quotflat-secn}

 Let $I$ be a finitely generated ideal in a commutative ring $R$ and
$J$ be a finitely generated ideal in a commutative ring~$S$.
 Suppose that we are given a ring homomorphism $f\:R\rarrow S$ such
that $f(I)\subset J$.
 The aim of this section is to discuss a flatness condition on
a ring homomorphism~$f$ depending on the ideals $I$ and~$J$.
 We proceed to deduce applications to the preservation of
fp\+injectivity (of torsion modules) and contraflatness
(of contramodules) by the restriction of scalars.

\begin{lem} \label{quotflat-characterizations}
 In the setting above, the following two conditions are equivalent:
\begin{enumerate}
\item there exist descending sequences of finitely generated ideals\/
$\dotsb\subset I_{n+1}\subset I_n\subset\dotsb\subset R$ and\/
$\dotsb\subset J_{n+1}\subset J_n\subset\dotsb\subset S$, indexed by
the integers $n\ge1$, such that
\begin{itemize}
\item one has $I_n\subset I^n$ and $J_n\subset J^n$ for every $n\ge1$;
\item for every $n\ge1$ there exists $q\ge n$ such that
$I^q\subset I_n$ and $J^q\subset J_n$;
\item one has $f(I_n)\subset J_n$ for all $n\ge1$, and the ring
$S/J_n$ is a flat module over the ring $R/I_n$.
\end{itemize}
\item for every pair of integers $p\ge m\ge1$ and every finitely
generated ideal $I'\subset R$ such that $I^p\subset I'\subset I^m$,
there exists an integer $q\ge m$ and a finitely generated ideal
$J'\subset S$ such that $J^q\subset J'\subset J^m$, and the following
conditions hold:
\begin{itemize}
\item one has $f(I')\subset J'$, and the ring $S/J'$ is a flat module
over the ring $R/I'$.
\end{itemize}
\end{enumerate}
\end{lem}

\begin{proof}
 (1)~$\Longrightarrow$~(2)
 By~(1), we have $f(I_p)\subset J_p\subset J^p$, and the ring $S/J_p$
is a flat module over $R/I_p$.
 Put $J'=SI'+J_p$ (where $SI'$ denotes the ideal generated by $f(I')$
in~$S$).
 Then there exists $q\ge p$ such that $J^q\subset J_p$, hence
$J^q\subset J'$.
 Furthermore, $J'\subset SI^m+J^p\subset J^m$.
 Finally, we have $S/J'=R/I'\ot_{R/I_p}S/J_p$, so flatness of
$S/J_p$ as a module over $R/I_p$ implies flatness of $S/J'$ as
a module over $R/I'$.

 (2)~$\Longrightarrow$~(1)
 The construction of ideals $I_n\subset R$ and $J_n\subset S$
proceeds by induction on $n\ge1$.
 For $n=1$, put $p=m=1$ and $I'=I$.
 By~(2), there exists an integer $q\ge1$ and a finitely generated ideal
$J'\subset S$ such that $J^q\subset J'\subset J$, \ $f(I)\subset J'$,
and the ring $S/J'$ is a flat module over $R/I$.
 So we can put $I_1=I$ and $J_1=J'$.
 
 Suppose that we already have ideals $I_n\subset R$ and $J_n\subset S$
such that $I^q\subset I_n$ and $J^q\subset J_n$ for some $q\ge n$,
and the other conditions listed in~(1) are satisfied.
 Put $p=m=q+1$ and $I'=I^{q+1}$.
 By~(2), there exists an integer $q'\ge m$ and a finitely generated
ideal $J'\subset S$ such that $J^{q'}\subset J'\subset J^{q+1}$,
\ $f(I^{q+1})\subset J'$, and the ring $S/J'$ is a flat module
over $R/I^{q+1}$.
 So we can put $I_{n+1}=I^{q+1}$ and $J_{n+1}=J'$, and have
$I^{q'}\subset I_{n+1}\subset I^{n+1}$, \ $I_{n+1}\subset I_n$
and $J^{q'}\subset J_{n+1}\subset J^{n+1}$, \ $J_{n+1}\subset J_n$.
\end{proof}

 We will speak of the \emph{morphism of pairs} $f\:(R,I)\rarrow
(S,J)$, meaning that $f\:R\rarrow S$ is a ring homomorphism,
$I\subset R$ and $J\subset S$ are finitely generated ideals,
and $f(I)\subset J$.
 We say that a morphism of pairs $f\:(R,I)\rarrow(S,J)$ is
\emph{quotflat} if the equivalent conditions of
Lemma~\ref{quotflat-characterizations} are satisfied.

\begin{exs} \label{quotflat-and-nonquotflat-examples}
 (0)~Let $f\:R\rarrow S$ be a homomorphism of commutative rings such
that $S$ is a flat $R$\+module, and let $I\subset R$ be a finitely
generated ideal.
 Let $J=SI$ be the ideal generated by $f(I)$ in~$S$.
 Then the morphism of pairs $f\:(R,I)\rarrow(S,J)$ is quotflat.
 Indeed, put $I_n=I^n$ and $J_n=J^n=SI^n$ for all $n\ge1$.
 Then the ring $S/J^n=R/I^n\ot_RS$ is a flat module over $R/I^n$ for
every $n\ge1$, so the conditions of
Lemma~\ref{quotflat-characterizations}(1) are satisfied.

\smallskip
 (1)~Let $k$~be a field, $R=k[x]$ be the ring of polynomials in
one variable~$x$ over~$k$, and $S=k[x,y]$ be the ring of polynomials
in two variables $x$,~$y$.
 Let $f\:R\rarrow S$ be the natural inclusion map.
 Consider the maximal ideal $I=(x)\subset R$ and the maximal ideal
$J=(x,y)\subset R$.
 Then the ring $S/J^n$ is \emph{not} a flat module over $R/I^n$ when
$n\ge2$.
 For example, the ring $S/J^2$ is not a flat module over $R/I^2$,
since the coset $y+J^2\in S/J^2$ is annihilated by multliplication
with the coset $x+I^2\in R/I^2$, but $y+J^2$ is \emph{not} divisible by
$x+I^2$ in $S/J^2$.

 Nevertheless, the morphism of pairs $f\:(R,I)\rarrow(S,J)$ is
quotflat.
 Indeed, put $I_n=(x^n)\subset R$ and $J_n=(x^n,y^n)\subset S$.
 Then the ring $S/J_n=k[x]/(x^n)\ot_k k[y]/(y^n)$ is a flat module
over the ring $R/I_n=k[x]/(x^n)$.
 We also have $I_n=I^n$ and $J^{2n-1}\subset J_n\subset J^n$ for
every $n\ge1$, so the conditions of
Lemma~\ref{quotflat-characterizations}(1) are satisfied.

\smallskip
 (2)~In the notation of~(1), let $S=R=k[x]$ be the ring of polynomials
in one variable~$x$ over~$k$, and let $f\:R\rarrow S$ be the identity
map.
 Put $I=0\subset R$ and $J=(x)\subset S$.
 Then the morphism of pairs $f\:(R,I)\rarrow(S,J)$ is \emph{not}
quotflat.
 Indeed, for any integers $q\ge n\ge 1$ and any ideals $I'\subset R$
and $J'\subset S$ such that $I^q\subset I'\subset I^n$ and $J^q\subset
J'\subset J^n$, the Artinian ring $S/J'$ is \emph{not} a flat module
over the polynomial ring $R/I'=R$.

 Let us emphasize that the ring $S$ is a flat $R$\+module in this
example.
 Moreover, the $J$\+adic completion $k[[x]]=\fS=
\varprojlim_{n\ge1}S/J^n$ of the ring $S$ is a flat module over
the $I$\+adic completion $k[x]=\fR=\varprojlim_{n\ge1}R/I^n$ of
the ring $R$, or in other words, one can say that $\fS$ is a contraflat
$I$\+contramodule $R$\+module.
 Nevertheless, the morphism of pairs is not quotflat.
\end{exs}

 The following proposition explains what we need the notion of
a quotflat morphism for.

\begin{prop} \label{restriction-of-scalars-fp-injective}
 Let $I$ be a finitely generated ideal in a commutative ring $R$, let
$J$ be a finitely generated ideal in a commutative ring $S$, and let
$f\:R\rarrow S$ be a ring homomorphism such that $f(I)\subset J$.
 Assume that the morphism of pairs $f\:(R,I)\rarrow(S,J)$ is
quotflat and the ring $R$ is $I$\+adically coherent.
 Then every fp\+injective $J$\+torsion $S$\+module is also fp\+injective
as an $I$\+torsion $R$\+module.
\end{prop}

\begin{proof}
 Let $I_n\subset R$ and $J_n\subset S$ be descending sequences of
ideals as in Lemma~\ref{quotflat-characterizations}(1).
 Let $H$ be an fp\+injective $J$\+torsion $S$\+module.
 According to Lemma~\ref{fp-injective-torsion-modules-lemma}(a),
we need to prove that the functor $M\longmapsto\Hom_R(M,H)$ is exact
on the abelian category of finitely presented $I$\+torsion
$R$\+modules~$M$.

 Notice that coherence of the rings $R/I^q$ for all $q\ge1$ implies
coherence of the rings $R/I_n$ for all $n\ge1$ by
Lemma~\ref{quotient-ring-coherence-lemma}(b).
 Any short exact sequence of finitely presented $I$\+torsion
$R$\+modules is a short exact sequence of finitely presented
$R/I_m$\+modules for some $m\ge1$; so it suffices to let $M$ range
over the abelian category of finitely presented $R/I_m$\+modules.
 In this case, we have
\begin{multline*}
 \Hom_R(M,H)\simeq
 \Hom_R\bigl(M,\,\varinjlim\nolimits_{n\ge1}\Hom_S(S/J_n,H)\bigr) \\
 \simeq\varinjlim\nolimits_{n\ge1}\Hom_R(M,\Hom_S(S/J_n,H))
 \simeq\varinjlim\nolimits_{n\ge m}\Hom_{R/I_n}(M,\Hom_S(S/J_n,H)) \\
 \simeq\varinjlim\nolimits_{n\ge m}
 \Hom_{S/J_n}(S/J_n\ot_{R/I_n}M,\>\Hom_S(S/J_n,H)) \\
 \simeq\varinjlim\nolimits_{n\ge m}
 \Hom_S(S/J_n\ot_{R/I_n}M,\>H),
\end{multline*}
where the second isomorphism holds by
Lemma~\ref{torsion-modules-locally-finitely-presentable}.
 Now, for any finitely presented $R/I_n$\+module $M$,
the $S/J_n$\+module $S/J_n\ot_{R/I_n}M$ is finitely presented;
hence $S/J_n\ot_{R/I_n}M$ is also finitely presented as
a ($J$\+torsion) $S$\+module.
 Since $S/J_n$ is flat as a module over $R/I_n$ and $H$ is
fp\+injective as a $J$\+torsion $S$\+module, it is clear that
the functor $M\longmapsto\Hom_S(S/J_n\ot_{R/I_n}M,\>H)$ is exact
on the category of finitely presented $R/I_m$\+modules $M$ for
all $n\ge m$.
\end{proof}

 In the rest of this section, our aim is to prove the following
dual-analogous version of
Proposition~\ref{restriction-of-scalars-fp-injective}.

\begin{prop} \label{restriction-of-scalars-contraflat}
 Let $I$ be a weakly proregular finitely generated ideal in
a commutative ring $R$, let $J$ be a finitely generated ideal in
a commutative ring $S$, and let $f\:R\rarrow S$ be a ring homomorphism
such that $f(I)\subset J$.
 Assume that the morphism of pairs $f\:(R,I)\rarrow(S,J)$ is
quotflat and the ring $R$ is $I$\+adically coherent.
 Then every contraflat quotseparated $J$\+contramodule $S$\+module is
also contraflat as an $I$\+contramodule $R$\+module.
\end{prop}

 The proof of Proposition~\ref{restriction-of-scalars-contraflat} is
based on a sequence of lemmas.

\begin{lem} \label{tor-with-product-of-contramodules}
 Let $I$ be a weakly proregular finitely generated ideal in
a commutative ring $R$ such that the ring $R$ is $I$\+adically coherent.
 Let\/ $\Xi$ be an indexing set and $(P_\xi)_{\xi\in\Xi}$ be a family
of $I$\+contramodule $R$\+modules.
 Let $M$ be a finitely presented $I$\+torsion $R$\+module.
 Then the natural map
$$
 \Tor_i^R\left(M,\>\prod\nolimits_{\xi\in\Xi}P_\xi\right)\lrarrow
 \prod\nolimits_{\xi\in\Xi}\Tor_i^R(M,P_\xi)
$$
is an isomorphism for all $i\ge0$.
\end{lem}

\begin{proof}
 For every index $\xi\in\Xi$, choose a resolution $F_{\xi,\bu}\rarrow
P_\xi$ of the $I$\+contramodule $R$\+module $P_\xi$ by contraflat
$I$\+contramodule $R$\+modules $F_{\xi,i}$, \,$i\ge0$.
 (For example, any projective resolutions in $R\Modl_{I\ctra}$
are suitable.)
 Then $\prod_{\xi\in\Xi}F_{\xi,\bu}\rarrow\prod_{\xi\in\Xi}P_\xi$
is a resolution of the $I$\+contramodule $R$\+module
$\prod_{\xi\in\Xi}P_\xi$, and by
Lemma~\ref{products-of-contraflat-contramodules-lemma}
the $I$\+contramodule $R$\+modules $\prod_{\xi\in\Xi}F_{\xi,i}$
are contraflat for all $i\ge0$.
 By Corollary~\ref{contraflat-contramods-tensor-underived=derived-cor},
the complex $M\ot_R F_{\xi,\bu}$ computes $\Tor_*^R(M,P_\xi)$ for
every $\xi\in\Xi$, and the complex $M\ot_R\prod_{\xi\in\Xi}F_{\xi,\bu}$
computes $\Tor_*^R\bigl(M,\>\prod_{\xi\in\Xi}P_\xi\bigr)$.
 It remains to point out that $M$ is a finitely presented $R$\+module
by Lemma~\ref{torsion-modules-locally-finitely-presentable}, so
the functor $M\ot_R{-}$ preserves infinite products of $R$\+modules.
\end{proof}

\begin{lem} \label{Tor1-pro-zero-lemma}
 Let $I$ be an ideal in a commutative ring $R$, and let $M$ be
an $R/I^m$\+mod\-ule for some integer $m\ge1$.
 Then the projective system of\/ $\Tor_1^R$ modules
$$
 (\Tor^R_1(M,R/I^n))_{n\ge1}
$$
is pro-zero (in the sense of the definition in
Section~\ref{prelims-on-wpr-secn}).
\end{lem}

\begin{proof}
 One can immediately see from the homological $\Tor_*^R$ sequence
induced by the short exact sequence of $R$\+modules $0\rarrow I^n
\rarrow R\rarrow R/I^n\rarrow0$ that $\Tor^R_1(M,R/I^n)$ is
the kernel of the natural map $M\ot_R I^n\rarrow M$.
 Clearly, any projective subsystem (projective system of subobjects)
of a pro-zero projective system is pro-zero; so it suffices to check
that the projective system $(M\ot_RI^n)_{n\ge1}$ is pro-zero.
 Indeed, we have $M\ot_RI^n\simeq M\ot_{R/I^m}(R/I^m\ot_R I^n)
\simeq M\ot_{R/I^m}(I^n/I^{n+m})$, and it remains to point out
that the map $I^k/I^{k+m}\rarrow I^j/I^{j+m}$ vanishes for all
positive integers $j$ and~$k$ such that $k\ge j+m$.
\end{proof}

\begin{lem} \label{derived-projlim-vanishing-lemma}
 Let $I$ be an ideal in a commutative ring $R$ and\/ $\dotsb\subset
I_{n+1}\subset I_n\subset\dotsb\subset R$ be a descending sequence of
ideals, indexed by the integers $n\ge1$, such that for every $n\ge1$
there exists an integer $q\ge n$ for which $I^q\subset I_n\subset I^n$.
 Let $(G_n)_{n\ge1}$ be a projective system of $R$\+modules such
that $G_n$ is a flat $R/I_n$\+module for every $n\ge1$.
 Let $M$ be an $R/I^m$\+module for some integer $m\ge1$.
 Then one has
$$
 \varprojlim\nolimits_{n\ge1}\Tor_1^R(M,G_n)=0=
 \varprojlim\nolimits^1_{n\ge1}\Tor_1^R(M,G_n),
$$
where\/ $\varprojlim^1_{n\ge1}$ denotes the first derived functor of
projective limit.
\end{lem}

\begin{proof}
 The point is that for any $R$\+module $N$ one has
$N\ot_RG_n\simeq(N\ot_RR/I_n)\ot_{R/I_n}G_n$, and the functor
${-}\ot_{R/I_n}G_n$ is exact.
 Therefore, there is a natural isomorphism
$\Tor_i^R(N,G_n)\simeq\Tor_i^R(N,R/I_n)\ot_{R/I_n}G_n$ for all $i\ge0$.
 Returning to the situation at hand, the projective system
$(\Tor_1^R(M,R/I_n))_{n\ge1}$ is pro-zero by 
Lemma~\ref{Tor1-pro-zero-lemma}, and it follows that the projective
system $(\Tor_1^R(N,R/I_n)\ot_{R/I_n}G_n)_{n\ge1}$ is pro-zero, too.
 For a pro-zero projective system (of abelian groups or $R$\+modules
indexed by nonnegative integers), both the underived and the first
derived projective limits vanish.
\end{proof}

\begin{lem} \label{tor-computing-telescope-type-sequence}
 Let $I$ be a weakly proregular finitely generated ideal in
a commutative ring $R$, let $J$ be a finitely generated ideal in
a commutative ring $S$, and let $f\:R\rarrow S$ be a ring homomorphism
such that $f(I)\subset J$.
 Assume that the morphism of pairs $f\:(R,I)\rarrow(S,J)$ is quotflat,
and let $I_n\subset R$ and $J_n\subset S$ be descending sequences
of ideals as in Lemma~\ref{quotflat-characterizations}(1).
 Let $F$ be a contraflat quotseparated $J$\+contramodule $S$\+module.
 Then, for every finitely presented $I$\+torsion $R$\+module $M$ and
every integer $m\ge1$, there is a natural short exact sequence
of $R$\+modules
$$
 0\lrarrow M\ot_RF\lrarrow
 \prod\nolimits_{n\ge m}M\ot_R F/J_nF\lrarrow
 \prod\nolimits_{n\ge m}M\ot_R F/J_nF\lrarrow0.
$$
\end{lem}

\begin{proof}
 By a very general result of~\cite[Lemmas~8.1 and~8.3]{Pflcc}, which is
applicable in view of~\cite[Proposition~1.5]{Pdc}, all contraflat
quotseparated $J$\+contramodule $S$\+modules are $J$\+adically separated
(and complete), so we have $F\simeq\varprojlim_{n\ge m}F/J_nF$.
 As the transition maps $F/J_{n+1}F\rarrow F/J_nF$ are surjective for
all $n\ge1$, we have the telescope short exact sequence of $R$\+modules
\begin{equation} \label{telescope-sequence}
 0\lrarrow F\lrarrow\prod\nolimits_{n\ge m}F/J_nF\lrarrow
 \prod\nolimits_{n\ge m}F/J_nF\lrarrow0.
\end{equation}
 Applying the derived functor $\Tor_*^R(M,{-})$
to~\eqref{telescope-sequence}, we obtain a long exact sequence
of $R$\+modules
\begin{multline} \label{Tor-into-telescope-sequence}
 \dotsb\lrarrow\Tor_1^R\left(M,\>\prod\nolimits_{n\ge m}F/J_nF\right)
 \lrarrow\Tor_1^R\left(M,\>\prod\nolimits_{n\ge m}F/J_nF\right) \\
 \lrarrow M\ot_RF\lrarrow M\ot_R\prod\nolimits_{n\ge m}F/J_nF
 \lrarrow M\ot_R\prod\nolimits_{n\ge m}F/J_nF\lrarrow0.
\end{multline}

 By Lemma~\ref{tor-with-product-of-contramodules}, we have
$\Tor_i^R\bigl(M,\>\prod_{n\ge m}F/J_nF)\simeq
\prod_{n\ge m}\Tor_i^R(M,F/J_nF)$ for all $i\ge0$ (notice that
every $S/J_n$\+module is a $J$\+contramodule $S$\+module, hence also
an $I$\+contramodule $R$\+module).
 So the sequence~\eqref{Tor-into-telescope-sequence} takes the form
\begin{multline} \label{Tor-into-telescope-sequence-rewritten}
 \dotsb\lrarrow\prod\nolimits_{n\ge m}\Tor_1^R(M,F/J_nF)
 \lrarrow\prod\nolimits_{n\ge m}\Tor_1^R(M,F/J_nF) \\
 \lrarrow M\ot_RF\lrarrow\prod\nolimits_{n\ge m}M\ot_RF/J_nF
 \lrarrow\prod\nolimits_{n\ge m}M\ot_RF/J_nF\lrarrow0.
\end{multline}

 It remains to show that the map
\begin{equation} \label{id-minus-shift-map-for-Tor1}
 \prod\nolimits_{n\ge m}\Tor_1^R(M,F/J_nF)
 \lrarrow\prod\nolimits_{n\ge m}\Tor_1^R(M,F/J_nF)
\end{equation}
in the first line of~\eqref{Tor-into-telescope-sequence-rewritten}
is surjective.
 In fact, we will see that~\eqref{id-minus-shift-map-for-Tor1}
is an isomorphism.
 Indeed, the kernel and cokernel of~\eqref{id-minus-shift-map-for-Tor1}
are the projective limit and the first derived projective limit
$$
 \varprojlim\nolimits_{n\ge m}\Tor_1^R(M,F/J_nF)
 \quad\text{and}\quad
 \varprojlim\nolimits^1_{n\ge m}\Tor_1^R(M,F/J_nF),
$$
respectively.
 It is important for us to show that
$\varprojlim\nolimits^1_{n\ge m}\Tor_1^R(M,F/J_nF)=0$.
 Here it suffices to notice that $F/J_nF$ is a flat $S/J_n$\+module
by the contraflatness assumption on~$F$.
 Since $S/J_n$ is a flat $R/I_n$\+module by one of the conditions in
Lemma~\ref{quotflat-characterizations}(1), it follows that
$F/J_nF$ is a flat $R/I_n$\+module.
 As a finitely presented $I$\+torsion $R$\+module, $M$ is
an $R/I^k$\+module for some $k\ge1$.
 Thus Lemma~\ref{derived-projlim-vanishing-lemma} is applicable.
\end{proof}

\begin{proof}[Proof of
Proposition~\ref{restriction-of-scalars-contraflat}]
 Let $I_n\subset R$ and $J_n\subset S$ be descending sequences
of ideals as in Lemma~\ref{quotflat-characterizations}(1), and
let $F$ be a contraflat quotseparated $J$\+contramodule $S$\+module.
 Similarly to the proof of
Lemma~\ref{tensor-hom-fp-injective-contraflat-lemma}(b), it suffices
to show that the functor $M\longmapsto M\ot_RF$ is exact on
the abelian category of finitely presented $R/I_m$\+modules for
every $m\ge1$.
 We use the result of
Lemma~\ref{tor-computing-telescope-type-sequence}.
 For every $n\ge m$, we have $M\ot_R F/J_nF\simeq M\ot_{R/I_n}F/J_nF$,
which is an exact functor of $M\in R/I_m\Modl$ since $F/J_nF$ is
a flat $R/I_n$\+module (as explained in the proof of
Lemma~\ref{tor-computing-telescope-type-sequence}).
 It remains to point out that the infinite products of $R$\+modules
are exact functors, and the kernel of a surjective morphism is
an exact functor.
\end{proof}

\begin{rems} \label{alternative-assumptions-remark}
 The assumptions in
Propositions~\ref{restriction-of-scalars-fp-injective}
and~\ref{restriction-of-scalars-contraflat} are sufficient for
the conclusions.
 They are certainly \emph{not} necessary.
 Other sets of sufficient assumptions exist.
 Let us describe two of them.

\smallskip
 (1)~Let $f\:R\rarrow S$ be a homomorphism of commutative rings such
that $S$ is a flat $R$\+module, let $I\subset R$ be a finitely generated
ideal, and let $J=SI$ be the ideal generated by $f(I)$ in~$S$
(as in Example~\ref{quotflat-and-nonquotflat-examples}(0)).
 Then all fp\+injective $J$\+torsion $S$\+modules are also fp\+injective
as $I$\+torsion $R$\+modules, and all contraflat $J$\+contramodule
$S$\+modules are contraflat as $I$\+contramodule $R$\+modules.

 Indeed, if $J=SI$, then the functor of extension of scalars
$S\ot_R{-}\,\:R\Modl\rarrow S\Modl$ takes $I$\+torsion $R$\+modules
to $J$\+torsion $S$\+modules.
 So we have an exact functor $S\ot_R{-}\,\:R\Modl_{I\tors}\rarrow
S\Modl_{J\tors}$ left adjoint to the exact functor of restriction of
scalars $S\Modl_{J\tors}\rarrow R\Modl_{I\tors}$.
 Therefore, for any $I$\+torsion $R$\+module $M$ and any $J$\+torsion
$S$\+module $H$, there is a natural isomorphism of Ext modules
$\Ext^i_{R\Modl_{I\tors}}(M,H)\simeq
\Ext^i_{S\Modl_{J\tors}}(S\ot_RM,\>H)$ for all $i\ge0$.
 As the functor $S\ot_R{-}$ also takes finitely presented ($I$\+torsion)
$R$\+modules to finitely presented ($J$\+torsion) $S$\+modules, it
follows that $H$ is fp\+injective as an $I$\+torsion $R$\+module
whenever it is fp\+injective as a $J$\+torsion $S$\+module.

 Dual-analogously, if $J=SI$, then one has $J^n=SI^n$ for every
$n\ge1$, and it follows that $R/I^n\ot_RF\simeq S/J^n\ot_SF$ for
any $S$\+module~$F$.
 Since $S$ is a flat $R$\+module, the ring $S/J^n$ is a flat module
over $R/I^n$ by Example~\ref{quotflat-and-nonquotflat-examples}(0).
 So the $R/I^n$\+module $F/I^nF$ is flat whenever the $S/J^n$\+module
$F/J^nF$ is flat.
 Thus $F$ is contraflat as an $I$\+contramodule $R$\+module whenever
it is contraflat as a $J$\+contramodule $S$\+module.

\smallskip
 (2)~Let $f\:R\rarrow S$ be a homomorphism of commutative rings such
that $S$ is a flat $R$\+module.
 Let $I\subset R$ and $J\subset S$ be finitely generated ideals such
that $f(I)\subset J$.
 Assume that the ring $S$ is Noetherian.
 Notice that the morphism of pairs $f\:(R,I)\rarrow(S,J)$
\emph{need not} be quotflat in this case,
as Example~\ref{quotflat-and-nonquotflat-examples}(2) illustrates.
 Nevertheless, all fp\+injective $J$\+torsion $S$\+modules are again
fp\+injective as $I$\+torsion $R$\+modules, and all contraflat
$J$\+contramodule $S$\+modules are contraflat as $I$\+contramodule
$R$\+modules.

 Indeed, if the ring $S$ is $J$\+adically Noetherian, then the classes
of injective and fp\+injective $J$\+torsion $S$\+modules coincide
(see Section~\ref{adically-coherent-secn}).
 If, moreover, the ring $S$ is Noetherian, then a $J$\+torsion
$S$\+module $H$ is injective in $S\Modl_{J\tors}$ if and only if it is
injective in $S\Modl$ (see
Remark~\ref{Bass-Auslander-direct-sum-product-closedness-remark}).
 If this is the case, then $H$ is also injective in $R\Modl$, since
$S$ is a flat $R$\+module.
 As $H$ is $I$\+torsion as an $R$\+module, it follows that $H$ is
injective in $R\Modl_{I\tors}$.

 Dual-analogously, if the ring $S$ is Noetherian, then
a $J$\+contramodule $S$\+module $F$ is contraflat if and only if $F$
is a flat $S$\+module~\cite[Corollary~10.3(a)]{Pcta}.
 If this is the case, then $F$ is also a flat $R$\+module, since $S$
is a flat $R$\+module.
 As $F$ is an $I$\+contramodule $R$\+module, it follows that $F$ is
a contraflat $I$\+contramodule $R$\+module.

 In the rest of this paper, whenever the quotflatness assumption is
invoked, it is only used in order to refer to the results of
Propositions~\ref{restriction-of-scalars-fp-injective}
and~\ref{restriction-of-scalars-contraflat}.
 So, if one is willing to assume that $S$ is Noetherian (and flat
as an $R$\+module) instead, then the quotflatness assumption
can be dropped.
\end{rems}

\Section{Relative Context and Base Change}  \label{base-change-secn}

 Let $I$ be a weakly proregular finitely generated ideal in
a commutative ring $R$ and $J$ be a weakly proregular finitely
generated ideal in a commutative ring~$S$.
 Suppose that we are given a morphism of pairs $f\:(R,I)\rarrow(S,J)$,
i.~e., a ring homomorphism $f\:R\rarrow S$ such that $f(I)\subset J$.
 Assume that $S$ is a flat $R$\+module.

 Let $L^\bu$ be a pseudo-dualizing complex of $I$\+torsion $R$\+modules
(in the sense of the definition in
Section~\ref{auslander-and-bass-secn}).
 Let $\bs$~be a finite sequence of generators of the ideal $J\subset S$,
and let $K^\bu_\infty(S,\bs)$ be the infinite dual Koszul complex from
Section~\ref{prelims-on-wpr-secn}.
 Recall that $K^\bu_\infty(S,\bs)$ is a finite complex of countably
presented flat $S$\+modules.

 We are interested in the finite complex of $S$\+modules
$K^\bu_\infty(S,\bs)\ot_RL^\bu$.
 Denoting by $SI\subset S$ the ideal generated by $f(I)$ in $S$,
we notice that both $S\ot_RL^\bu$ and $K^\bu_\infty(S,\bs)\ot_RL^\bu$
are finite complexes of $SI$\+torsion $S$\+modules.
 The first key observation is that $K^\bu_\infty(S,\bs)\ot_RL^\bu$
is quasi-isomorphic to a finite complex of $J$\+torsion $S$\+modules
$U^\bu$, which is defined uniquely up to a quasi-isomorphism of
finite complexes in $S\Modl_{J\tors}$.

 Indeed, by
Theorem~\ref{torsion-modules-inclusion-derived-fully-faithful},
the functor $\sD^\bb(S\Modl_{J\tors})\rarrow\sD^\bb(S\Modl)$ is fully
faithful, and its essential image consists of all the bounded complexes
of $S$\+modules with $J$\+torsion cohomology modules.
 By~\cite[Lemma~1.1(a)]{Pmgm}, the cohomology $S$\+modules of
the complex $K^\bu_\infty(S,\bs)\ot_SM^\bu$ are $J$\+torsion for
any complex of $S$\+modules~$M^\bu$.
 Returning to the situation at hand, it follows that there is
a uniquely defined object $U^\bu\in\sD^\bb(S\Modl_{J\tors})$ isomorphic
to $K^\bu_\infty(S,\bs)\ot_RL^\bu$ in $\sD^\bb(S\Modl)$.

\begin{thm} \label{base-change-is-pseudo-dualizing}
 The finite complex of $J$\+torsion $S$\+modules $U^\bu$ is
a pseudo-dualizing complex of $J$\+torsion $S$\+modules.
\end{thm}

\begin{proof}
 As mentioned in Section~\ref{prelims-on-wpr-secn}, the complex
$K^\bu_\infty(S,\bs)$, viewed up to quasi-isomorphism, does not
depend on the choice of a finite sequence~$\bs$ of generators of
the ideal $J\subset S$.
 More specifically, for any two finite sequences $\bs'$ and~$\bs''$
of generators of the ideal $J\subset S$, the related two complexes
$K^\bu_\infty(S,\bs')$ and $K^\bu_\infty(S,\bs'')$ are naturally
connected by a chain of quasi-isomorphisms.
 These are quasi-isomorphisms of finite complexes of flat $S$\+modules,
so the functor ${-}\ot_RL^\bu$ takes them to quasi-isomorphisms.
 Thus we can choose the finite sequence~$\bs$ as we prefer.
 Pick a finite sequence $\br=(r_1,\dotsc,r_l)$ of generators of
the ideal $I\subset R$, and put $\bs=(\bt,f(\br))$, where $\bt$~is some
finite sequence of elements of the ideal $J\subset S$ such that
the finite sequence~$\bs$ generates~$J$.
 Here the notation is $f(\br)=(f(r_1),\dotsc,f(r_l))$.

 Let us prove condition~(ii) from
Section~\ref{auslander-and-bass-secn} for the complex~$U^\bu$.
 Similarly to the proof of
Proposition~\ref{definitions-of-dualizing-complex-finiteness},
by~\cite[Proposition~6.1]{BN} or~\cite[Proposition~6.6]{Rou}
(see also~\cite[Proposition~5.1 and proof of Lemma~5.4(a)]{Pmgm}),
it suffices to consider the case of the dual Koszul complex
$K^\bu(S,\bs)$ in the role of~$K^\bu$.
 Then the complex of $S$\+modules $\Hom_S(K^\bu(S,\bs),U^\bu)$ is
quasi-isomorphic to $K_\bu(S,\bs)\ot_SK^\bu_\infty(S,\bs)\ot_RL^\bu$.

 As mentioned in the proof of
Lemma~\ref{torsion-and-abstract-homological-dimensions},
following, e.~g., the discussion in~\cite[Section~2]{Pdc},
the functor $K^\bu_\infty(S,\bs)\ot_S{-}\,\:\sD(S\Modl)\rarrow
\sD_{J\tors}(S\Modl)$ is right adjoint to the inclusion functor
$\sD_{J\tors}(S\Modl)\rarrow\sD(S\Modl)$, where
$\sD_{J\tors}(S\Modl)$ denotes the full subcategory of complexes with
$J$\+torsion cohomology modules in $\sD(S\Modl)$.
 The complex of $S$\+modules $K_\bu(S,\bs)$ has $J$\+torsion cohomology
modules, hence the natural morphism of complexes
$K^\bu_\infty(S,\bs)\ot_S K_\bu(S,\bs)\rarrow K_\bu(S,\bs)$ is
a quasi-isomorphism.
 So the complex of $S$\+modules $\Hom_S(K^\bu(S,\bs),U^\bu)$ is
quasi-isomorphic to $K_\bu(S,\bs)\ot_RL^\bu$.

 Finally, we have an isomorphism of complexes
$K_\bu(S,\bs)\simeq K_\bu(S,\bt)\ot_RK_\bu(R,\br)$.
 By assumption, condition~(ii) holds for the complex of $I$\+torsion
$R$\+modules $L^\bu$, so the complex of $R$\+modules
$K_\bu(R,\br)\ot_RL^\bu$ is quasi-isomorphic to a bounded above
complex of finitely generated projective $R$\+modules~$P^\bu$.
 It follows that the complex of $S$\+modules
$\Hom_S(K^\bu(S,\bs),U^\bu)$ is quasi-isomorphic to the bounded
above complex of finitely generated projective $S$\+modules
$K_\bu(S,\bt)\ot_RP^\bu$.

 Let us prove condition~(iii) for~$U^\bu$.
 For this purpose, it is convenient to use the complexes
$T^\bu(S,\bs)$ and $T^\bu_n(S,\bs)$, \,$n\ge1$, from
Section~\ref{prelims-on-wpr-secn}.
 We compute
\begin{multline} \label{RHom-U-U-computation-started}
 \boR\Hom_S(U^\bu,U^\bu)\simeq
 \boR\Hom_S(T^\bu(S,\bs)\ot_RL^\bu,\>T^\bu(S,\bs)\ot_RL^\bu) \\
 \simeq\boR\varprojlim\nolimits_{n\ge1}
 \boR\Hom_S(T^\bu_n(S,\bs)\ot_RL^\bu,\>T^\bu(S,\bs)\ot_RL^\bu),
\end{multline}
where $\boR\varprojlim_{n\ge1}$ denotes the derived functor of
projective limit of projective systems of complexes of $S$\+modules
indexed by the linearly ordered set of positive integers.

 Using the notation above, we have $T^\bu_n(S,\bs)\simeq
T^\bu_n(S,\bt)\ot_ST^\bu_n(S,\br)\simeq
T^\bu_n(S,\bt)\ot_RT^\bu_n(R,\br)$, where $T^\bu_n(S,\bt)$ is a finite
complex of finitely generated free $S$\+modules and $T^\bu_n(R,\br)$
is a finite complex of finitely generated free $R$\+modules.
 Following Section~\ref{prelims-on-wpr-secn} and the argument above,
the complex $T^\bu_n(R,\br)\ot_RL^\bu$ is homotopy equivalent to
the complex $K_\bu(R,\br^n)\ot_RL^\bu$, which is quasi-isomorphic
to a bounded above complex of finitely generated projective
$R$\+modules~$P_n^\bu$.
 Hence we have
\begin{multline} \label{RHom-U-U-computation-intermediate-step}
 \boR\Hom_S(T^\bu_n(S,\bs)\ot_RL^\bu,\>T^\bu(S,\bs)\ot_RL^\bu) \\
 =\Hom_S(T^\bu_n(S,\bt)\ot_RP_n^\bu,\>T^\bu(S,\bs)\ot_RL^\bu) \\
 \simeq\Hom_S(T^\bu_n(S,\bt),T^\bu(S,\bs))\ot_R\Hom_R(P_n^\bu,L^\bu) \\
 = \Hom_S(T^\bu_n(S,\bt),T^\bu(S,\bs))\ot_R
 \boR\Hom_R\bigl(T^\bu_n(R,\br)\ot_RL^\bu,\>L^\bu\bigr) \\
 \simeq \Hom_S(T^\bu_n(S,\bt),T^\bu(S,\bs))\ot_R
 \Hom_R\bigl(T^\bu_n(R,\br),\boR\Hom_R(L^\bu,L^\bu)\bigr) \\
 \simeq \Hom_S(T^\bu_n(S,\bt),T^\bu(S,\bs))\ot_R
 \Hom_R(T^\bu_n(R,\br),\fR) \\ \simeq
 \Hom_S\bigl(T^\bu_n(S,\bt)\ot_RT^\bu_n(R,\br),
 \>T^\bu(S,\bs)\ot_R\fR\bigr) \\
 \simeq\Hom_S\bigl(T^\bu_n(S,\bs),\>T^\bu(S,\bs)\ot_R\fR\bigr),
\end{multline}
where $\fR=\varprojlim_{m\ge1}R/I^m$.

 Now we have isomorphisms of complexes of $S$\+modules
$T^\bu(S,\bs)\simeq T^\bu(S,\bt)\ot_ST^\bu(S,\br)\simeq
T^\bu(S,\bt)\ot_RT^\bu(R,\br)$.
 By~\cite[Lemma~5.3(b)]{Pmgm}, the completion map $R\rarrow\fR$
induces a quasi-isomorphism of complexes of $R$\+modules
$T^\bu(R,\br)\rarrow T^\bu(R,\br)\ot_R\fR$.
 Hence the same completion map also induces a quasi-isomorphism of
complexes of $S$\+modules $T^\bu(S,\bs)\rarrow T^\bu(S,\bs)\ot_R\fR$.
 So, using~\eqref{RHom-U-U-computation-intermediate-step}, we can
finish the computation~\eqref{RHom-U-U-computation-started} as
\begin{multline}
 \boR\varprojlim\nolimits_{n\ge1}
 \boR\Hom_S\bigl(T^\bu_n(S,\bs)\ot_RL^\bu,\>T^\bu(S,\bs)\ot_RL^\bu\bigr)
 \\ \simeq
 \boR\varprojlim\nolimits_{n\ge1}\Hom_S(T^\bu_n(S,\bs),T^\bu(S,\bs))
 =\Hom_S(T^\bu(S,\bs),T^\bu(S,\bs)).
\end{multline}
 Finally, the completion map $S\rarrow\fS$ induces a quasi-isomorphism
of complexes of $S$\+modules $T^\bu(S,\bs)\rarrow T^\bu(S,\bs)\ot_S\fS$
by~\cite[Lemma~5.3(b)]{Pmgm}, and it follows that the homothety morphism
$\fS\rarrow\Hom_S(T^\bu(S,\bs),T^\bu(S,\bs))$ is an isomorphism in
$\sD(S\Modl)$ by~\cite[Lemma~5.2(b)]{Pmgm}.
\end{proof}

 Assume that the complex $L^\bu$ is concentrated in the cohomological
degrees $-d_1\le m\le d_2$ and the complex $U^\bu$ is concentrated in
the cohomological degrees $-t_1\le m\le t_2$ (where $d_1$, $d_2$, $t_1$,
$t_2$ are some integers).
 The definitions of the Bass and Auslander classes $\sE_{l_1}$ and
$\sF_{l_1}$ can be found in Section~\ref{auslander-and-bass-secn}.

 The following proposition is our version
of~\cite[Proposition~8.5]{Pps}.
 Notice the difference between their formulation, however:
 The assertions of~\cite[Proposition~8.5]{Pps} are ``if and only if''
results, while the assertions of our proposition are only implications
in one direction.

\begin{prop} \label{base-change-bass-auslander-classes}
 Let $l_1$~be an integer such that $l_1\ge d_1$ and $l_1\ge t_1$.
 Then \par
\textup{(a)} a $J$\+torsion $S$\+module belongs to the full
subcategory\/ $\sE_{l_1}(U^\bu)\subset S\Modl_{J\tors}$ whenever
its underlying $I$\+torsion $R$\+module belongs to the full
subcategory\/ $\sE_{l_1}(L^\bu)\subset R\Modl_{I\tors}$; \par
\textup{(b)} a $J$\+contramodule $S$\+module belongs to the full
subcategory\/ $\sF_{l_1}(U^\bu)\subset S\Modl_{J\ctra}$ whenever
its underlying $I$\+contramodule $R$\+module belongs to the full
subcategory\/ $\sF_{l_1}(L^\bu)\subset R\Modl_{I\ctra}$. \hbadness=1550
\end{prop}

\begin{proof}
 Part~(a): for any $S$\+module $E$, we have
$$
 \boR\Hom_S(U^\bu,E)\simeq\Hom_S(T^\bu(S,\bs),\boR\Hom_R(L^\bu,E)).
$$
 Now if $H^n\boR\Hom_R(L^\bu,E)=0$ for $n>l_1$, then
$H^n\boR\Hom_S(U^\bu,E)=0$ for $n>l_1$ (since the finite complex
of countably generated projective $S$\+modules $T^\bu(S,\bs)$ is
concentrated in the nonnegative cohomological degrees).

 Similarly, one computes
\begin{multline*}
 U^\bu\ot_S^\boL\boR\Hom_S(U^\bu,E)\simeq
 L^\bu\ot_R^\boL T^\bu(S,\bs)\ot_S
 \Hom_S(T^\bu(S,\bs),\boR\Hom_R(L^\bu,E)) \\
 \simeq L^\bu\ot_R^\boL T^\bu(S,\bs)\ot_S\boR\Hom_R(L^\bu,E)
 \simeq T^\bu(S,\bs)\ot_S(L^\bu\ot_R^\boL\boR\Hom_R(L^\bu,E)),
\end{multline*}
where the second isomorphism holds because the natural map
$T^\bu(S,\bs)\ot_SM^\bu\rarrow
T^\bu(S,\bs)\ot_S\Hom_S(T^\bu(S,\bs),M^\bu)$ is a quasi-isomorphism
for any complex of $S$\+modules $M^\bu$
by~\cite[proof of Lemma~5.2(a)]{Pmgm}.

 Now if $E$ is a $J$\+torsion $S$\+module and the adjunction morphism
$L^\bu\ot_R^\boL\boR\Hom_R(L^\bu,E)\rarrow E$ is an isomorphism in
$\sD(R\Modl)$ (hence also in $\sD(S\Modl)$), then $L^\bu\ot_R^\boL
\boR\Hom_R(L^\bu,E)$ is a complex of $S$\+modules with $J$\+torsion
cohomology modules.
 Hence the natural map $T^\bu(S,\bs)\ot_S
(L^\bu\ot_R^\boL\boR\Hom_R(L^\bu,E))\rarrow
L^\bu\ot_R^\boL\boR\Hom_R(L^\bu,E)$ is a quasi-isomorphism of
complexes of $S$\+modules by~\cite[Lemma~1.1(c)]{Pmgm}, and it follows
that the adjunction morphism $U^\bu\ot_S^\boL\boR\Hom_S(U^\bu,E)
\allowbreak\rarrow E$ is an isomorphism whenever the adjunction morphism
$L^\bu\ot_R^\boL\boR\Hom_R(L^\bu,E)\allowbreak\rarrow E$
is an isomorphism.  {\hbadness=1525\par}

 Part~(b): for any $S$\+module $F$, we have
$$
 U^\bu\ot_S^\boL F \simeq T^\bu(S,\bs)\ot_S(L^\bu\ot_R^\boL F).
$$
 Now if $H^{-n}(L^\bu\ot_R^\boL F)=0$ for $n>l_1$, then
$H^{-n}(U^\bu\ot_S^\boL F)=0$ for $n>l_1$.

 Similarly, one computes
\begin{multline*}
 \boR\Hom_S(U^\bu,\>U^\bu\ot_S^\boL F)\simeq
 \boR\Hom_R\bigl(L^\bu,\>\Hom_S\bigl(T^\bu(S,\bs),\>T^\bu(S,\bs)\ot_S
 (L^\bu\ot_R^\boL F)\bigr)\bigr) \\
 \simeq\boR\Hom_R(L^\bu,\>\Hom_S(T^\bu(S,\bs),\>L^\bu\ot_R^\boL F))
 \simeq\Hom_S(T^\bu(S,\bs),\>\boR\Hom_R(L^\bu,\>L^\bu\ot_R^\boL F)),
\end{multline*}
where the second isomorphism holds because the natural map
$\Hom_S(T^\bu(S,\bs),\allowbreak\> T^\bu(S,\bs)\ot_SM^\bu)
\rarrow\Hom_S(T^\bu(S,\bs),M^\bu)$ is a quasi-isomorphism for any
complex of $S$\+modules $M^\bu$ by~\cite[proof of Lemma~5.2(b)]{Pmgm}.

 Now if $F$ is a $J$\+contramodule $S$\+module and the adjunction
morphism $F\allowbreak\rarrow\boR\Hom_R(L^\bu,\>L^\bu\ot_R^\boL F)$ is
an isomorphism in $\sD(R\Modl)$ (hence also in $\sD(S\Modl)$), then
$\boR\Hom_R(L^\bu,\>L^\bu\ot_R^\boL F)$ is a complex of $S$\+modules
with $J$\+contramodule cohomology modules.
 Hence the natural map $\boR\Hom_R(L^\bu,\>L^\bu\ot_R^\boL F)\rarrow
\Hom_S(T^\bu(S,\bs),\>\boR\Hom_R(L^\bu,\>L^\bu\ot_R^\boL F))$
is a quasi-isomorphism of complexes of $S$\+modules
by~\cite[Lemma~2.2(c)]{Pmgm}, and it follows that the adjunction
morphism $F\rarrow\boR\Hom_S(U^\bu,\>U^\bu\ot_S^\boL F)$ is
an isomorphism whenever the adjunction morphism $F\rarrow
\boR\Hom_R(L^\bu,\>L^\bu\ot_R^\boL F)$ is an isomorphism.
\hbadness=1600
\end{proof}

 The proposition above pertains to the Bass and Auslander classes.
 Now we turn to abstract corresponding classes.
 Given a class of $I$\+torsion $R$\+modules
$\sE\subset R\Modl_{I\tors}$, we denote by $\sG_\sE\subset
S\Modl_{J\tors}$ the class of all $J$\+torsion $S$\+modules whose
underlying $I$\+torsion $R$\+modules belong to~$\sE$.
 Similarly, given a class of $I$\+contramodule $R$\+modules
$\sF\subset R\Modl_{I\ctra}$, we denote by $\sH_\sF\subset
S\Modl_{J\ctra}$ the class of all $J$\+contramodule $S$\+modules
whose underlying $I$\+contramodule $R$\+modules belong to~$\sF$.

 As a special case of the setting described in the beginning of
this Section~\ref{base-change-secn}, one can consider the situation
when $J=SI$ is the ideal generated by $f(I)$ in~$S$.
 Notice that weak proregularity of $I$ in $R$ and flatness of $S$ over
$R$ imply weak proregularity of $J$ in $S$ in this
case~\cite[Example~2.4]{Pmgm}.
 Furthermore, an $S$\+module is $SI$\+torsion if and only if it is
$I$\+torsion as an $R$\+module, and an $S$\+module is
an $SI$\+contramodule if and only if it is $I$\+contramodule as
an $R$\+module.

 For an ideal $J\subset S$ such that $f(I)\subset J$, one has
$SI\subset J$.
 Hence the exact inclusions of abelian categories $S\Modl_{J\tors}
\subset S\Modl_{SI\tors}$ and $S\Modl_{J\ctra}\subset S\Modl_{SI\ctra}$.

 In this context, we will use the notation $\sG_\sE^\circ\subset
S\Modl_{SI\tors}$ for the class of all ($SI$\+torsion) $S$\+modules
whose underlying ($I$\+torsion) $R$\+modules belong to~$\sE$.
 Similarly, we denote by $\sH_\sF^\circ\subset S\Modl_{SI\ctra}$
the class of all ($SI$\+contramodule) $S$\+modules whose underlying
($I$\+contramodule) $R$\+modules belong to~$\sF$.
 So we have $\sG_\sE=S\Modl_{J\tors}\cap\sG_\sE^\circ\subset
S\Modl_{SI\tors}$ and $\sH_\sF=S\Modl_{J\ctra}\cap\sH_\sF^\circ
\subset S\Modl_{SI\ctra}$.

\begin{prop} \label{base-change-abstract-classes}
 Let $I$ be a weakly proregular finitely generated ideal in
a commutative ring $R$ and $J$ be a weakly proregular finitely
generated ideal in a commutative ring~$S$.
 Assume that the ring $R$ is $I$\+adically coherent.
 Let $f\:R\rarrow S$ be a ring homomorphism such that $f(I)\subset J$.
 Assume that $S$ is a flat $R$\+module and the morphism of pairs
$f\:(R,I)\rarrow(S,J)$ is quotflat in the sense of
Section~\ref{quotflat-secn}.
 Let $t\ge0$ be an integer such that the ideal $J/SI$ in the ring
$S/SI$ can be generated by $t$~elements.
 Let $L^\bu$ be a pseudo-dualizing complex of $I$\+torsion $R$\+modules
concentrated in the cohomological degrees $-d_1\le m\le d_2$,
and let $t_1$ and~$t_2$ be two integers such that the complex of
$S$\+modules $K^\bu_\infty(S,\bs)\ot_RL^\bu$ is quasi-isomorphic to
a finite complex of $J$\+torsion $S$\+modules $U^\bu$ concentrated
in the cohomological degrees $-t_1\le m\le t_2$ (as per the discussion
in the beginning of this Section~\ref{base-change-secn}).

 Let\/ $\sE\subset R\Modl_{I\tors}$ and\/ $\sF\subset R\Modl_{I\ctra}$
be a pair of full subcategories satisfying conditions (I\+-IV) from
Section~\ref{abstract-classes-secn} with respect to the pseudo-dualizing
complex of $I$\+torsion $R$\+modules $L^\bu$ with some parameters
$l_1\ge d_1$ and $l_2\ge d_2$.
 Assume that the class\/ $\sE$ is closed under countable direct sums in
$R\Modl_{I\tors}$ and contains all the fp\+injective $I$\+torsion
$R$\+modules, while the class\/ $\sF$ is closed under countable products
in $R\Modl_{I\ctra}$ and contains all the contraflat $I$\+contramodule
$R$\+modules.
 Let $u_1$ and~$u_2$ be two integers such that $u_1\ge l_1$,
\,$u_1\ge t_1$ and $u_2\ge l_2+t$, \,$u_2\ge t_2$.
 Then the pair of full subcategories\/ $\sG_\sE\subset S\Modl_{J\tors}$
and\/ $\sH_\sF\subset S\Modl_{J\ctra}$ satisfies conditions (I\+-IV)
with respect to the pseudo-dualizing complex of $J$\+torsion
$S$\+modules $U^\bu$ with the parameters $u_1$ and~$u_2$.
\end{prop}

\begin{proof}
 By Proposition~\ref{restriction-of-scalars-fp-injective}, all
the injective objects of $S\Modl_{J\tors}$ belong to~$\sG_\sE$.
 By Proposition~\ref{restriction-of-scalars-contraflat}, all
the projective objects of $S\Modl_{J\ctra}$ belong to~$\sH_\sF$.
 It follows that conditions~(I) and~(II) are satisfied for $\sG_\sE$
and~$\sH_\sF$.

 Now let us choose a finite sequence of generators~$\bs$ for
the ideal $J\subset S$ as described in the first paragraph of
the proof of Theorem~\ref{base-change-is-pseudo-dualizing}.
 By assumption, it can be done in such a way that the sequence~$\bt$
consists of $t$~elements.
 Then we have $K^\bu_\infty(S,\bs)=K^\bu_\infty(S,t)\ot_R
K^\bu_\infty(R,\br)$.
 By~\cite[Lemma~1.1(c)]{Pmgm}, the complex of $R$\+modules
$K^\bu_\infty(R,\br)\ot_RL^\bu$ is quasi-isomorphic to $L^\bu$, so
the complex of $S$\+modules $K^\bu_\infty(S,\bt)\ot_RL^\bu$ is
quasi-isomorphic to~$U^\bu$.

 Hence, similarly to the beginning of the proof of
Proposition~\ref{base-change-bass-auslander-classes}(a), for any
$S$\+module $E$ we have
$$
 \boR\Hom_S(U^\bu,E)\simeq\Hom_S(T^\bu(S,\bt),\boR\Hom_R(L^\bu,E)).
$$
 Assume that $E\in\sG_\sE$ (or more generally, $E\in\sG_\sE^\circ$).
 By condition~(III) for the classes $\sE$ and $\sF$ with respect to
the pseudo-dualizing complex of $I$\+torsion $R$\+modules $L^\bu$,
the derived category object $\boR\Hom_R(L^\bu,E)\in
\sD^\bb(R\Modl_{I\ctra})\subset\sD^\bb(R\Modl)$ can be represented by
a complex of $I$\+contramodule $R$\+modules concentrated in
the cohomological degrees $-l_2\le m\le l_1$ with the terms
belonging to~$\sF$.

 The full subcategories $\sF\subset R\Modl_{I\ctra}$
and $\sH_\sF\subset S\Modl_{J\ctra}$ are resolving (by condition~(II),
which we have already proved).
 In particular, so is the full subcategory $\sH_\sF^\circ\subset
S\Modl_{SI\ctra}$ (cf.\ Remark~\ref{alternative-assumptions-remark}(1)).
 As the resolution dimension of a finite complex does not depend on
the choice of a resolution (by~\cite[Corollary~A.5.2]{Pcosh}), it
follows that the derived category object $\boR\Hom_R(L^\bu,E)\in
\sD^\bb(S\Modl_{SI\ctra})$ can be represented by a complex of
$SI$\+contramodule $S$\+modules $F^\bu$ concentrated in
the cohomological degrees $-l_2\le m\le l_1$ with the terms belonging
to~$\sH_\sF^\circ$.

 As $T^\bu(S,\bt)$ is a complex of countably generated free $S$\+modules
concentrated in the cohomological degrees $0\le m\le t$, and the full
subcategory $\sH_\sF^\circ$ is closed under countable products in
$S\Modl$, we can conclude that the complex
$\Hom_S(T^\bu(S,\bt),F^\bu)$ has the terms belonging to $\sH_\sF^\circ$,
is concentrated in the cohomological degrees $-l_2-t\le m\le l_1$,
and represents the derived category object
$\boR\Hom_S(U^\bu,E)\in\sD^+(S\Modl)$.
 On the other hand, we actually have $\boR\Hom_S(U^\bu,E)\in
\sD^\bb(S\Modl_{J\ctra})\subset\sD^\bb(S\Modl)$ by
Lemma~\ref{tors-contra-tensor-hom-for-complexes}(b), since $U^\bu$ is
a complex of $J$\+torsion $S$\+modules.

 Recall that we have $\sH_\sF=S\Modl_{J\ctra}\cap\sH_\sF^\circ\subset
S\Modl_{SI\ctra}$.
 By~\cite[Corollary~A.5.5]{Pcosh}, it follows that the object
$\boR\Hom_S(U^\bu,E)\in\sD^\bb(S\Modl_{J\ctra})$ can be represented
by a complex of $J$\+contramodule $S$\+modules concentrated in
the cohomological degrees $-l_2-t\le m\le l_1$ with the terms
belonging to~$\sH_\sF$.
 This proves condition~(III) for the classes $\sG_\sE$ and $\sH_\sF$
with respect to the pseudo-dualizing complex of $J$\+torsion
$S$\+modules~$U^\bu$.

 Dual-analogously, as in the beginning of the proof of
Proposition~\ref{base-change-bass-auslander-classes}(b), for any
$S$\+module $F$ we have
$$
 U\ot_S^\boL F\simeq T^\bu(S,\bt)\ot_S(L^\bu\ot_R^\boL F).
$$
 Assume that $F\in\sH_\sF$ (or more generally, $F\in\sH_\sF^\circ$).
 By condition~(IV) for the classes $\sE$ and $\sF$ with respect to
the pseudo-dualizing complex of $I$\+torsion $R$\+modules $L^\bu$,
the derived category object $L^\bu\ot_R^\boL F\in
\sD^\bb(R\Modl_{I\tors})\subset\sD^\bb(R\Modl)$ can be represented by
a complex of $I$\+torsion $R$\+modules concentrated in the cohomological
degrees $-l_1\le m\le l_2$ with the terms belonging to~$\sE$.

 The full subcategories $\sE\subset R\Modl_{I\tors}$ and
$\sG_\sE\subset S\Modl_{J\tors}$ are coresolving (by condition~(I),
which we have already proved).
 In particular, so is the full subcategory $\sG_\sE^\circ\subset
S\Modl_{SI\tors}$ (cf.\ Remark~\ref{alternative-assumptions-remark}(1)).
 In view of the dual version of~\cite[Corollary~A.5.2]{Pcosh}, it
follows that the derived category object $L^\bu\ot_R^\boL F\in
\sD^\bb(S\Modl_{SI\tors})$ can be represented by a complex of
$SI$\+torsion $S$\+modules $E^\bu$ concentrated in the cohomological
degrees $-l_1\le m\le l_2$ with the terms belonging to~$\sG_\sE^\circ$.

 As $T^\bu(S,\bt)$ is a complex of countably generated free $S$\+modules
concentrated in the cohomological degrees $0\le m\le t$, and the full
subcategory $\sG_\sE^\circ$ is closed under countable direct sums in
$S\Modl$, we can conclude that the complex $T^\bu(S,\bt)\ot_R E^\bu$ has
the terms belonging to $\sG_\sE^\circ$,
is concentrated in the cohomological degrees $-l_1\le m\le l_2+t$,
and represents the derived category object
$U^\bu\ot_S^\boL F\in\sD^-(S\Modl)$.
 On the other hand, we actually have $U^\bu\ot_S^\boL F\in
\sD^\bb(S\Modl_{J\tors})\subset\sD^\bb(S\Modl)$ by
Lemma~\ref{tors-contra-tensor-hom-for-complexes}(a), since
$U^\bu$ is a complex of $J$\+torsion $S$\+modules.

 Recall that we have $\sG_\sE=S\Modl_{J\tors}\cap\sG_\sE^\circ\subset
S\Modl_{SI\tors}$.
 By the dual version of~\cite[Corollary~A.5.5]{Pcosh}, it follows that
the object $U^\bu\ot_S^\boL F\in\sD^\bb(S\Modl_{J\tors})$ can be
represented by a complex of $J$\+torsion $S$\+modules concentrated in
the cohomological degrees $-l_1\le m\le l_2+t$ with the terms
belonging to~$\sG_\sE$.
 This proves condition~(IV) for the classes $\sG_\sE$ and $\sH_\sF$
with respect to the pseudo-dualizing complex of $J$\+torsion
$S$\+modules~$U^\bu$.
\end{proof}

\begin{cor} \label{base-change-abstract-classes-derived-equivalence}
 In the context and assumptions of
Proposition~\ref{base-change-abstract-classes}, for any conventional
or absolute derived category symbol\/ $\st=\bb$, $+$, $-$,
$\varnothing$, $\abs+$, $\abs-$, or\/~$\abs$, there is a triangulated
equivalence
$$
 \sD^\st(\sG_\sE)\simeq\sD^\st(\sH_\sF)
$$
provided by (appropriately defined) mutually inverse derived functors\/
$\boR\Hom_S(U^\bu,{-})$ and $U^\bu\ot_S^\boL{-}$.
\end{cor}

\begin{proof}
 This is a particular case of
Theorem~\ref{abstract-classes-derived-equivalence}, which is applicable
in view of Theorem~\ref{base-change-is-pseudo-dualizing}
and Proposition~\ref{base-change-abstract-classes}.
\end{proof}

 Notice that, according to
Remark~\ref{alternative-assumptions-remark}(2),
the quotflatness assumption can be dropped in
Proposition~\ref{base-change-abstract-classes}
and Corollary~\ref{base-change-abstract-classes-derived-equivalence}
if one is willing to assume the ring $S$ to be Noetherian instead.

 Let us also \emph{warn} the reader that, unlike in the context
of~\cite[diagram~(14)]{Pps}, the square diagram formed by
the triangulated equivalences and triangulated forgetful functors
\begin{equation} \label{pseudo-derived-equivs-and-forgetful}
\begin{gathered}
 \xymatrix{
  \sD^\st(\sG_\sE) \ar@{=}[r] \ar[d] & \sD^\st(\sH_\sF) \ar[d] \\
  \sD^\st(\sE) \ar@{=}[r] & \sD^\st(\sF)
 }
\end{gathered}
\end{equation}
is usually \emph{not commutative}.
 In fact, the diagram of triangulated
functors~\eqref{pseudo-derived-equivs-and-forgetful} is commutative
when $J=SI$, but \emph{not} in the general case.
 This is clear from the proof of
Proposition~\ref{base-change-abstract-classes}.

\Section{Pseudo-Derived Categories in the Relative Context}
\label{pseudo-derived-secn}

 Roughly speaking, a pseudo-coderived category is an intermediate
triangulated quotient category between the conventional derived and
the coderived category, while a pseudo-contraderived category is
an intermediate triangulated quotient category between the derived
and the contraderived category.
 The concept of pseudo-derived categories presently works better in
the context of the Positselski co/contraderived categories than
the Becker ones, for the simple reason that we do not have a good
technology for proving that all Becker co/contraacyclic complexes in
an exact category are acyclic (cf.\
Lemma~\ref{becker-co-contra-acyclic-are-acyclic}).
 That is why we only consider the Positselski co/contraderived
categories in this section.

 We start with a brief recollection of the discussion of pseudo-derived
categories from~\cite[Introduction and Section~4]{PS2}
and~\cite[Section~1]{Pps}.

 Let $\sA$ be an exact category with exact functors of infinite direct
sum (as defined in Section~\ref{prelim-exotic-derived-secn}).
 Then it is clear that every Positselski-coacyclic complex in $\sA$ is
acyclic (where, to be precise, acyclicity of a complex in an exact
category is understood in the sense of~\cite[Section~A.1]{Pcosh},
which agrees with the terminology in~\cite{Neem0,Bueh} when $\sA$
is idempotent-complete).
 So $\Ac^\co(\sA)\subset\Ac(\sA)$, and it follows that the derived
category $\sD(\sA)$ is naturally a triangulated Verdier quotient
category of $\sD^\co(\sA)$.
 In other words, the triangulated Verdier quotient functor
$Q_\sA\:\sK(\sA)\twoheadrightarrow\sD(\sA)$ factorizes naturally through
the triangulated Verdier quotient functor $Q_\sA^\co\:\sK(\sA)
\twoheadrightarrow\sD^\co(\sA)$, so $Q_\sA$ is the composition of
triangulated Verdier quotient functors
$$
 \xymatrix{
  \sK(\sA) \ar@{->>}[r]^-{Q_\sA^\co}
  & \sD^\co(\sA) \ar@{->>}[r]^-{R_\sA^\co} & \sD(\sA).
 }
$$
 A triangulated category $\sD'$ is said to be a \emph{pseudo-coderived
category} of $\sA$ if it is endowed with two triangulated Verdier
quotient functors $\sD^\co(\sA)\twoheadrightarrow\sD'\twoheadrightarrow
\sD(\sA)$ whose composition is the triangulated Verdier quotient
functor~$R_\sA^\co$.

 Dually, let $\sB$ be an exact category with exact functors of
infinite products.
 Then every Positselski-contraacyclic complex in $\sB$ is acyclic,
$\Ac^\ctr(\sB)\subset\Ac(\sB)$.
 So the triangulated Verdier quotient functor $Q_\sB\:\sK(\sB)
\twoheadrightarrow\sD(\sB)$ factorizes naturally through
the triangulated Verdier quotient functor $Q_\sB^\ctr\:\sK(\sB)
\twoheadrightarrow\sD^\ctr(\sB)$,
$$
 \xymatrix{
  \sK(\sB) \ar@{->>}[r]^-{Q_\sB^\ctr}
  & \sD^\ctr(\sB) \ar@{->>}[r]^-{R_\sB^\ctr} & \sD(\sB),
 }
$$
and we obtain a natural triangulated Verdier quotient functor
$R_\sB^\ctr\:\sD^\ctr(\sB)\twoheadrightarrow\sD(\sB)$.
 A triangulated category $\sD''$ is said to be
a \emph{pseudo-contraderived category} of $\sB$ if it is endowed
with two triangulated Verdier quotient functors $\sD^\ctr(\sB)
\twoheadrightarrow\sD''\twoheadrightarrow\sD(\sB)$ whose composition
is the triangulated Verdier quotient functor~$R_\sB^\ctr$.

 Now let $\sE\subset\sA$ be a coresolving full subcategory closed
under infinite direct sums.
 Then, by the dual version of~\cite[Proposition~A.3.1(b)]{Pcosh},
the inclusion functor $\sE\rarrow\sA$ induces a triangulated
equivalence $\sD^\co(\sE)\simeq\sD^\co(\sA)$.
 So we obtain a triangulated Verdier quotient functor
$\sD^\co(\sA)\simeq\sD^\co(\sE)\rarrow\sD(\sE)$, and it follows
that the functor $\sD(\sE)\rarrow\sD(\sA)$ induced by the inclusion
$\sE\rarrow\sA$ is a triangulated Verdier quotient functor,
too~\cite[Proposition~4.2(a)]{PS2}.
 Thus $\sD(\sE)$ is a pseudo-coderived category of~$\sA$,
$$
 \xymatrix{
  \sD^\co(\sA) \ar@{->>}[r] & \sD(\sE) \ar@{->>}[r] & \sD(\sA)
 }
$$
\cite[Section~1]{Pps}.

 Following the terminology in~\cite[Section~8]{Pps},
we will say that a complex in $\sA$ is \emph{$\sE$\+pseudo-coacyclic}
if it is annihilated by the composition of triangulated Verdier
quotient functors $\sK(\sA)\twoheadrightarrow\sD^\co(\sA)
\twoheadrightarrow\sD(\sE)$.
 Denote the thick subcategory of $\sE$\+pseudo-coacyclic complexes
by $\Ac^{\sE\psco}(\sA)\subset\sK(\sA)$.
 So we have the inclusions of thick subcategories
$$
 \Ac^\co(\sA)\subset\Ac^{\sE\psco}(\sA)\subset\Ac(\sA)\subset\sK(\sA)
$$
and a natural triangulated equivalence
\begin{equation} \label{pseudo-coderived-category-from-coresolving}
 \sD(\sE)\simeq\sK(\sA)/\Ac^{\sE\psco}(\sA).
\end{equation}

 Dually, let $\sF\subset\sB$ be a resolving full subcategory closed
under infinite products.
 Then, by~\cite[Proposition~A.3.1(b)]{Pcosh}, the inclusion functor
$\sF\rarrow\sB$ induces a triangulated equivalence
$\sD^\ctr(\sF)\simeq\sD^\ctr(\sB)$.
 Hence $\sD(\sF)$ becomes a pseudo-contraderived category of~$\sB$,
$$
 \xymatrix{
  \sD^\ctr(\sB) \ar@{->>}[r] & \sD(\sF) \ar@{->>}[r] & \sD(\sB)
 }
$$
\cite[Proposition~4.2(b)]{PS2}, \cite[Section~1]{Pps}.

 We will say that a complex in $\sB$ is
\emph{$\sF$\+pseudo-contraacyclic} if it is annihilated by
the composition of triangulated Verdier quotient functors $\sK(\sB)
\twoheadrightarrow\sD^\ctr(\sB)\twoheadrightarrow\sD(\sF)$.
 Denote the thick subcategory of $\sF$\+pseudo-contraacyclic complexes
by $\Ac^{\sF\psctr}(\sB)\subset\sK(\sB)$.
 So we have the inclusions of thick subcategories
$$
 \Ac^\ctr(\sB)\subset\Ac^{\sF\psctr}(\sB)\subset\Ac(\sB)\subset\sK(\sB)
$$
and a natural triangulated equivalence
\begin{equation} \label{pseudo-contraderived-category-from-resolving}
 \sD(\sF)\simeq\sK(\sB)/\Ac^{\sF\psctr}(\sB).
\end{equation}

\begin{lem} \label{pseudo-co-contra-acyclic-for-finite-homol-dim}
\textup{(a)} If the exact category\/ $\sE$ has finite homological
dimension, then the class of\/ $\sE$\+pseudo-coacyclic complexes in\/
$\sA$ coincides with the class of Positselski-coacyclic complexes,
$\Ac^{\sE\psco}(\sA)=\Ac^\co(\sA)$. \par
\textup{(b)} If the exact category\/ $\sF$ has finite homological
dimension, then the class of\/ $\sF$\+pseudo-contraacyclic complexes
in\/ $\sB$ coincides with the class of Positselski-contraadyclic
complexes, $\Ac^{\sF\psctr}(\sB)=\Ac^\ctr(\sB)$.
\end{lem}

\begin{proof}
 Let us prove part~(a) (part~(b) is dual).
 By the definition, the $\sE$\+pseudo-coacyclic complexes in $\sA$
are the complexes annihilated by the composition of triangulated
Verdier quotient functors
$$
 \xymatrix{
  \sK(\sA) \ar@{->>}[r] & \sD^\co(\sA) \ar@{=}[r]
  & \sD^\co(\sE) \ar@{->>}[r] & \sD(\sE).
 }
$$
 For an exact category $\sE$ of finite homological dimension with
exact functors of infinite direct sum, the classes of acyclic and
Positselski-coacyclic complexes coincide~\cite[Remark~2.1]{Psemi},
so the functor $\sD^\co(\sE)\rarrow\sD(\sE)$ is a triangulated
equivalence (in fact, isomorphism of triangulated categories).
\end{proof}

 The following lemma is a category-theoretic generalization
of~\cite[Lemma~8.3]{Pps}.

\begin{lem} \label{relative-pseudo-derived-conservative}
\textup{(a)} Let\/ $\sA$ and\/ $\sX$ be abelian categories with
exact functors of infinite direct sum and\/ $\Theta\:\sX\rarrow\sA$ be
a faithful exact functor preserving infinite direct sums.
 Let\/ $\sE\subset\sA$ be a coresolving subcategory closed under
infinite direct sums, and let\/ $\sG_\sE\subset\sX$ be the full
subcategory of all objects $E\in\sX$ such that\/ $\Theta(E)\in\sE$.
 Assume that the full subcategory\/ $\sG_\sE$ is coresolving in\/~$\sX$
(i.~e., in other words, every object of\/ $\sX$ is a subobject of
an object from\/~$\sG_\sE$).
 Then a complex $X^\bu$ in\/ $\sX$ belongs to the thick subcategory\/
$\Ac^{\sG_\sE\psco}(\sX)\subset\sK(\sX)$ if and only if the complex\/
$\Theta(X^\bu)\in\sK(\sA)$ belongs to the full subcategory\/
$\Ac^{\sE\psco}(\sA)\subset\sK(\sA)$. \par
\textup{(b)} Let\/ $\sB$ and\/ $\sY$ be abelian categories with
exact functors of infinite product and\/ $\Theta\:\sY\rarrow\sB$ be
a faithful exact functor preserving infinite products.
 Let\/ $\sF\subset\sB$ be a resolving subcategory closed under
infinite products, and let\/ $\sH_\sF\subset\sY$ be the full
subcategory of all objects $F\in\sY$ such that\/ $\Theta(F)\in\sF$.
 Assume that the full subcategory\/ $\sH_\sF$ is resolving in\/~$\sY$
(i.~e., in other words, every object of\/ $\sY$ is a quotient object
of an object from\/~$\sH_\sF$).
 Then a complex $Y^\bu$ in\/ $\sY$ belongs to the thick subcategory\/
$\Ac^{\sH_\sF\psctr}(\sY)\subset\sK(\sY)$ if and only if the complex\/
$\Theta(Y^\bu)\in\sK(\sB)$ belongs to the full subcategory\/
$\Ac^{\sF\psctr}(\sB)\subset\sK(\sB)$.
\end{lem}

\begin{proof}
 Let us prove part~(a).
 The full subcategory $\Ac(\sE)\subset\sK(\sE)$ consists of all
complexes in $\sE$ that are acyclic in $\sA$ with the objects of
cocycles belonging to~$\sE$.
 Similarly, the full subcategory $\Ac(\sG_\sE)\subset\sK(\sG_\sE)$
consists of all complexes in $\sG_\sE$ that are acyclic in $\sX$ with
the objects of cocycles belonging to~$\sG_\sE$.
 Now a complex $X^\bu$ is acyclic in $\sX$ if and only if the complex
$\Theta(X^\bu)$ is acyclic in~$\sA$ (since the functor of abelian
categories $\Theta\:\sX\rarrow\sA$ is exact and faithful).
 Thus a complex $E^\bu$ in $\sG_\sE$ belongs to $\Ac(\sG_\sE)$ if and
only if the complex $\Theta(E^\bu)$ belongs to $\Ac(\sE)$.

 We have proved that the triangulated functor $\Theta\:\sD(\sG_\sE)
\rarrow\sD(\sE)$ takes nonzero objects to nonzero objects (or in
other words, takes nonisomorphisms to nonisomorphisms, i.~e., it is
conservative).
 In order to prove the assertion of part~(a), it remains to consider
the commutative diagram of triangulated functors
$$
 \xymatrix{
  \sK(\sX) \ar@{->>}[r] \ar[d]^\Theta
  & \sD^\co(\sX) \ar@{=}[r] \ar[d]^\Theta
  & \sD^\co(\sG_\sE) \ar@{->>}[r] \ar[d]^\Theta
  & \sD(\sG_\sE) \ar[d]^\Theta \\
  \sK(\sA) \ar@{->>}[r] & \sD^\co(\sA) \ar@{=}[r]
  & \sD^\co(\sE) \ar@{->>}[r] & \sD(\sE)
 }
$$
implying that a complex $X^\bu$ in $\sX$ is annihilated by
the triangulated Verdier quotient functor
$\sK(\sX)\twoheadrightarrow\sD(\sG_E)$ if and only if the complex
$\Theta(X^\bu)$ is annihilated by the triangulated Verdier quotient
functor $\sK(\sA)\twoheadrightarrow\sD(\sE)$.
\end{proof}

\Section{Semiderived Categories}  \label{semiderived-secn}

 The ``semiderived categories'' is an umbrella term for
the semicoderived and the semicontraderived categories.
 The notion of the semiderived category goes back to
the book~\cite{Psemi}.
 Other sources relevant to our context include the paper~\cite{Pfp},
the book~\cite{Psemten}, and the preprint~\cite[Section~8]{Pcosh}.

 Let $I$ be an ideal in a commutative ring $R$, let $J$ be an ideal
in a commutative ring $S$, and let $f\:R\rarrow S$ be a ring
homomorphism such that $f(I)\subset J$.
 Denote by functor of restriction of scalars by $\Theta\:S\Modl
\rarrow R\Modl$.
 We will use the same notation for functors of restriction of scalars
acting between the categories of torsion modules or contramodules,
that is $\Theta\:S\Modl_{J\tors}\rarrow R\Modl_{I\tors}$
and $\Theta\:S\Modl_{J\ctra}\rarrow R\Modl_{I\ctra}$.
 When the ideals $I\subset R$ and $J\subset S$ are finitely generated,
the latter functor preserves the quotseparatedness property of
contramodules, so we also have the forgetful functor
$\Theta\:S\Modl_{J\ctra}^\qs\rarrow R\Modl_{I\ctra}^\qs$.

 A complex of $J$\+torsion $S$\+modules $X^\bu$ is said to be
\emph{Positselski-semicoacyclic} (\emph{relative to $(R,I)$}) if it is
Positselski-coacyclic as a complex of $I$\+torsion $R$\+modules, i.~e.,
if $\Theta(X^\bu)\in\Ac^\co(R\Modl_{I\tors})$.
 We denote the thick subcategory of Positselski-semicoacyclic complexes
by $\Ac^\sico_{(R,I)}(S\Modl_{J\tors})\subset\sK(S\Modl_{J\tors})$.
 The \emph{Positselski semicoderived category of $J$\+torsion
$S$\+modules} (\emph{relative to $(R,I)$}) is defined as
the triangulated Verdier quotient category
$$
 \sD^\sico_{(R,I)}(S\Modl_{J\tors})=
 \sK(S\Modl_{J\tors})/\Ac^\sico_{(R,I)}(S\Modl_{J\tors}).
$$

 The forgetful functor $\Theta\:S\Modl_{J\tors}\rarrow R\Modl_{I\tors}$
is exact and preserves infinite direct sums, so it takes
Positselski-coacyclic complexes to Positselski-coacyclic complexes.
 Hence any Positselski-coacyclic complex in $S\Modl_{J\tors}$ is
Positselski-semicoacyclic.
 The forgetful functor $\Theta\:S\Modl_{J\tors}\rarrow R\Modl_{I\tors}$
is also exact and faithful, so a complex $X^\bu$ in $S\Modl_{J\tors}$
is acyclic if and only if the complex $\Theta(X^\bu)$ is acyclic
in $R\Modl_{I\tors}$.
 The abelian category $R\Modl_{I\tors}$ has exact functors of infinite
direct sums, so all Positselski-coacyclic complexes in $R\Modl_{I\tors}$
are acyclic.
 It follows that any Positselski-semicoacyclic complex in
$S\Modl_{J\tors}$ is acyclic.
 Therefore, we have
$$
 \Ac^\co(S\Modl_{J\tors})\subset\Ac^\sico_{(R,I)}(S\Modl_{J\tors})
 \subset\Ac(S\Modl_{J\tors}).
$$
 Thus the Positselski semicoderived category is an intermediate
triangulated quotient category between the derived and the coderived
categories of $S\Modl_{J\tors}$, i.~e., the Positselski semicoderived
category is an example of a pseudo-coderived category of $J$\+torsion
$S$\+modules in the sense of Section~\ref{pseudo-derived-secn}.

 Dual-analogously, a complex of $J$\+contramodule $S$\+modules $Y^\bu$
is said to be \emph{Positselski-semicontraacyclic} (\emph{relative to
$(R,I)$}) if it is Positselski-contraacyclic as a complex of
$I$\+contramodule $R$\+modules, i.~e., if
$\Theta(Y^\bu)\in\Ac^\ctr(R\Modl_{I\ctra})$.
 We denote the thick subcategory of Positselski-semicontraacyclic
complexes by $\Ac^\sictr_{(R,I)}(S\Modl_{J\ctra})\subset
\sK(S\Modl_{J\ctra})$.
 The \emph{Positselski semicontraderived category of $J$\+contramodule
$S$\+modules} (\emph{relative to $(R,I)$}) is defined as
the triangulated Verdier quotient category
$$
 \sD^\sictr_{(R,I)}(S\Modl_{J\ctra})=
 \sK(S\Modl_{J\ctra})/\Ac^\sictr_{(R,I)}(S\Modl_{J\ctra}).
$$

 The forgetful functor $\Theta\:S\Modl_{J\ctra}\rarrow R\Modl_{I\ctra}$
is exact and preserves infinite products, so it takes
Positselski-contraacyclic complexes to Positselski-contraacyclic
complexes.
 Hence any Positselski-contraacyclic complex in $S\Modl_{J\ctra}$ is
Positselski-semicontraacyclic.
 The forgetful functor $\Theta\:S\Modl_{J\ctra}\rarrow R\Modl_{I\ctra}$
is also exact and faithful, so a complex $Y^\bu$ in $S\Modl_{J\ctra}$
is acyclic if and only if the complex $\Theta(Y^\bu)$ is acyclic
in $R\Modl_{I\ctra}$.
 The abelian category $R\Modl_{I\ctra}$ has exact functors of infinite
product, so all Positselski-contraacyclic complexes in $R\Modl_{I\ctra}$
are acyclic.
 It follows that any Positselski-semicontraacyclic complex in
$S\Modl_{J\ctra}$ is acyclic.
 Therefore, we have
$$
 \Ac^\ctr(S\Modl_{J\ctra})\subset\Ac^\sictr_{(R,I)}(S\Modl_{J\ctra})
 \subset\Ac(S\Modl_{J\ctra}).
$$
 Thus the Positselski semicontraderived category is an intermediate
triangulated quotient category between the derived and the contraderived
categories of $S\Modl_{J\ctra}$, i.~e., the Positselski
semicontraderived category is an example of a pseudo-con\-tra\-de\-rived
category of $J$\+contramodule $S$\+modules in the sense of
Section~\ref{pseudo-derived-secn}.

 Assume that the ideals $I\subset R$ and $J\subset S$ are finitely
generated.
 Then a complex of quotseparated $J$\+contramodule $S$\+modules
$Y^\bu$ is said to be \emph{Positselski-semicontraacyclic}
(\emph{relative to $(R,I)$}) if it is Positselski-contraacyclic as
a complex of quotseparated $I$\+contramodule $R$\+modules, i.~e., if
$\Theta(Y^\bu)\in\Ac^\ctr(R\Modl_{I\ctra}^\qs)$.
 We denote the thick subcategory of Positselski-semicontraacyclic
complexes by $\Ac^\sictr_{(R,I)}(S\Modl_{J\ctra}^\qs)\subset
\sK(S\Modl_{J\ctra}^\qs)$.
 The \emph{Positselski semicontraderived category of quotseparated
$J$\+contramodule $S$\+modules} (\emph{relative to $(R,I)$}) is
defined as the triangulated Verdier quotient category
$$
 \sD^\sictr_{(R,I)}(S\Modl_{J\ctra}^\qs)=
 \sK(S\Modl_{J\ctra}^\qs)/\Ac^\sictr_{(R,I)}(S\Modl_{J\ctra}^\qs).
$$

 All the arguments in the discussion above are applicable in
the case of quotseparated contramodules just as well.
 So we have the inclusions of thick subcategories in the homotopy
category
$$
 \Ac^\ctr(S\Modl_{J\ctra}^\qs)\subset
 \Ac^\sictr_{(R,I)}(S\Modl_{J\ctra}^\qs)
 \subset\Ac(S\Modl_{J\ctra}^\qs).
$$
 Thus the Positselski semicontraderived category is an intermediate
triangulated quotient category between the derived and the contraderived
categories of $S\Modl_{J\ctra}^\qs$, i.~e., the Positselski
semicontraderived category is an example of a pseudo-con\-tra\-de\-rived
category of quotseparated $J$\+contramodule $S$\+modules in the sense
of Section~\ref{pseudo-derived-secn}.

 The discussion of Becker semiderived categories requires more care.

\begin{lem} \label{forgetful-preserve-becker-co-contra-acyclicity}
 Let $I$ be an ideal in a commutative ring $R$, let $J$ be an ideal in
a commutative ring $S$, and let $f\:R\rarrow S$ be a ring homomorphism
such that $f(I)\subset J$.
 In this setting: \par
\textup{(a)} The functor of restriction of scalars\/
$\Theta\:S\Modl_{J\tors}\rarrow R\Modl_{I\tors}$ takes
Becker-coacyclic complexes in $S\Modl_{J\tors}$ to Becker-coacyclic
complexes in $R\Modl_{I\tors}$. {\hbadness=1800\par}
\textup{(b)} The functor of restriction of scalars\/
$\Theta\:S\Modl_{J\ctra}\rarrow R\Modl_{I\ctra}$ takes
Becker-contraacyclic complexes in $S\Modl_{J\ctra}$ to
Becker-contraacyclic complexes in $R\Modl_{I\ctra}$. \par
\textup{(c)} Assume that the ideals $I\subset R$ and $J\subset S$
are finitely generated.
 Then the functor of restriction of scalars\/
$\Theta\:S\Modl_{J\ctra}^\qs\rarrow R\Modl_{I\ctra}^\qs$ takes
Becker-con\-tra\-acyclic complexes in $S\Modl_{J\ctra}^\qs$ to
Becker-contraacyclic complexes in $R\Modl_{I\ctra}^\qs$.
\end{lem}

\begin{proof}
 Part~(a): as the functor $\Theta$ is exact and preserves infinite
direct sums (hence also all colimits), it suffices to refer to
the result of~\cite[Lemma~A.5]{Psemten}, which is applicable to
Grothendieck abelian categories; or even directly
to~\cite[Corollary~7.17]{PS5}.
 Alternatively, one can refer to the more general result
of~\cite[Lemma~B.7.5(a)]{Pcosh}, and then one needs to know that
the forgetful functor $\Theta\:S\Modl_{J\tors}\rarrow R\Modl_{I\tors}$
has a right adjoint.
 The point is that the desired right adjoint functor is easy to
construct explicitly: it takes an $I$\+torsion $R$\+module $M$
to the $J$\+torsion $S$\+module $\Gamma_J(\Hom_R(S,M))$.

 Parts~(b\+-c): the argument is based on~\cite[Lemma~B.7.5(b)]{Pcosh}.
 In both the cases~(b) and~(c), the respective functor $\Theta$ is
exact and preserves infinite products (hence all limits).
 So it remains to show that the functor $\Theta$ has a left adjoint.
 One can observe that all the contramodule categories in question
are locally $\lambda$\+presentable, and the functor $\Theta$
preserves $\lambda$\+directed colimits for a suitable regular
cardinal~$\lambda$ ($\lambda=\aleph_1$ in part~(c) and in the case
of finitely generated ideals $I\subset R$ and $J\subset S$ in part~(b)).
 See the discussion in Section~\ref{prelims-on-wpr-secn}.
 Hence a left adjoint functor to $\Theta$ exists by~\cite[Adjoint
Functor Theorem~1.66]{AR}.

 Alternatively, the left adjoint functors can be constructed explicitly.
 In the context of part~(b), the left adjoint functor to
$\Theta\:S\Modl_{J\ctra}\rarrow R\Modl_{I\ctra}$ takes
an $I$\+contramodule $R$\+module $P$ to the $J$\+contramodule
$S$\+module $\Delta_J(S\ot_RP)$.
 In the context of part~(c), the left adjoint functor to
$\Theta\:S\Modl_{J\ctra}^\qs\rarrow R\Modl_{I\ctra}^\qs$ takes
a quotseparated $I$\+contramodule $R$\+module $P$ to the quotseparated
$J$\+contramodule $S$\+module $\boL_0\Lambda_J(S\ot_RP)$.
 See Section~\ref{prelims-on-wpr-secn} for the notation.
\end{proof}

 Let $I$ be an ideal in a commutative ring $R$, let $J$ be an ideal in
a commutative ring $S$, and let $f\:R\rarrow S$ be a ring homomorphism
such that $f(I)\subset J$.
 A complex of $J$\+torsion $S$\+modules $X^\bu$ is said to be
\emph{Becker-semicoacyclic} (\emph{relative to $(R,I)$}) if it is
Becker-coacyclic as a complex of $I$\+torsion $R$\+modules, i.~e.,
if $\Theta(X^\bu)\in\Ac^\bco(R\Modl_{I\tors})$.
 We denote the thick subcategory of Becker-semicoacyclic complexes by
$\Ac^\bsico_{(R,I)}(S\Modl_{J\tors})\subset\sK(S\Modl_{J\tors})$.
 The \emph{Becker semicoderived category of $J$\+torsion $S$\+modules}
(\emph{relative to $(R,I)$}) is defined as the triangulated Verdier
quotient category
$$
 \sD^\bsico_{(R,I)}(S\Modl_{J\tors})=
 \sK(S\Modl_{J\tors})/\Ac^\bsico_{(R,I)}(S\Modl_{J\tors}).
$$

 The forgetful functor $\Theta\:S\Modl_{J\tors}\rarrow R\Modl_{I\tors}$
takes Becker-coacyclic complexes to Becker-coacyclic complexes by
Lemma~\ref{forgetful-preserve-becker-co-contra-acyclicity}(a).
 Hence any Becker-coacyclic complex in $S\Modl_{J\tors}$ is
Becker-semicoacyclic.
 A complex $X^\bu$ in $S\Modl_{J\tors}$ is acyclic if and only if
the complex $\Theta(X^\bu)$ is acyclic in $R\Modl_{I\tors}$,
as explained in the first half of this section.
 The abelian category $R\Modl_{I\tors}$ has enough injective objects,
so all Becker-coacyclic complexes in $R\Modl_{I\tors}$ are acyclic
by Lemma~\ref{becker-co-contra-acyclic-are-acyclic}(a).
 It follows that any Becker-semicoacyclic complex in
$S\Modl_{J\tors}$ is acyclic.
 Thus we have
$$
 \Ac^\bco(S\Modl_{J\tors})\subset\Ac^\bsico_{(R,I)}(S\Modl_{J\tors})
 \subset\Ac(S\Modl_{J\tors}).
$$ 

 Dual-analogously, a complex of $J$\+contramodule $S$\+modules $Y^\bu$
is said to be \emph{Becker-semicontraacyclic} (\emph{relative to
$(R,I)$}) if it is Becker-contraacyclic as a complex of
$I$\+contramodule $R$\+modules, i.~e., if
$\Theta(Y^\bu)\in\Ac^\bctr(R\Modl_{I\ctra})$.
 We denote the thick subcategory of Becker-semicontraacyclic
complexes by $\Ac^\bsictr_{(R,I)}(S\Modl_{J\ctra})\subset
\sK(S\Modl_{J\ctra})$.
 The \emph{Becker semicontraderived category of $J$\+contramodule
$S$\+modules} (\emph{relative to $(R,I)$}) is defined as
the triangulated Verdier quotient category
$$
 \sD^\bsictr_{(R,I)}(S\Modl_{J\ctra})=
 \sK(S\Modl_{J\ctra})/\Ac^\bsictr_{(R,I)}(S\Modl_{J\ctra}).
$$

 The forgetful functor $\Theta\:S\Modl_{J\ctra}\rarrow R\Modl_{I\ctra}$
takes Becker-contraacyclic complexes to Becker-contraacyclic complexes
by Lemma~\ref{forgetful-preserve-becker-co-contra-acyclicity}(b).
 Hence any Becker-contraacyclic complex in $S\Modl_{J\ctra}$ is
Becker-semicontraacyclic.
 A complex $Y^\bu$ in $S\Modl_{J\ctra}$ is acyclic if and only if
the complex $\Theta(Y^\bu)$ is acyclic in $R\Modl_{I\ctra}$,
as explained in the first half of this section.
 The abelian category $R\Modl_{I\ctra}$ has enough projective objects
(see Section~\ref{prelims-on-wpr-secn}), so all Becker-contraacyclic
complexes in $R\Modl_{I\ctra}$ are acyclic by
Lemma~\ref{becker-co-contra-acyclic-are-acyclic}(b).
 It follows that any Becker-semicontraacyclic complex in
$S\Modl_{J\ctra}$ is acyclic.
 Thus we have
$$
 \Ac^\bctr(S\Modl_{J\ctra})\subset\Ac^\bsictr_{(R,I)}(S\Modl_{J\ctra})
 \subset\Ac(S\Modl_{J\ctra}).
$$

 Assume that the ideals $I\subset R$ and $J\subset S$ are finitely
generated.
 Then a complex of quotseparated $J$\+contramodule $S$\+modules
$Y^\bu$ is said to be \emph{Becker-semicontraacyclic}
(\emph{relative to $(R,I)$}) if it is Becker-contraacyclic as
a complex of quotseparated $I$\+contramodule $R$\+modules, i.~e., if
$\Theta(Y^\bu)\in\Ac^\bctr(R\Modl_{I\ctra}^\qs)$.
 We denote the thick subcategory of Becker-semicontraacyclic
complexes by $\Ac^\bsictr_{(R,I)}(S\Modl_{J\ctra}^\qs)\subset
\sK(S\Modl_{J\ctra}^\qs)$.
 The \emph{Becker semicontraderived category of quotseparated
$J$\+con\-tra\-mod\-ule $S$\+modules} (\emph{relative to $(R,I)$})
is defined as the triangulated Verdier quotient category
{\hbadness=2650
$$
 \sD^\bsictr_{(R,I)}(S\Modl_{J\ctra}^\qs)=
 \sK(S\Modl_{J\ctra}^\qs)/\Ac^\bsictr_{(R,I)}(S\Modl_{J\ctra}^\qs).
$$

 Similarly} to the arguments above, one proves the inclusions of
thick subcategories in the homotopy category
$$
 \Ac^\bctr(S\Modl_{J\ctra}^\qs)\subset
 \Ac^\bsictr_{(R,I)}(S\Modl_{J\ctra}^\qs)
 \subset\Ac(S\Modl_{J\ctra}^\qs).
$$
 Lemma~\ref{forgetful-preserve-becker-co-contra-acyclicity}(c) is
relevant here.

\Section{Relative Dualizing Complexes
and~Semico-Semicontra~Correspondence}  \label{semico-semicontra-secn}

 Let $I$ be a weakly proregular finitely generated ideal in
a commutative ring $R$ such that the ring $R$ is $I$\+adically coherent.
 Let $D^\bu$ be a dualizing complex of $I$\+torsion $R$\+modules,
as defined in Section~\ref{dualizing-complexes-secn}
(cf.\ Theorem~\ref{definitions-of-dualizing-complex-comparison}).

 We will need to make the assumption that the injective dimensions of
fp\+injective $I$\+torsion $R$\+modules (as objects of the abelian
category $R\Modl_{I\tors}$) are finite.
 For some results, we will also need the assumption that the projective
dimensions of contraflat $I$\+contramodule $R$\+modules (as objects
of the abelian category $R\Modl_{I\ctra}$) are finite.
 See the discussion in Section~\ref{dualizing-complexes-secn}.

 Let $J$ be a weakly proregular finitely generated ideal in
a commutative ring $S$, and let $f\:R\rarrow S$ be a ring homomorphism
such that $f(I)\subset J$.
 Assume that $S$ is a flat $R$\+module.

 As in Section~\ref{base-change-secn}, we consider the dual Koszul
complex $K^\bu_\infty(S,\bs)$ for some finite sequence of
generators~$\bs$ of the ideal $J\subset S$.
 Let $U^\bu$ be a finite complex of $J$\+torsion $S$\+modules isomorphic
to $K^\bu_\infty(S,\bs)\ot_RD^\bu$ in $\sD^\bb(S\Modl)$.
 We will say that $U^\bu$ is a \emph{relative dualizing complex}
for the morphism of ring-ideal pairs $f\:(R,I)\rarrow(S,J)$.

 The following theorem is the second main result of this paper.
 
\begin{thm} \label{becker-semico-semicontra-correspondence}
 Let $U^\bu$ be a relative dualizing complex for a morphism of
ring-ideal pairs $f\:(R,I)\rarrow(S,J)$ corresponding to a dualizing
complex of $I$\+torsion $R$\+modules $D^\bu$, as defined above.
 The assumptions above are enforced; so the ideals $I\subset R$
and $J\subset S$ are finitely generated and weakly proregular,
the ring $R$ is $I$\+adically coherent, and $S$ is a flat $R$\+module.
 Assume further that the morphism of pairs $(R,I)\rarrow(S,J)$ is
quotflat in the sense of Section~\ref{quotflat-secn}, and that
the injective dimensions of fp\+injective $I$\+torsion $R$\+modules
(as objects of $R\Modl_{I\tors}$) are finite.
 Then there is a triangulated equivalence between the Becker
semicoderived and semicontraderived categories (defined in
Section~\ref{semiderived-secn})
$$
 \sD^\bsico_{(R,I)}(S\Modl_{J\tors})\simeq
 \sD^\bsictr_{(R,I)}(S\Modl_{J\ctra})
$$
provided by (appropriately defined) mutually inverse derived functors\/
$\boR\Hom_S(U^\bu,{-})$ and $U^\bu\ot_S^\boL{-}$.
\end{thm}

 The proof of Theorem~\ref{becker-semico-semicontra-correspondence}
is based on two propositions.
 Let us start with introducing notation.
 Denote by $S\Modl_{J\tors}^{(R,I)\dfpinj}\subset S\Modl_{J\tors}$
the full subcategory of $J$\+torsion $S$\+modules that are
\emph{fp\+injective as $I$\+torsion $R$\+modules}.
 Similarly, denote by $S\Modl_{J\ctra}^{(R,I)\dctrfl}\subset
S\Modl_{J\ctra}$ the full subcategory of $J$\+contramodule $S$\+modules
that are \emph{contraflat as $I$\+contramodule $R$\+modules}.

\begin{prop} \label{fp-injective-over-base-semicoderived-prop}
 Let $I$ be a finitely generated ideal in a commutative ring $R$
such that the ring $R$ is $I$\+adically coherent, let $J$ be a finitely
generated ideal in a commutative ring $S$, and let $f\:R\rarrow S$
be a ring homomorphism such that $f(I)\subset J$.
 Assume that the morphism of pairs $f\:(R,I)\rarrow(S,J)$ is quotflat.
 Then the inclusion of exact/abelian categories
$S\Modl_{J\tors}^{(R,I)\dfpinj}\rarrow S\Modl_{J\tors}$ induces
a triangulated equivalence between the conventional derived
category and the Becker semicoderived category
$$
 \sD(S\Modl_{J\tors}^{(R,I)\dfpinj})\simeq
 \sD^\bsico_{(R,I)}(S\Modl_{J\tors}).
$$
\end{prop}

\begin{proof}
 First we observe that, for any complex of $J$\+torsion $S$\+modules
$M^\bu$, there exists a complex $H^\bu$ in
$S\Modl_{J\tors}^{(R,I)\dfpinj}$ together with a morphism of
complexes of $S$\+modules $M^\bu\rarrow H^\bu$ with a cone belonging
to $\Ac^\bsico_{(R,I)}(S\Modl_{J\tors})$.
 Indeed, the result of
Theorem~\ref{becker-co-contra-derived-of-loc-pres-abelian}(a) for
the abelian category $\sA=S\Modl_{J\tors}$ essentially says that
for any complex $M^\bu$ in $S\Modl_{J\tors}$ there exists a complex
of injective objects $H^\bu$ in $S\Modl_{J\tors}$ together with
a morphism of complexes of $S$\+modules $M^\bu\rarrow H^\bu$ whose
cone is Becker-coacyclic in $S\Modl_{J\tors}$.
 It remains to point out that all injective $J$\+torsion $S$\+modules
are fp\+injective as $I$\+torsion $R$\+modules by
Proposition~\ref{restriction-of-scalars-fp-injective}, and all
Becker-coacyclic complex of $J$\+torsion $S$\+modules are
Becker-coacyclic as complexes of $I$\+torsion $R$\+modules
by Lemma~\ref{forgetful-preserve-becker-co-contra-acyclicity}(a).

 Now the well-known result of~\cite[Proposition~10.2.7(i)]{KS}
or~\cite[Lemma~A.3.3(b)]{Pcosh} is applicable, and it remains to show
that a complex in $S\Modl_{J\tors}^{(R,I)\dfpinj}$ is acyclic
in $S\Modl_{J\tors}^{(R,I)\dfpinj}$ if and only if it belongs to
$\Ac^\bsico_{(R,I)}(S\Modl_{J\tors})$.
 Notice that a complex in $S\Modl_{J\tors}^{(R,I)\dfpinj}$ is
acyclic in $S\Modl_{J\tors}^{(R,I)\dfpinj}$ if and only if its
underlying complex of $R$\+modules is acyclic in
$R\Modl_{I\tors}^\fpinj$ (see the first paragraph of the proof
of Lemma~\ref{relative-pseudo-derived-conservative}(a)).
 So it remains to point that a complex in $R\Modl_{I\tors}^\fpinj$
is acyclic in $R\Modl_{I\tors}^\fpinj$ if and only if it is
Becker-coacyclic in $R\Modl_{I\tors}$, by
Proposition~\ref{torsion-modules-coderived-categories-prop}(b)
and its proof.
\end{proof}

\begin{prop} \label{contraflat-over-base-semicontraderived-prop}
 Let $I$ be a weakly proregular finitely generated ideal in
a commutative ring $R$ such that the ring $R$ is $I$\+adically coherent,
let $J$ be a weakly proregular finitely generated ideal in a commutative
ring $S$, and let $f\:R\rarrow S$ be a ring homomorphism such that
$f(I)\subset J$.
 Assume that the morphism of pairs $f\:(R,I)\rarrow(S,J)$ is quotflat.
 Then the inclusion of exact/abelian categories
$S\Modl_{J\ctra}^{(R,I)\dctrfl}\rarrow S\Modl_{J\ctra}$ induces
a triangulated equivalence between the conventional derived
category and the Becker semicontraderived category
$$
 \sD(S\Modl_{J\ctra}^{(R,I)\dctrfl})\simeq
 \sD^\bsictr_{(R,I)}(S\Modl_{J\ctra}).
$$
\end{prop}

\begin{proof}
 This proposition is dual-analogous to the previous one, and the proof
is also dual-analogous.
 First we observe that, for any complex of $J$\+contramodule
$S$\+modules $P^\bu$, there exists a complex $F^\bu$ in
$S\Modl_{J\ctra}^{(R,I)\dctrfl}$ together with a morphism of complexes
of $S$\+modules $F^\bu\rarrow P^\bu$ with a cone belonging
to $\Ac^\bsictr_{(R,I)}(S\Modl_{J\ctra})$.
 Indeed, the result of
Theorem~\ref{becker-co-contra-derived-of-loc-pres-abelian}(b) for
the abelian category $\sB=S\Modl_{J\ctra}$ essentially says that
for any complex $P^\bu$ in $S\Modl_{J\ctra}$ there exists a complex
of projective objects $F^\bu$ in $S\Modl_{J\ctra}$ together with
a morphism of complexes of $S$\+modules $F^\bu\rarrow P^\bu$ whose
cone is Becker-contraacyclic in $S\Modl_{J\ctra}$.
 It remains to point out that all projective $J$\+contramodule
$S$\+modules are contraflat as $I$\+contramodule $R$\+modules by
Proposition~\ref{restriction-of-scalars-contraflat}, and all
Becker-contraacyclic complex of $J$\+contramodule $S$\+modules are
Becker-contraacyclic as complexes of $I$\+contramodule $R$\+modules by
Lemma~\ref{forgetful-preserve-becker-co-contra-acyclicity}(b) or~(c).

 Now the well-known result of~\cite[Proposition~10.2.7(ii)]{KS}
or~\cite[Lemma~A.3.3(a)]{Pcosh} is applicable, and it remains to show
that a complex in $S\Modl_{J\ctra}^{(R,I)\dctrfl}$ is acyclic
in $S\Modl_{J\ctra}^{(R,I)\dctrfl}$ if and only if it belongs to
$\Ac^\bsictr_{(R,I)}(S\Modl_{J\ctra})$.
 Notice that a complex in $S\Modl_{J\ctra}^{(R,I)\dctrfl}$ is
acyclic in $S\Modl_{J\ctra}^{(R,I)\dctrfl}$ if and only if its
underlying complex of $R$\+modules is acyclic in
$R\Modl_{I\ctra}^\ctrfl$ (cf.\ the first paragraph of the proof
of Lemma~\ref{relative-pseudo-derived-conservative}(a)).
 So it remains to point that a complex in $R\Modl_{I\ctra}^\ctrfl$
is acyclic in $R\Modl_{I\ctra}^\ctrfl$ if and only if it is
Becker-contraacyclic in $R\Modl_{I\ctra}$, by
Proposition~\ref{contramodules-contraderived-categories-prop}(b)
and its proof.

 Let us mention that the assumption of weak proregularity of the ideal
$I\subset R$ is needed in the argument above because it is used in
the proof of Proposition~\ref{restriction-of-scalars-contraflat}.
 Besides, Proposition~\ref{restriction-of-scalars-contraflat} is
only applicable to quotseparated $J$\+contramodule $S$\+modules.
 The assumption of weak proregularity of the ideal $J\subset S$
is only used in the argument above in order to claim that all
$J$\+contramodule $S$\+modules are quotseparated,
$S\Modl_{J\ctra}^\qs=S\Modl_{J\ctra}$.
 Without the weak proregularity assumption on the finitely generated
ideal $J\subset S$, the assertion of the present proposition holds in
the context of quotseparated $J$\+contramodule $S$\+modules, i.~e.,
for the abelian category $S\Modl_{J\ctra}^\qs$.
\end{proof}

\begin{proof}[Proof of
Theorem~\ref{becker-semico-semicontra-correspondence}]
 The desired triangulated equivalence is constructed as the composition
of triangulated equivalences
\begin{multline*}
 \sD^\bsico_{(R,I)}(S\Modl_{J\tors})\simeq
 \sD(S\Modl_{J\tors}^{(R,I)\dfpinj}) \\
 \simeq\sD(S\Modl_{J\ctra}^{(R,I)\dctrfl})
 \simeq\sD^\bsictr_{(R,I)}(S\Modl_{J\ctra}).
\end{multline*}
 Here the first and the third triangulated equivalences are provided by
Propositions~\ref{fp-injective-over-base-semicoderived-prop}
and~\ref{contraflat-over-base-semicontraderived-prop}, respectively.
 The middle triangulated equivalence is obtained as a particular case
of Corollary~\ref{base-change-abstract-classes-derived-equivalence}.
 Let us spell out the details.

 In the context of Section~\ref{base-change-secn}, we put $L^\bu=D^\bu$,
and consider the full subcategories $\sE=R\Modl_{I\tors}^\fpinj
\subset R\Modl_{I\tors}$ and $\sF=R\Modl_{I\ctra}^\ctrfl\subset
R\Modl_{I\ctra}$.
 By Proposition~\ref{dualizing-complex-minimal-corresponding-classes},
the pair of classes $\sE$ and $\sF$ satisfies
conditions (I\+-IV) from Section~\ref{abstract-classes-secn}
for the dualizing complex of $R$\+torsion $I$\+modules $L^\bu=D^\bu$
(that is where the assumption of finite injective dimension of
fp\+injective $I$\+torsion $R$\+modules is used).
 Then, in the notation of Section~\ref{base-change-secn}, we have
$\sG_\sE=S\Modl_{J\tors}^{(R,I)\dfpinj}$ and
$\sH_\sF=S\Modl_{J\ctra}^{(R,I)\dctrfl}$, and
Corollary~\ref{base-change-abstract-classes-derived-equivalence}
(for $\st=\varnothing$) is applicable.
\end{proof}

 Our final theorem is the version of
Theorem~\ref{becker-semico-semicontra-correspondence} for
the Positselski semiderived categories instead of the Becker ones.

\begin{thm} \label{positselski-semico-semicontra-correspondence}
 Let $U^\bu$ be a relative dualizing complex for a morphism of
ring-ideal pairs $f\:(R,I)\rarrow(S,J)$ corresponding to a dualizing
complex of $I$\+torsion $R$\+modules~$D^\bu$.
 The assumptions from the beginning of this section are enforced; so
the ideals $I\subset R$ and $J\subset S$ are finitely generated and
weakly proregular, the ring $R$ is $I$\+adically coherent, and $S$
is a flat $R$\+module.
 Assume that the morphism of pairs $(R,I)\rarrow(S,J)$ is
quotflat in the sense of Section~\ref{quotflat-secn}.
 Assume further that the injective dimensions of fp\+injective
$I$\+torsion $R$\+modules (as objects of $R\Modl_{I\tors}$) are finite,
and the projective dimensions of contraflat $I$\+contramodule
$R$\+modules (as objects of $R\Modl_{I\ctra}$) are finite.
 Then there is a triangulated equivalence between the Positselski
semicoderived and semicontraderived categories (defined in
Section~\ref{semiderived-secn})
$$
 \sD^\sico_{(R,I)}(S\Modl_{J\tors})\simeq
 \sD^\sictr_{(R,I)}(S\Modl_{J\ctra})
$$
provided by (appropriately defined) mutually inverse derived functors\/
$\boR\Hom_S(U^\bu,{-})$ and $U^\bu\ot_S^\boL{-}$.
\end{thm}

\begin{proof}
 Under the finite injective/projective dimension assumptions of
the present theorem, the Positselski-coacyclicity property of complexes
in $R\Modl_{I\tors}$ agrees with the Becker-coacyclicity by
Proposition~\ref{torsion-modules-coderived-categories-prop}(c), and
the Positselski-contraacyclicity property of complexes in
$R\Modl_{I\ctra}$ agrees with the Becker-contraacyclicity by
Proposition~\ref{contramodules-contraderived-categories-prop}(c).
 Hence we have $\sD^\sico_{(R,I)}(S\Modl_{J\tors})=
\sD^\bsico_{(R,I)}(S\Modl_{J\tors})$ and
$\sD^\sictr_{(R,I)}(S\Modl_{J\ctra})=
\sD^\bsictr_{(R,I)}(S\Modl_{J\ctra})$, and the present theorem is
a particular case of
Theorem~\ref{becker-semico-semicontra-correspondence}.

 Alternatively, here is a direct proof based on the results
of Section~\ref{pseudo-derived-secn}.
 In the notation of the proof of
Theorem~\ref{becker-semico-semicontra-correspondence}, the exact
categories $\sE=R\Modl_{I\tors}^\fpinj\subset R\Modl_{I\tors}$
and $\sF=R\Modl_{I\ctra}^\ctrfl\subset R\Modl_{I\ctra}$ have finite
homological dimensions in our present assumptions.
 In the context of Section~\ref{pseudo-derived-secn}, put
$\sA=R\Modl_{I\tors}$ and $\sB=R\Modl_{I\ctra}$.
 Then we have $\Ac^{\sE\psco}(\sA)=\Ac^\co(\sA)$ and
$\Ac^{\sF\psctr}(\sB)=\Ac^\ctr(\sB)$ by
Lemma~\ref{pseudo-co-contra-acyclic-for-finite-homol-dim}.

 Put $\sX=S\Modl_{J\tors}$ and $\sY=S\Modl_{J\ctra}$, and denote
by $\Theta$ both the forgetful functors $\sX\rarrow\sA$ and
$\sY\rarrow\sB$ (as in Section~\ref{semiderived-secn}).
 Then Lemma~\ref{relative-pseudo-derived-conservative} tells us that
$$
 \Ac^{\sG_\sE\psco}(\sX)=\Ac^\sico_{(R,I)}(S\Modl_{J\tors})
 \quad\text{and}\quad
 \Ac^{\sH_\sF\psctr}(\sY)=\Ac^\sictr_{(R,I)}(S\Modl_{J\ctra}).
$$
 Following the discussion in Section~\ref{pseudo-derived-secn},
we have triangulated equivalences
$$
 \sK(\sX)/\Ac^{\sG_\sE\psco}(\sX)\simeq\sD(\sG_\sE)
 \quad\text{and}\quad
 \sK(\sY)/\Ac^{\sH_\sF\psctr}(\sY)\simeq\sD(\sH_\sF).
$$
 It remains to refer to
Corollary~\ref{base-change-abstract-classes-derived-equivalence}
(for $\st=\varnothing$) for the triangulated equivalence
$\sD(\sG_\sE)\simeq\sD(\sH_\sF)$.
\end{proof}

 Finally, let us reiterate the comments from the discussion at
the end of Section~\ref{base-change-secn}.
 Firstly, according to Remark~\ref{alternative-assumptions-remark}(2),
the quotflatness assumption can be dropped in
Theorems~\ref{becker-semico-semicontra-correspondence}
and~\ref{positselski-semico-semicontra-correspondence} if one assumes
the ring $S$ to be Noetherian.
 In Propositions~\ref{fp-injective-over-base-semicoderived-prop}
and~\ref{contraflat-over-base-semicontraderived-prop}, the quotflatness
assumption can be replaced by the assumptions that the $R$\+module $S$
is flat and the ring $S$ is Noetherian.

 Secondly, we reiterate the \emph{warning} that, unlike in the context
of~\cite[diagram~(15)]{Pps}, the square diagram
\begin{equation} \label{becker-semico-semicontra-co-contra-diagram}
\begin{gathered}
 \xymatrix{
  \sD^\bsico_{(R,I)}(S\Modl_{J\tors}) \ar@{=}[r] \ar[d]
  & \sD^\bsictr_{(R,I)}(S\Modl_{J\ctra}) \ar[d] \\
  \sD^\bco(R\Modl_{I\tors}) \ar@{=}[r] & \sD^\bctr(R\Modl_{I\ctra})
 }
\end{gathered}
\end{equation}
formed by the triangulated equivalences from
Theorem~\ref{becker-semico-semicontra-correspondence}
and Corollary~\ref{dualizing-complex-becker-co-contra-derived-equiv}
together with the obvious triangulated forgetful functors
is usually \emph{not commutative}.
 The same applies to the square diagram
\begin{equation} \label{positselski-semico-semicontra-co-contra-diagram}
\begin{gathered}
 \xymatrix{
  \sD^\sico_{(R,I)}(S\Modl_{J\tors}) \ar@{=}[r] \ar[d]
  & \sD^\sictr_{(R,I)}(S\Modl_{J\ctra}) \ar[d] \\
  \sD^\co(R\Modl_{I\tors}) \ar@{=}[r] & \sD^\ctr(R\Modl_{I\ctra})
 }
\end{gathered}
\end{equation}
formed by the triangulated equivalences from
Theorem~\ref{positselski-semico-semicontra-correspondence} and
Corollary~\ref{dualizing-complex-all-derived-equivs} together
with the triangulated forgetful functors, which is usually
\emph{not commutative}, either.
 In fact, the diagrams of triangulated
functors~\eqref{becker-semico-semicontra-co-contra-diagram}
and~\eqref{positselski-semico-semicontra-co-contra-diagram}
are commutative when $J=SI$, but \emph{not} in the general case.
 This is clear from the proof of
Proposition~\ref{base-change-abstract-classes}.

\end{document}